\documentclass[10pt,amsfonts, epsfig]{amsart}
\usepackage{amsmath, amscd, amssymb}
\usepackage{graphpap, color}
\usepackage{mathrsfs}
\usepackage{pstricks}
\usepackage{color}
\usepackage{cancel}
\usepackage[mathscr]{eucal}
\usepackage{pstricks}
\usepackage{color}
\usepackage{cancel}
\usepackage{verbatim}
\usepackage{latexsym}
\usepackage[all]{xy}

\def\sing{\mathrm{sing}}

\usepackage{lipsum}% http://ctan.org/pkg/lipsum
\makeatletter
\g@addto@macro{\endabstract}{\@setabstract}
\newcommand{\authorfootnotes}{\renewcommand\thefootnote{\@fnsymbol\c@footnote}}%
\makeatother

\usepackage{upgreek}

 \usepackage{wasysym}
\usepackage{ esint }

\def\index{{\mathbb I}}

\def\vdim{\mathrm{vir}.\dim}

\def\virt{^{\vir}}
\def\virtloc{\virt\loc}

\numberwithin{equation}{section}

\def\gm{\GG_m}

\def\loc{_1^\circ}

\def\sO{{\mathscr O}}
\def\sC{{\mathscr C}}
\def\sB{{\mathscr B}}
\def\sM{{\mathscr M}}
\def\sN{{\mathscr N}}
\def\sL{{\mathscr L}}

\def\sO{\mathscr{O}}

\def\sE{\mathscr{E}}

\def\sV{\mathscr{V}}
\def\sA{\mathscr{A}}
\def\sU{\mathscr{U}}

\def\sX{\mathscr{X}}

\newcommand{\CC}{\mathbb{C}}

\newcommand{\PP}{\mathbb{P}}
\newcommand{\QQ}{\mathbb{Q}}

\newcommand{\ZZ}{\mathbb{Z}}
\newcommand{\GG}{\mathbb{G}}

\newcommand{\bt}{\mathbf{t}}
\newcommand{\fw}{{\mathfrak w}}

\def\redd{{\mathrm{red}}}

\def\Gm{{\bG_m}}
 
\def\upmo{^{-1}}

\newcommand{\ev}{ \mathrm{ev} }
\newcommand{\cal}{\mathcal}

\def\cB{{\cal B}}

\def\cC{{\cal C}}
\def\cD{{\cal D}}

\def\cH{{\cal H}}

\def\cL{{\cal L}}
\def\cM{{\cal M}}
\def\cN{{\cal N}}
\def\cO{{\cal O}}
\def\cP{{\cal P}}

\def\cW{{\cal W}}

\def\ft{\mathfrak{t}}

\def\v1{{\vec{1}}}

\def\sP{{\mathscr P}}

\newcommand{\Mbar}{\overline{\cM}}

\def\mapright#1{\,\smash{\mathop{\lra}\limits^{#1}}\,}

\def\dual{^{\vee}}

\def\sta{^\ast}

\def\virt{^{\mathrm{vir}}}
\def\upmo{^{-1}}
\def\sta{^{\ast}}

\def\sta{^*}

\def\sB{{\mathscr B}}

\def\lra{\longrightarrow}

\def\lsta{_{\ast}}

\newcommand{\Si}{\Sigma}

\newcommand{\si}{\sigma}
\def\be{{\mathbf e}}

\def\wti{\widetilde}
\def\begeq{\begin{equation}}
\def\endeq{\end{equation}}
\def\and{\quad{\rm and}\quad}
\def\bl{\bigl(}
\def\br{\bigr)}
\def\defeq{:=}

\def\sub{\subset}
\def\Ao{{\mathbb A}^{\!1}}

\def\Po{{\mathbb P^1}}
\def\and{\quad\text{and}\quad}

\def\lalp{_\alpha}

\DeclareMathOperator{\pr}{pr}

 \DeclareMathOperator{\Aut}{Aut}

\DeclareMathOperator{\spec}{Spec}

\def\lggd{_{g,\gamma,\bd}}
\def\cWg{\cW\lggd}

\def\cWgg{\cW\lggd}
\newcommand{\pre}{ {\mathrm{pre}} }

\newtheorem{prop}{Proposition}[section]
\newtheorem{proposition}[prop]{Proposition}
\newtheorem{theo}[prop]{Theorem}
\newtheorem{lemm}[prop]{Lemma}
\newtheorem{coro}[prop]{Corollary}

\newtheorem{exam}[prop]{Example}
\newtheorem{defi}[prop]{Definition}
\newtheorem{definition}[prop]{Definition}

\newtheorem{subl}[prop]{Sublemma}

\newtheorem{defi-prop}[prop]{Definition-Proposition}
\newtheorem{defi-theo}[prop]{Definition-Theorem}
\def\Ob{\cO b}

\def\loc{_{\mathrm{loc}}}

\def\bul{^\bullet}

\def\ev{\text{ev}}

\def\sta{^\ast}

\def\sO{{\mathscr O}}

\def\sD{{\mathscr D}}

\def\beq{\begin{equation}}
\def\eeq{\end{equation}}

\def\vsp{\vskip5pt}
\def\Pf{{\PP^4}}

\def\bee{\begin{equation}}
\def\eeq{\end{equation}}

\def\sC{{\mathscr C}}

\def\bd{{\mathbf d}}

\def\ti{\tilde}

\def\Lam{{\Lambda}}

\def\barM{{\overline{M}}}

\def\mapright#1{\,\smash{\mathop{\lra}\limits^{#1}}\,}

\def\mufive{{\boldsymbol\mu_5}}

\def\mof{^{\vee\otimes 5}}

\def\Gm{T}

\begin{document}

\title[Mixed-Spin-P fields of Fermat quintic polynomials]{
Mixed-Spin-P fields of Fermat quintic polynomials}
%Theory of Gromov-Witten Invariants of Quintic Calabi-Yau Threefolds\\

\author[Huai-Liang Chang]{Huai-Liang Chang$^1$}
\address{Department of Mathematics, Hong Kong University of Science and Technology, Hong Kong} \email{mahlchang@ust.hk}
\thanks{${}^1$Partially supported by  Hong Kong GRF grant 600711}

\author[Jun Li]{Jun Li$^2$}
\address{Shanghai Center for Mathematical Sciences, Fudan University, China; \hfil\newline 
\indent Department of Mathematics, Stanford University,
USA} \email{jli@math.stanford.edu}
\thanks{${}^2$Partially supported by  NSF grant DMS-1104553 and DMS-1159156.}

\author[Wei-Ping Li]{Wei-Ping Li$^3$}
\address{Department of Mathematics, Hong Kong University of Science and Technology, Hong Kong} \email{mawpli@ust.hk}
\thanks{${}^3$Partially supported by by Hong Kong GRF grant 602512 and HKUST grant FSGRF12SC10}

\author[Chiu-Chu Melissa Liu]{Chiu-Chu Melissa Liu$^4$}
\address{Mathematics Department, Columbia University}
 \email{ccliu@math.columbia.edu}
\thanks{${}^4$Partially supported by  NSF grant DMS-1206667 and DMS-1159416.
}

\maketitle

\begin{abstract}  This is the first part of the project toward an effective algorithm to evaluate all genus Gromov-Witten invariants
of quintic Calabi-Yau threefolds. In this paper, we introduce the notion of Mixed-Spin-P fields, construct their moduli spaces, 
and construct the virtual cycles of these moduli spaces.  
\end{abstract}

\section{Introduction}
Explicitly solving all genus Gromov-Witten invariants (in short GW invariants) of Calabi-Yau threefolds is one of the major goals in the subject of Mirror Symmetry. For quintic Calabi-Yau threefolds, the mirror formula of genus-zero GW invariants was conjectured in \cite{CdGP} and proved in \cite{Gi, LLY}. The mirror formula of genus-one GW invariants was
conjectured in \cite{BCOV} and proved in \cite{LZ,Zi}. A complete
determination of all genus GW invariants based on degeneration
is provided in \cite{MP} and plays a crucial role in the 
proof of the GW/Pairs correspondence \cite{PP}. 
However, the mirror prediction on genus $g$ GW invariants for $2\leq g\leq 51$ in \cite{HKQ} is still open, even in 
the $g=2$ case.

The mirror prediction in \cite{HKQ} includes both 
GW invariants of quintic threefolds and 
FJRW invariants of the Fermat quintic.
In this paper, we introduce the notion of Mixed-Spin-P fields (in short MSP fields) of the Fermat quintic polynomial, construct their moduli spaces, and establish basic properties of these moduli spaces. This class of moduli spaces will be employed in the sequel of this paper \cite{CLLL} toward developing an effective theory evaluating 
all genus GW invariants of quintic threefolds 
and all genus FJRW invariants of the Fermat quintic.

The theory of MSP fields provides a transition between 
FJRW invariants \cite{FJR1} and GW invariants of stable maps with p-fields \cite{CL}. 
It is known that  the FJRW invariants of the Fermat quintic is the LG theory taking values
in $[\CC^5/\mufive]$ (via spin fields), and the GW invariants of stable maps with p-fields is the 
LG theory taking values in the canonical line bundle $K_{\Pf}$ (via P-fields). 
As one can use the master space technique to study the two GIT quotients $[\CC^5/\mufive]$ and $K_{\Pf}$ 
of $[\CC^6/\gm]$, 
MSP fields is a field theory taking values in this master space, which provides a
geometric transition between the LG theories of the two GIT quotients of $[\CC^6/\gm]$.

%Our theory is based on a vanishing of the virtual cycles of the moduli of MSP fields which provides relations
%among GW invariants of quintics and the FJRW invariants of the Fermat quintic polynomial. % $x_1^5+\cdots+x_5^5$.

An MSP field is a collection
\beq\label{MSP0}
\xi=( {\Si^\sC} \subset \sC, \sL, \sN,\varphi,\rho, \nu),
\eeq
consisting of a pointed twisted curve $\Sigma^\sC\sub\sC$, two fields $\varphi\in H^0(\sC,\sL^{\oplus 5})$ 
and $\rho\in H^0(\sC, \sL^{\vee\otimes5}\otimes \omega^{\log}_{\sC})$, and a %projectivized spin 
gauge field
$\nu=(\nu_1,\nu_2)\in H^0(\sC, \sL\otimes\sN\oplus \sN)$. % satisfying some non-degeneracy condition.
The numerical invariants of an MSP field are the genus of $\sC$, the
monodromy $\gamma_i$ of $\sL$ at the marking $\Si^\sC_i$ (of $\Si^\sC$), 
and the bi-degrees $d_0=\deg(\sL\otimes\sN)$ and $d_\infty=\deg \sN$.
For a choice of 
$g$, $\gamma=(\gamma_1,\cdots,\gamma_\ell)$ and $\bd=(d_0,d_\infty)$, we form the moduli $\cW_{g,\gamma,\bd}$ of equivalence classes of 
stable MSP fields of numerical data $(g,\gamma, \bd)$. It is a separated DM stack, locally of finite
type, though usually not proper.

The stack $\cW\lggd$ comes with a perfect (relative) obstruction theory, with 
%of virtual dimension (in case $\gamma=\emptyset$)
%\beq\label{vdim}
$$\delta(g,\emptyset,\bd)\defeq \vdim \cW_{g,\gamma=\emptyset,\bd}=d_0+d_\infty-g+1.
$$ %\eeq
Its relative obstruction sheaf comes with a cosection 
$$\sigma:\Ob_{\cW\lggd}\lra \sO_{\cW\lggd},
$$
using the quintic polynomial $\fw_5=x_1^5+\cdots+x_5^5$.
Furthermore, we have a $T=\gm$ action on $\cW\lggd$ making it
%\beq\label{Gm}
%\bl {\Si^\sC} \subset \sC, \sL, \sN,\varphi,\rho, (\nu_1,\nu_2)\br^t
%=\bl {\Si^\sC} \subset \sC, \sL, \sN,\varphi,\rho, (t\nu_1,\nu_2)\br,
%\quad t\in \Gm.
%\eeq
a $\Gm$-stack with the mentioned relative perfect obstruction
theory and the  $T$-equivariant cosection $\sigma$.
%One shortcoming of the stack $\cW\lggd$ is that it is usually not proper.

We define the degeneracy locus of $\sigma$ to be
$$\cW\lggd^-=\{\xi\in \cW\lggd\mid  \sigma|_\xi=0\}.
$$
%Despite the fact that $\cW\lggd$ usually is non-proper, the degeneracy locus  
%(cf. \eqref{deg-loci}) is proper. 
Applying the theory of cosection localized virtual cycles of \cite{KL},
we obtain a cosection localized $T$-equivariant virtual cycle 
of $\cW\lggd$.

%We construct a properly supported virtual cycle of $\cW\lggd$ by 
%constructing a cosection $\sigma$ of its obstruction sheaf
%and applying the theory of cosection localized virtual cycles of \cite{KL}. 
%It is a cycle supported in the degeneracy locus $\cW\lggd^-=(\sigma=0)_{\text{red}}$ (cf. \eqref{deg-loci}).
%where $\cW\lggd^-$ is a closed (reduced) substack characterized by
%\beq\label{MSP-}
%\cW\lggd^-(\CC)=\{\xi\mid (\varphi_1^5+\cdots+\varphi_5^5=\rho=0)\cup(\varphi=0)=\sC\}.
%\eeq
%
%We prove in Section 3 and 4 the following structure result:

\begin{theo} \label{main-1}
The stack $\cW\lggd$ is a separated DM stack, locally of finite type. 
The $\sigma$-degeneracy locus $\cW\lggd^-$ is  closed,  proper 
and of finite type. 
The cosection localized virtual cycle of $(\cW\lggd,\sigma)$ is a $T$-equivariant cycle
$$[\cW\lggd]\virtloc\in A^{\Gm}\lsta (\cW\lggd^-)^T.
$$ 
\end{theo}

In the sequel of this paper \cite{CLLL}, we will apply the virtual localization formula to derive a doubly indexed
polynomial relations among the GW invariants of quintic threefolds and the FJRW invariants of the
Ferman quintic polynomial $\fw_5$. These relations provide an effective algorithm in evaluating all genus GW invariants
of quintics in terms of all genus FJRW invariants of $\fw_5$ (with the insertion $2/5$), and provide ample relations 
among all genus FJRW invariants of $\fw_5$ (with the insertion $2/5$).

This work is inspired by Witten's vision that the ``Landau-Ginzburg looks like the analytic continuation of Calabi-Yau
to negative Kahler class." (See \cite[3.1]{GLSM}.) One interpretation of this analytic continuation is that the 
LG theory of $[\CC^5/\ZZ_5]$ and that of $K_{\Pf}$ differ by a fields version of ``wall-crossing". 
The MSP fields introduced
can be viewed as a geometric construction to realize this ``wall-crossing".

Here we mention the recent approaches for high genus LG/CY correspondence \cite{CK,FJR2}.

%The ultimate goal of this project is to use this field version of master space to uncover the mystery of high
%genus GW invariants of quintics.
%The first phase of this project consists of constructing the moduli space of MSP fields, constructing their "virtual cycles",
%and deriving a collection of vanishings that give effective algorithm determining GW from FJRW invariants of quintics. 
%This paper addresses the first part. The second part will be addressed in \cite{CLLL}.

This paper is organized as follows. In Section one, we will introduce the notion of Mixed-Spin-P fields of the
Fermat quintic polynomial; construct the moduli spaces of stable Mixed-Spin-P fields, and construct the cosection localized 
virtual cycles of these moduli spaces. These cycles lie in the degeneracy loci of the cosection mentioned.
In Section two and three, we will prove that these degeneration loci are proper, separated and of finite type. This allows us
to use these cycles to define numerical invariants of the Mixed-Spin-P fields.

In this paper, we work over the field of complex numbers $\CC$. All schemes are of finite type over $\CC$ unless
otherwise is mentioned.

\medskip

{\sl Acknowledgement}. The third author thanks the Stanford University for several months visit there in the spring of 2011 where the project started.  The second and the third author thank the Shanghai Center for Mathematical Sciences at Fudan University for many visits. 
The authors thank Y.B. Ruan for stimulating discussions on the FJRW invariants. 

\section{The moduli of Mixed-Spin-P fields}

In this section, we introduce the notion of MSP (Mixed-Spin-P) fields,  construct their moduli stacks,
and form their cosection localized virtual cycles. We introduce the $\Gm$-structure on it.
%and derive the vanishing result \eqref{van-1} using an analogue of virtual localization formula of \cite{GP}. 
The proof of the localization formula of cosection localized virtual cycles will appear in
\cite{CoVir}.

\def\fo{\frac{1}{5}}
\let\ofth=\fo
\def\lred{_{\mathrm{red}}}

\def\sm{{\mathrm{sm}}}
\def\bmu{{\boldsymbol \mu}}
\def\sch{{\mathrm{sch}}}

\subsection{Twisted curves and invertible sheaves}\label{Sub2.1}

We recall the basic notions and properties of twisted curves with representable invertible sheaves on them. 
The materials are drawn from \cite{ACV, AJ, AF, A-G-V, Cad}. 

A prestable twisted curve with $\ell$-markings is a one-dimensional proper, separated connected DM stack $\sC$, with at
most nodal singularities, together with a collection of disjoint closed substacks $\Sigma_1,\cdots,\Si_\ell$ of smooth locus
of $\sC$ such that the open locus $\sC^\sm-\cup_i \Si_i$ is a scheme, and each node is a balanced node.

Here an index $r$ balanced node looks like the following  model 
$$\sV_r\defeq \left[ \spec\bl \CC[u,v]/(uv)\br\big/ \bmu_r\right],\quad \zeta\cdot(u,v)=(\zeta u,\zeta\upmo v).
$$
%where $\bmu_r$ acts via $(u,v)^\zeta=(\zeta u,\zeta\upmo v)$.
Similarly, an index $r$ marking of a twisted curve looks like the model
$$\sU_r\defeq \left[ \spec\CC[u]\big/ \bmu_r\right],\quad \zeta\cdot u=\zeta u.
$$

Denote by
\beq\label{Apir}\pi_r: \sV_r\to V_r\defeq \spec\bl \CC[x,y]/(xy)\br \and
\pi_r: \sU_r\to U_r\defeq \spec \CC[x]
\eeq
defined by $x\mapsto u^r$ and $y\mapsto v^r$, their coarse moduli spaces. 
Note that $\sV_r$ contains two subtwisted curves $\sU_{r,u}=\left[ \spec\CC[u]\big/ \bmu_r\right]$
and $\sU_{r,v}=\left[ \spec\CC[v]\big/ \bmu_r\right]\sub \sV_r$, meeting at the node of $\sV_r$. The process of $\sV_r\mapsto \sU_{r,u}\coprod \sU_{r,v}$
is called the decomposation of  $\sV_r$ along its node. The reverse process is called the gluing, which can be defined via the push out as follows.

Let $0_u$ be the origin in $\spec\CC[u]$ and $N_u$ be the normal bundle of $0_u$ in it. Then the stacky point
of $\sU_{r,u}$ is $\sqrt[r]{N_u/0_u}$ (see \cite{A-G-V} for the notation). Similarly, the stacky point of $\sU_{r,v}$ is $\sqrt[r]{N_v/0_v}$.
Let $\sqrt[r]{N_u/0_u}\cong \sqrt[r]{N_v/0_v}$ be induced by $N_u\otimes N_v\cong \CC$, where the later is
the isomorphism induced by the projection $\spec\CC[u,v]\to\spec \CC[t]$ via $t\mapsto uv$. Then
$\sV_r$ is the push out via the square
$$\begin{CD}
\sqrt[r]{N_u/0_u}\cong \sqrt[r]{N_v/0_v} @>>> \sU_{r,v}\\
@VVV @VVV\\
\sU_{r,u} @>>> \sV_r
\end{CD}
$$
We denote by $[0]$ the stacky node of $\sV_r$.

\vsp
Invertible sheaves on a twisted curve $\sC$ are invertible sheaves on $\sC^\sm-\cup_i\Si_i^\sC$ whose extension to
$\sC$ in the model case $\sC=\sV_r$ is a $\bmu_r$-module $\sM_{m}$ as follows. 
When $0<m<r$, we have 
$$
\sM_{m}\defeq u^{-(r-m)}\CC[u]\oplus_{[0]} v^{-m}\CC[v]\defeq \ker\{u^{-(r-m)}\CC[u]\oplus v^{-m}\CC[v]\to\CC_{m}\},
$$
where  the arrow is a homomorphism of $\bmu_r$-modules, and
$\bmu_r$ leaves $1\in \CC[u]$ and $1\in \CC[v]$ fixed and acts on $1\in \CC_{m}\cong \CC$ via $\zeta\cdot1=\zeta^{m}1$,
and both $u^{-(r-m)}\CC[u]\to \CC_{m}$ and $v^{-m}\CC[v]\to\CC_{m}$ are surjective.
When $m=0$, $$
\sM_{0}\defeq \CC[u]\oplus_{[0]} \CC[v]\defeq \ker\{\CC[u]\oplus \CC[v]\to\CC\},
$$
 where maps are defined similarly as the case of $m\neq 0$. 
Note that the isomorphism classes are indexed by $m\in \{0,\cdots,r-1\}$. Here we use $u^{-m}$ and $v^{-(r-m)}$
because then (for $m\ne 0$)
\beq\label{Adirect}
\pi_{r\ast}\sM_{m}=\CC[x]\oplus \CC[y],
\eeq
in the convention \eqref{Apir}. To be more precise, if we let $1_u\in \CC[u]$ be the element $1\in\CC\sub\CC[u]$, and let
$1_v$, $1_x$ and $1_y$ be similarly defined elements, then $\pi_{r\ast}$ in \eqref{Adirect} sends $1_u$ and $1_v$
to $1_x$ and $1_y$, respectively.

Similarly, invertible sheaves on $\sC$ near an index $r$ marking in the model case $\sC=\sU_r$ looks like 
the $\bmu_r$-module
$$\sM_m\defeq u^{-(r-m)}\CC[u].
$$

Let $\zeta_r=\text{exp}(2\pi i/r)\in \bmu_r$. Because under $u\mapsto \zeta_r u$, the generator $u^{-(r-m)}1_u\mapsto \zeta_r^m u^{-(r-m)}1_u$,
we call $\zeta_r^m$ the monodromy of $\sM_m$ at the marking and $m$ the monodromy index at the marking. 
Note that for $\sM_m$ over $\sV_r$,
$\sM_m$ restricted to $\sU_{r,u}$ and to $\sU_{r,v}$ have monodromies $\zeta_r^m$
and $\zeta_r^{-m}$, at their respective stacky points.

%:
%:
\begin{defi} We call $\sM_m$ representable if $(m,r)=1$.%\footnote{There is a typo in the Definition 2.7 of \cite{CLL}. Instead of ``if the monodromy representations ... are injective", it should be ``if any projection to its factor $(\gm)^n\to\gm$ of the monodromy representations of marked sections and nodesis injective (representable)". It doesn't affect results in that paper since that paper used this correct definition. }
\end{defi}

\subsection{Definition of MSP fields}
As we will focus on the quintic case, we will  only  consider  the monodromy group $\bmu_5\le\gm$,
%Let $S$ be a scheme of finite type (over $\CC$). 
the subgroup of the 5-th roots of unity. Let
$$\ti\bmu_5=\{(1,\rho),(1,\varphi),\zeta_5,\cdots,\zeta_5^4\},\and \ti\bmu_5^+=\{(1,\rho),(1,\varphi),1,\zeta_5,\cdots,\zeta_5^4\}.
$$
%, \ti\bmu_5^{+}$ be $\{1_\varphi,1_\rho,\zeta_5,\cdots,\zeta_5^4\}$. 
We agree 
$\langle (1,\rho)\rangle=\langle (1,\varphi)\rangle=\langle 1\rangle=\{1\}\le \gm$ is the trivial subgroup of $\gm$.
Let
$$g\ge 0,\quad \gamma=(\gamma_1,\cdots,\gamma_\ell)\in (\ti\bmu_5)^{\times\ell},\and \bd=(d_0, d_\infty)\in
\QQ^{\times 2},
$$ 
 and call the triple 
$(g,\gamma,\bd)$ a numerical data (for MSP fields),
and call $(g,\gamma,\bd)$ a broad numerical data if $\gamma\in (\ti\bmu_5^+)^{\times\ell}$ instead.

For  an $\ell$-pointed twisted nodal curve $\Si^\sC\sub \sC$,  denote 
$\omega^{\log}_{\sC/S}\colon=\omega_{\sC/S}(\Si^\sC)$, and for 
$\alpha\in \ti\bmu_5^+$,  let $\Si^\sC_\alpha=\coprod_{\gamma_i=\alpha}\Si^\sC_i$.
%We denote by $\Si^\sC_\sch\sub\Si^\sC$ the scheme part of $\Si^\sC$.

%For $\sL$ an invertible sheaf on $\sC$,
%%Let $(\sC,\sL)$ be a pair of a twisted curve $\sC$ with an invertible sheaf $\sL$; 
%let $x\in\sC$ be a (smooth) stacky point,
%meaning $\Aut_\sC(x)\ne \{1\}$, then $\Aut_\sC(x)$ acts linearly and faithfully on $T_{x}\sC\cong \CC$, 
%thus $\Aut_\sC(x)\le\Gm$, canonically. In case it is a cyclic group of order $r$, we pick a standard generator $e_x$
%of $\Aut_\sC(x)$ by the rule that $e_x\in \Aut_\sC(x)$ corresponds to $\zeta_r=\exp(\frac{2\pi \sqrt{-1}}{r})\in \Gm$ 
%under the {injection} $\Aut_\sC(x)\le\Gm$.
%Let $\bm_x: \Aut_\sC(x)\to \Aut_\sL(x)\cong \Gm$ be the monodromy representation; we define
%the image $\bm_x(e_x)\in\Gm$ the monodromy of $\sL$ along $x$. %In case $x=\Si^\sC_i$, we denote
%It is canonically defined.
%In case $x$ is a scheme point, i.e. $Aut_\sC(x)=\{1\}$, the monodromy of $\sL$ along $x$ is $1\in \Gm$.

\begin{definition}\label{def-curve} Let $S$ be a scheme,  $(g,\gamma,\bd)$ be a numerical data.
An $S$-family of MSP-fields of type $(g,\gamma,\bd)$ is a datum
$$
(\sC,\Si^\sC,\sL,\sN, \varphi,\rho,\nu)
$$
such that
\begin{enumerate}
\item[(1)] $\cup_{i=1}^\ell\Si_i^\sC = \Si^\sC\subset \sC$ is an $\ell$-pointed, genus $g$, 
twisted curve over $S$ such that the $i$-th marking $\Si^\sC_i$ is banded by the group $\langle\gamma_i\rangle\le \gm$;
%thus is an object in the groupoid $\fM_{g,n}^\tw(S)$.
\item[(2)] $\sL$ and $\sN$ are representable invertible sheaves on $\sC$, and  $\sL\otimes \sN$ and $\sN$ have fiberwise degrees $d_0$ and $d_\infty$ respectively. The monodromy of $\sL$ along $\Si^\sC_i$ is
$\gamma_i$ when $\langle \gamma_i\rangle\ne\langle 1\rangle$;
%, and $\sN$ has fiberwise degree $d_\infty$.
\item[(3)] $\nu=(\nu_1, \nu_2)\in H^0( \sL\otimes\sN)\oplus  H^0( \sN)$ such that $(\nu_1,\nu_2)$ is nowhere vanishing;
\item[(4)] 
%an $S$-pre-MSP field $(\sC,\Si^\sC,\varphi,\rho, \nu)$ 
$\varphi=(\varphi_1,\ldots, \varphi_{5}) \in H^0(\sL)^{\oplus 5}$ so that $(\varphi,\nu_1)$ is nowhere zero,
and $\varphi|_{\Si^\sC_{(1,\varphi)}}=0$. %, and $\Si^\sC_{(1,\varphi)}\sub (\nu_1\ne0)$.
\item[(5)] $\rho \in H^0(\sL^{\vee 5}\otimes \omega^{\log}_{\sC/S})$ such that
$(\rho,\nu_2)$ is nowhere vanishing, and $\rho|_{\Si^\sC_{(1,\rho)}}=0$;
%\item[(6)] $\Si^\sC_{(1,\varphi)}\sub (\nu_1\ne0)$ and $\Si^\sC_{(1,\rho)}\sub (\nu_2\ne 0)$
\end{enumerate}
\end{definition}

%Note that (6) is implied by (4) and (5). 
In the future, we
call $\varphi$ (resp. $\rho$) the $\varphi$-field (resp. $\rho$-field) of the MSP-field.\footnote{We call 
$(\sC,\Si^\sC,\sL,\sN,\nu)$ satisfying (1)-(3),
and $\Si^\sC_{(1,\varphi)}\sub (\nu_1\ne0)$ and $\Si^\sC_{(1,\rho)}\sub (\nu_2\ne 0)$, a gauged twisted $S$-curve.}

%\begin{definition}\label{def-MSP}
%Given an $S$-family of mixed-curves $(\sC,\Si^\sC,\sL, \sN, \nu)$, a $\varphi$-field (resp. $\rho$-field) on it
%%an $S$-pre-MSP field $(\sC,\Si^\sC,\varphi,\rho, \nu)$ 
%is a section $\varphi=(\varphi_1,\ldots, \varphi_{5}) \in H^0(\sL)^{\oplus 5}$
%(resp. $\rho \in H^0(\sL^{\vee 5}\otimes \omega^{\log}_{\sC/S})$) such that
%%\begin{enumerate}
%%
%%\end{enumerate}
%An $S$-MSP field $(\sC,\Si^\sC,\sL, \sN,\varphi,\rho, \nu)$ of type $(g,\gamma,\bd)$ consists of an $S$-family of mixed-curves $(\sC,\Si^\sC,\sL, \sN, \nu)$
%of type $(g,\gamma,\bd)$ together with a $\varphi$-field $\varphi$ and a $\rho$-field $\rho$ on it.
%\end{definition}

\begin{definition}\label{def-broad}
In case $(g,\gamma,\bd)$ is a broad numerical data, a similarly defined $(\sC,\Si^\sC,\sL,\sN,\varphi,\rho, \nu)$ as
in Definition \ref{def-curve} is called an $S$-family of broad MSP-fields.
\end{definition}

We remark that in this paper we will only be concerned with MSP fields. 
%The notion of broad-MSP fields will be used in Section \ref{} to prove a vanishing result.

%$\varphi=(\varphi_1,\ldots, \varphi_{5}) \in H^0(\sL)^{\oplus 5}$;
%, and 
%$\nu=(\nu_1, \nu_2)\in H^0( \sL\otimes\sN)\oplus  H^0( \sN)$;
%\item $(\nu_1, \nu_2)$, $(\varphi,\nu_1)$ %$\in H^0(\sL^{\oplus 5} \oplus \sL\otimes \sN)$
%and $(\rho,\nu_2)$ %\in H^0(\sL^{\vee 5}\otimes \omega^{\log}_{\sC/S} \oplus\sN)$ 
%are nowhere vanishing;
%\item  then
%$\Si^\sC_{1_\varphi}\sub (\nu_2\ne 0)$ and $\rho|_{\Si^\sC_{1_\varphi}}=0$;
%and $\Si^\sC_{1_\rho}\sub (\nu_2=0)$
%%markinglet $\Si^\sC_{\sch,0}=(\nu_1=0)\cap \Si^\sC_{\sch}$ and $\Si^\sC_{\sch,\infty}=(\nu_2=0)\cap \Si^\sC_{\sch}$,
% and 
%%$x\in \sC$, the induced $\Aut_\sC(x)\to \Aut_{\sL}(x)$
%%is injective. (I.e., the induced $\sL: \sC \to B\Gm$ is representable).
%\end{enumerate}

%In case all $\gamma_i=1$, we will replace $\gamma$ by the numeric $\ell$.

\begin{definition}\label{df:Wmor}
An arrow
$$
( \sC' ,\Si^{\sC'} , \sL', \sN',\varphi',\rho', \nu')
\lra ( \sC,\Si^\sC, \sL, \sN,\varphi,\rho, \nu)
$$
from an $S'$-MSP-field to an $S$-MSP-field consists of a morphism $S'\to S$ and a 3-tuple $(a,b,c)$ such that
\begin{enumerate}
\item $a:(\Si^{\sC'} \subset \sC')\to  (\Si^\sC \subset \sC)\times_SS'$
is an $S'$-isomorphism of pointed twisted curves;
\item $b: a^*\cL\to \cL'$ and $c:a^*\sN\to \sN'$ are isomorphisms of invertible
sheaves such that the pullbacks of $\varphi_k$, $\rho$ and $\nu_i$ are identical to $\varphi_k'$, $\rho'$
and $\nu_i'$, where the pullbacks and the isomorphisms are induced by $a$, $b$ and $c$.
%(a). $b (a^*\varphi_i) =\varphi_i'$ for all $i$; (b). $(b^{\vee 5})(a^*\rho)=\rho'$, and 
%(c). $(b\otimes c)(a^*\nu_1)=\nu_1'$, and $c(a^*\nu_2)=\nu_2'$.
\end{enumerate}
\end{definition}

%Given $S'\to S$, we can pullback an $S$-family in Definition \ref{def-curve} to an $S'$-family
%via fiber product.
We define $\cW\lggd^{\pre}$ to be the category fibered in groupoids over
the category of schemes, such that the objects in $\cW\lggd^{\pre}$ over $S$ are
$S$-families of MSP-fields, and morphisms are given by Definition \ref{df:Wmor}.
%$$
%p:\cW\lggd^{\pre}\lra  (\textup{Schemes}), \quad
% \text{an $S$-family}\ ( \cC,\Si^\sC , \cL, \sN,\varphi,\rho, \nu)\longmapsto S.
%$$
%We have a contravariant functor from the category of schemes to the category of groupoids
%$$
%\cW\lggd^{\pre}: (\textup{Schemes})\lra (\textup{Groupoids}),\quad S\longmapsto p^{-1}(S)=\cW\lggd^{\pre}(S).
%$$

\begin{defi} $\xi\in \cW\lggd^{\pre}(\CC)$ is {\em stable} if $\Aut(\xi)$ is finite.
%\begin{enumerate}
%\item ;
%\item for any stacky point
%$x\in \sC$, the tautological $\Aut_\sC(x)\to \Aut_{\sL}(x)$
%and $\Aut_\sC(x) \to Aut_\sN(x)$ are injective (i.e., the morphisms $\sL$ and $\sN:
%\sC \to B\Gm$ are representable), and 
%\item there is no smooth subcurve $\sD\sub\sC$ so that $\deg\sN|_{\sD}=\frac{1}{5}$.
%\end{enumerate}
 $\xi\in \cW\lggd^{\pre}(S)$ is stable if $\xi|_s$ is stable for every closed point $s\in S$.
%$\xi|_s\in \cW\lggd^{\pre}(\bk(s))$ are stable for all geometric $s\in S$.
\end{defi}

Let $\cW\lggd\subset \cW^{\pre}\lggd$ be the open substack of families of stable objects in $\cW\lggd^{\pre}$.
We introduce a $T=\gm$ action on $\cW\lggd$ by 
\beq\label{Gm}
t\cdot (\Si^\sC, \sC, \sL, \sN,\varphi,\rho, (\nu_1, \nu_2))
=  (\Si^\sC, \sC, \sL, \sN,\varphi,\rho, (t\nu_1,\nu_2)),\quad t\in \Gm.
\eeq

\begin{theo}
%Given numerical data $(g,\gamma,\bd)$, 
The stack $\cW\lggd$  
is a DM $T$-stack, locally of finite type.
\end{theo}

\begin{proof}
The theorem follows immediately from that the stack $\cM^{\text{tw}}_{g,\ell}$ of 
stable twisted $\ell$-pointed curves is a DM stack, and each of its connected components
is proper and of finite type (see \cite{AJ, O}). %We skip the proof here.
\end{proof}

In this paper, we will reserve the symbol $T=\gm$ for this action on $\cW\lggd$.

\begin{exam} [Stable maps with $p$-fields]\label{pfield}
A stable MSP-field $\xi\in\cW\lggd$ having $\nu_1=0$ will have $\sN\cong\sL\dual$, $\nu_2=1$. 
Then $\xi=(\Si^\sC,\sC,\cdots)$ reduces to  a stable map $f=[\varphi]: \Si^\sC\sub \sC\to\Pf$
together with a $p$-field $\rho\in H^0(f\sta \sO_\Pf(-5)\otimes\omega_\sC^{\log})$. Moduli of genus $g$ $\ell$-pointed stable maps with 
$p$-fields will be denoted by $\barM_{g,\ell}(\Pf,d)^p$.
\end{exam}

\begin{exam} [$5$-spin curves with five $p$-fields]\label{spinp}
A stable MSP-field $\xi\in\cW\lggd$ having $\nu_2=0$ will have $\sN\cong\sO_\sC$, $\nu_1=1$.
Then $\xi$ reduces to a pair of a $5$-spin curve $(\Si^\sC,\sC, \rho: \sL^{\otimes 5}\cong \omega_\sC^{\log})$ 
and five $p$-fields $\varphi\in H^0(\sL)^{\oplus 5}$.
Moduli of $5$-spin curves with fixed monodromy $\gamma$ and  five $p$-fields will be denoted by $\barM_{g,\gamma}^{1/5,5p}$.
\end{exam}

\subsection{Cosection localized virtual cycle}\label{sub2.3}

The DM stack $\cWgg$ admits a tautological $\Gm$-equivariant perfect obstruction theory. 

Let $\cD_{g,\gamma}$ be the stack of triples $(\sC,\Si^{\sC},\sL,\sN)$, where $\Si^{\sC}\sub\sC$ are
$\ell$-pointed genus $g$ connected twisted curves (i.e. objects in $\cM^{\text{tw}}_{g,\ell}$), 
$\sL$ and $\sN$ are invertible sheaves on $\sC$ such that the monodromy of $\sL$ along the marked points
are given by $\gamma$.
Because $\cM^{\text{tw}}_{g,\ell}$ is a smooth DM stack, $\cD_{g,\gamma}$ is a smooth Artin stack, 
locally of finite type and of dimension $(3g-3)+\ell+2(g-1)=5g-5+\ell$, where the automorphisms of 
$(\Si^{\sC}\sub \sC,\sL,\sN)$ are triples $(a,b,c)$ as in Definition \ref{df:Wmor}.

Define
\beq\label{q-mor}
q: \cWgg\lra \cD_{g,\gamma}
\eeq
to be the forgetful morphism, forgetting $(\varphi,\rho,\nu)$ from points $\xi\in\cWgg$. The morphism
$q$ is $\Gm$-equivariant with $\Gm$ acting on $\cD_{g,\gamma}$ trivially.
Let 
\beq\label{universal}
\pi: \Si^\cC\sub \cC\to\cWgg\quad \text{with}\quad  (\cL,\cN,\varphi,\rho,\nu)
\eeq
being the universal family over $\cWgg$. %As we will see later, for $\cW\lggd\ne\emptyset$, every
Let $0\le m_i\le 4$ be so that $\gamma_i=\zeta_5^{m_i}$.
For convenience, we abbreviate
$$\cP=\cL^{\vee\otimes 5}\otimes\omega^{\log}_{\cC/\cWgg}.
$$

%where $\zeta_5=\exp(\frac{2\pi\sqrt{-1}}{5})$. %$e^{\frac{2\pi\sqrt{-1}}{5}}$
%Recall that {\red $\Si^\cC_\sch\sub\Si^\cC$} is the scheme part of $\Si^\cC$.
%\beq\label{Sisch}\Si^\sC_{\text{sch}}=\sum_{m_i=0} \Si^\sC_i.
%\eeq

%Let $q:\cWgg\to\cD_{g,d}$ be the forgetful morphism (forgetting $(\varphi,\rho,\nu)$) defined before. Note that $q$ is $\Gm$-equivariant, with $\Gm$ acts trivially on $\cD_{g,d}$. 

\begin{proposition}
The pair $q: \cWgg\to\cD_{g,\gamma}$ admits a tautological $\Gm$-equivariant relative perfect obstruction theory taking the form
$$\Bigl( R\pi\lsta\bl \cL(-\Si^\cC_{(1,\varphi)})^{\oplus 5}\oplus \cP(-\Si^\cC_{(1,\rho)})\oplus
(\cL\otimes\cN)\oplus \cN\br\Bigr)\dual\lra L\bul_{\cWgg/\cD_{g,\gamma}}.
$$
The (associated) virtual dimension of $\cWgg$ is (letting $\ell_{(1, \varphi)}=\#\{i\,| \,\gamma_i=(1, \varphi)\}$)
\beq\label{vdim}
\delta(g,\gamma,\bd)= (d_0-5\ell_{(1,\varphi)})+\bl d_\infty+\frac{1}{5}\sum_{i=1}^\ell m_i\br + 1-g +\sum_{i=1}^\ell (1-m_i).
\eeq
%where $m_i$ is the monodromy index at $\Sigma_i$.
\end{proposition}

\begin{proof}The construction of the obstruction theory is parallel to that in \cite[Prop 2.5]{CL}, and will be omitted.  
We compute its virtual dimension. Let $\xi=(\Si^\sC,\sC,\sL,\cdots)$ be a closed point in $\cW\lggd$.
Observe that when $\langle\gamma_i\rangle\ne \{1\}$, $\varphi|_{\Si^\sC_i}=0$. Thus by that $(\varphi,\nu_1)$
is  nowhere
vanishing, we see that $\nu_1|_{\Si^\sC_i}\ne 0$, and  the 
monodromy of $\sL\otimes\sN$ along $\Si^{\sC}_i$ is trivial.
%By the nowhere vanishing of $(\varphi,\nu_1)$ and $(\nu_1,\nu_2)$, we see that $\Si^\sC_i$ for $\gamma_i\ne 1$ must lie over
%where $\nu_2=0$, and thus the monodromy of $\sN$ along $\Si^{\sC}_i$ must be $\gamma_i^{-1}$. Further, if
%$\gamma_i=1$, then $(\nu_1,\nu_2)$ is nowhere vanishing
%implies that the monodromy of $\sN$ along $\Si^{\sC}_i$ must be $1$ too.
Therefore, 
$$
d_0=\deg(\sL\otimes\sN) \in \ZZ,\and d_\infty +\ofth \sum_{i=1}^\ell m_i =d_0-\deg \sL+ \ofth \sum_{i=1}^\ell m_i \in \ZZ.
$$
%t $\eps_i=0$ (resp. $=1$) if $\gamma_i\ne 1$ (resp. $=1$), and let $\Si^\sC_{\gamma=1}=\sum_i\eps_i \Si_i^\sC$. 
Applying the Riemann-Roch theorem for twisted curves:
$\chi(\sL)=\deg \sL + 1-g -\ofth\sum_{i=1}^\ell m_i$, we obtain that the
relative virtual dimension of $\cWgg\to \cD_{g,\gamma}$ is
$$ \chi(\sL(-\Si^\sC_{(1,\varphi)})^{\oplus 5}\oplus\sL^{\vee\otimes5}(-\Si^\sC_{(1,\rho)}) \otimes \omega_\sC^{\log}\oplus \sL\otimes \sN
\oplus \sN)
=d_0-5\ell_\varphi+ d_\infty +6-6g -\frac{4}{5}\sum_{i=1}^\ell m_i .
$$
Here we insert $\Si^\sC_{(1,\varphi)}$ and $\Si^\sC_{(1,\rho)}$ because of (4) and (5) in Definition \ref{def-curve}.
Thus, using that $\dim \cD_{g,\gamma}=5g-5 + \ell$, the virtual dimension of $\cW\lggd$ is as in \eqref{vdim}.
%$$
%\delta(g,\gamma,\bd) = d_0+  \bl d_\infty+\ofth \sum_{i=1}^\ell m_i\br  + 1-g + \sum_{i=1}^\ell(1-m_i).
%$$
%This proves the Proposition.
\end{proof}

The relative obstruction sheaf of $\cWgg\to \cD_{g,\gamma}$ is
$$\Ob_{\cWgg/\cD_{g,\gamma}}\defeq R^1\pi\lsta\bl\sL(-\Si^\cC_{(1,\varphi)})^{\oplus 5}\oplus 
%(\cL^{-5}\otimes\omega_{\sC/\cWgg})\oplus(\cL\otimes\cN);
\cP( -\Si^\cC_{(1,\rho)})\oplus
(\sL\otimes\sN)\oplus \sN\br,
$$
and the absolute obstruction sheaf $\Ob_{\cWgg}$ is the cokernel of the tautological map $q\sta T_{\cD_{g,\gamma}}\to \Ob_{\cWgg/\cD_{g,\gamma}}$,
fitting into the exact sequence
\beq\label{quot}
q\sta T_{\cD_{g,\gamma}}\lra \Ob_{\cWgg/\cD_{g,\gamma}} 
%R^1\pi\lsta\bl \sL^{\oplus 5}\oplus (\sL^{-5}\otimes\omega_{\sC/\cWgg})\oplus(\sL\otimes\sN)\oplus \sN\br
\lra \Ob_{\cWgg}\lra 0. 
\eeq

We define a cosection %of $\Ob_{\cWgg/\cD_{g,\gamma}}$ by the rule
\beq\label{co-1}
\sigma: \Ob_{\cWgg/\cD_{g,\gamma}}\lra \sO_{\cWgg}
\eeq
by the rule that at an $S$-point $\xi\in\cWgg(S)$, (in the notation $\xi=(\sC,\Sigma^\sC,\cdots)$ as in
\eqref{MSP0}),
\beq\label{mixed-cosection}
\sigma(\xi) %\Ob_{\cWgg}\otimes\cO_S\to\cO_S,\quad 
(\dot\varphi,\dot\rho,\dot\nu_{1},\dot\nu_{2}) = 5\rho\sum \varphi_{i}^4\dot\varphi_{i}
+\dot\rho\sum\varphi_i^5\in H^1(\omega_{\sC/S})\equiv H^0(\sO_\sC)\dual,
\eeq
where 
$$
(\dot\varphi,\dot\rho,\dot\nu_{1},\dot\nu_{2})\in
H^1\bl \sL(-\Si^\sC_{(1,\varphi)})^{\oplus 5}\oplus
\sP(-\Si^\sC_{(1,\rho)})\oplus
\sL\otimes\sN\oplus \sN\br.
$$
(Here $\sP=\sL^{\vee\otimes5}\otimes\omega^{\log}_{\sC/\cWgg}$.)
Note that the term $5\rho\sum \varphi_{i}^4\dot\varphi_{i}
+\dot\rho\sum\varphi_i^5$ a priori lies in $H^1\big (\sC,\omega_{\sC/S}^{\log}({-\Si^\sC_{(1,\rho)}})\big)$. However,
 when $\langle\gamma_i\rangle\neq (1, \rho)$, $\varphi_j|_{\Si^{\sC}_i}=0$.
%and when $\gamma_i=1$, the item (4) of Definition \ref{def-curve} ensures that
%$\rho|_{\Si^{\sC}_i}=0$.
Thus it lies in  $H^1(\sC,\omega_{\sC/S})$.

\begin{lemm}\label{lem:co}
The rule \eqref{mixed-cosection} defines a $\Gm$-equivariant homomorphism $\sigma$ as in \eqref{co-1}. 
Via \eqref{quot} the homomorphism
$\si$ lifts to a $\Gm$-equivariant cosection of
$\Ob_{\cWgg}$.
\end{lemm}

\begin{proof}
The proof that the cosection $\sigma$ lifts is exactly the same as in
\cite{CLL}, and will be omitted. 
That the homomorphism $\sigma$ is $\Gm$-equivariant is because 
$\Gm$ acts on $\cWgg$ via scaling $\nu_1$ and $\sigma$ is independent of $\nu_1$.
\end{proof}

As in \cite{KL}, we define the degeneracy locus of $\sigma$ to be %the closed reduced substack defined by
\begin{eqnarray}\label{deg-loci}
\cWgg^{-}(\CC)=\{\xi\in \cWgg(\CC)\mid \sigma|_\xi=0\},
%\cWgg^{-}(\CC)=\cWgg^{\text{pre}-}(\CC)\cap \cWgg(\CC).
\end{eqnarray}
endowed with the reduced structure. It is a closed substack of $\cWgg$.

\begin{lemm}\label{degenerate-locus}
The closed points of $\cWgg^-(\CC)$ are $\xi\in \cWgg(\CC)$ such that
\beq
(\varphi=0 )\cup ( \varphi_1^5+\cdots+\varphi_5^5=\rho=0)=\sC.
\eeq
\end{lemm}

\begin{proof}We consider individual terms in \eqref{mixed-cosection}.
Taking the term $\rho\varphi_i^4\dot\varphi_i$, by the vanishing along $\Si^\sC_i$ recalled
before the statement of Lemma \ref{lem:co}, we conclude that %(abbreviating $\Si^\sC=\sum_{i=1}^\ell \Si^\sC_i$),
$$
\rho\varphi_i^4\in  H^0\bl  \sL^{\vee}\otimes\omega_\sC^{\log}(-\Si^\sC)\br
=H^0\bl   \sL^{\vee}\otimes\omega_\sC\br.
$$
 By Serre duality, when $\rho\varphi_i^4\ne0$, there is a $\dot\varphi_i\in H^1(\sL)$ so that
$\rho\varphi_i^4\cdot \dot\varphi_i\ne 0\in H^1(\omega_\sC)$.

Repeating this  argument, we conclude that $\sigma|_\xi=0$ if and only if
$$\rho\varphi_1^4=\cdots=\rho\varphi_5^4=\varphi_1^5+\cdots+\varphi_5^5=0.
$$
This is equivalent to that $(\varphi=0 )\cup ( \varphi_1^5+\cdots+\varphi_5^5=\rho=0)=\sC$.
\end{proof}

Note that \eqref{deg-loci} makes sense for $\xi\in \cW^{\text{pre}}\lggd(\CC)$ as well. For convenience, we denote
$$\cW^{\text{pre}-}\lggd(\CC)=\bigl\{\xi\in \cW^{\text{pre}}\lggd(\CC)\mid \text{\eqref{deg-loci} holds for $\xi$}\bigr\};
%\and \cW^{-}\lggd(\bk)=\cW^{\text{pre}-}\lggd(\bk)\cap \cW\lggd(\bk).
$$
Applying  \cite{KL,CoVir}, we obtain the cosection localized virtual cycle
$$
[\cWgg]\virtloc \in A^T_{\delta} \cWgg^-,\quad \delta=\delta(g,\gamma,\bd).
$$

\subsection{MSP invariants}

\newcommand{\cCgg}{ {\cC_{g,\gamma,\bd}} }
\def\fo{\frac{1}{5}}
\let\ofth=\fo
\def\lred{_{\mathrm{red}}}
\newcommand{\MSP}{ {\mathrm{MSP}} }
\newcommand{\CR}{ {\mathrm{CR}} }
\newcommand{\tN}{ \tilde{N} }
\newcommand{\One}{\mathbf{1}}
\newcommand{\age}{ {\mathrm{age}} }

%\subsubsection{Evaluation maps}
Using the universal family \eqref{universal} %$(\cC,\Si^\cC, \cL,\cN, \varphi, \rho,\nu)$ of $\cWgg$,
%\begin{eqnarray*}
%&& \varphi \in H^0(\cCgg,\cL^{\oplus 5}),\quad
%\rho\in H^0(\cCgg,\cL^{\vee 5}\otimes\omega_{\cCgg/\cWgg}^{\log}),\\
%&& \nu_1\in H^0(\cCgg,\cL\otimes\cN),\quad 
%\nu_2\in H^0(\cCgg,\cN)
%\end{eqnarray*}
%Let $\gamma=(\gamma_1,\ldots,\gamma_\ell) \in (\mufive)^\ell$. 
we define the evaluation maps
(associated to the marked sections $\Si^\cC_i$):
$$
\ev_i: \cWgg \to X:= \PP^5 \cup  (\mufive)
$$
as follows. In case $\langle\gamma_i\rangle\neq 1$ (resp. $\gamma_i=(1,\varphi)$), we define $\ev_i$ to be the constant
map to $\langle\gamma_i\rangle \in \mufive$.
In the case $\gamma_i=(1,\rho)$,   for $s_i:\cWgg\to \cCgg$ the $i$-th marked section
of the universal curve,
by (2) of Definition \ref{def-curve} we have $s_i^*\rho=0$. Thus $s_i^*\nu_2$ is
nowhere vanishing, and $s_i^*\sN \cong \cO_{\cWgg}$.
Therefore, $s_i^*(\varphi,\nu_1)$ is a nonwhere vanishing section of
%$s_i^*(\sL^{\oplus 5}\oplus \cN\otimes \cL) \cong 
$s_i^*\sL^{\oplus 6}$, defining   the desired evaluation morphism
\beq\label{ev-i}
\ev_i=[s_i\sta\varphi_1,\cdots,s_i\sta\varphi_5,s_i\sta\nu_1]:\cWgg\to \PP^5
\eeq
such that $\ev_i^*\cO_{\PP^5}(1)= s_i^*\sL$.

Let $\Gm$ act on $\PP^5$ by 
$$
t\cdot[\varphi_1,\ldots,\varphi_5 , \nu_1]=[\varphi_1,\ldots,\varphi_5 , t \nu_1],
$$
and let $\Gm$ act trivially on $\mufive $.  It makes
$\ev_i$ $\Gm$-equivariant.

%\subsubsection{The state space}
\vsp
We introduce the MSP state space. 
As a $\CC$-vector space, the MSP state space and the $\Gm$-equivariant MSP state space are the cohomology
group and the $\Gm$-equivariant
cohomology group of the evaluation space $X=\PP^5\cup (\mufive)$:
$$
\cH^\MSP =H^*(X;\CC),\and \cH^{\MSP,\Gm} = H^*_\Gm(X;\CC). %= H^*_\Gm(\PP^5;\CC)\oplus  H^*_\Gm(\mufive\setminus\{1\};\CC),
$$
In terms of generators, we have
$$
H^*_\Gm(\PP^5;\CC) = \CC[H,\ft]/\langle H^5(H+\ft) \rangle,\and
%:
H^*_\Gm(\mufive;\CC) = \bigoplus_{j=1}^4 \CC[\ft] \One_{\frac{j}{5}},
$$
and the (non-equivariant) MSP state space is by setting $\ft=0$, while the grading is given by
\beq \label{eqn:deg}
\deg H= 2,\quad \deg\ft=2\and \deg \One_{\frac{j}{5}}=\frac{8j}{5}.
\eeq

\vsp

We formulate the gravitational descendants.
Given 
$$
a_1,\ldots, a_\ell\in \ZZ_{\geq 0}, \quad
\phi_1,\ldots,\phi_\ell\in \cH^\MSP =H^*(X;{\CC}),
$$
we define the MSP-invariants
\beq \label{eqn:invariants}
\langle\tau_{a_1}\phi_1\cdots \tau_{a_\ell}\phi_\ell \rangle^\MSP_{g,\ell,\bd}
:=\int_{[\cW_{g,{\ell},\bd}]\virtloc} \prod_{k=1}^\ell \psi_k^{a_k} \ev_k^*\phi_k\in {\CC}.
\eeq
where
$$
[\cW_{g,\ell,\bd}]\virtloc = \sum_{\gamma\in (\mufive)^\ell} [\cWgg]\virtloc.
$$
%Given 
%$$
%a_1,\ldots, a_\ell\in \ZZ_{\geq 0}, \quad
%\phi_1,\ldots,\phi_\ell\in \cH^{\MSP,\Gm} =H^*_\Gm(X;\CC),
%$$
Similarly, we define $\Gm$-equivariant genus $g$ MSP-invariants to be
\beq \label{eqn:Gm-invariants}
 \langle\tau_{a_1}\phi_1\cdots \tau_{a_\ell}\phi_\ell \rangle^{\MSP,\Gm}_{g,\ell,\bd}
 :=\int_{[\cW_{g,\ell,\bd}]\virtloc} \prod_{k=1}^\ell \psi_k^{a_k} \ev_k^*\phi_k\in 
H^*(B\Gm;\QQ)=\QQ[\ft],
\eeq
where $\phi_i\in \cH^{\MSP,\Gm}$, and
$$
[\cW_{g,\ell,\bd}]\virtloc = \sum_{\gamma\in (\mufive)^\ell} [\cWgg]\virtloc.
$$
(Here we use the same $[\cdot]\virtloc$ to mean the $T$-equivariant class.)
Suppose $\phi_1,\ldots,\phi_\ell$ are homogeneous, and let
$$\be(a_\cdot,\phi_\cdot)\defeq  {\sum_{k=1}^\ell (a_k +\frac{\deg \phi_k}{2}) -(d_0+d_\ell + 1-g +\ell)}.
$$
By the formula of the virtual dimension of $\cW_{g,\ell,\bd}$, we
see that
$$
\langle\tau_{a_1}\phi_1\cdots \tau_{a_\ell}\phi_\ell \rangle^{\MSP,\Gm}_{g,\ell,\bd}
\in \CC \ft^{\be(a_\cdot,\phi_\cdot)}.
$$
In case $\be(a_\cdot,\phi_\cdot)<0$, we have vanishing
\beq\label{v-1}
\Bigl[\bt^{-\be(a_\cdot,\phi_\cdot)}\cdot \langle\tau_{a_1}\phi_1\cdots \tau_{a_\ell}\phi_\ell \rangle^{\MSP,\Gm}_{g,\ell,\bd}\Bigr]_0=0,\quad
\eeq 
where $[\cdot]_0$ is the dimension $0$ part of the pushforward to $H_0(pt)$.
%\sum_{k=1}^\ell (a_k +\frac{\deg \phi_k }{2}) < d_0+d_\ell + 1-g +\ell$.
%\vsp

By virtual localization, we  will express all genus full descendant MSP invariants in terms of 
(1): GW invariants of the quintic threefold $Q\subset \Pf$; (2): FJRW invariants of 
the Fermat quintic; and 
(3): the descendant integrals on $\Mbar_{g,n}$.
The invariants in item (1) has been solved in genus zero \cite{Gi, LLY} and genus one for all degrees 
\cite{LZ, Zi}, and in all genus for degree zero,
those in item (2) has been solved in genus zero \cite{CR}, and those in item (3) have been solved in all genera.
One of our goals to introduce MSP invariants is to use vanishing \eqref{v-1} to obtain recursive relations to 
determine (1) and (2) in all genus. This will be addressed in details in the sequel \cite{CLLL}.

\section{Properness of the degeneracy loci}
\def\Slsta{{S\lsta}}

In this section, we will prove that $\cW\lggd^-$ is
proper over $\mathbb C$.

% and of finite type, and $\cW\lggd$ is separated near $\cW^-\lggd$.}

\subsection{The conventions}\label{Sub3.1}

In this section, we denote by $\eta_0\in S$ a closed point in an affine smooth curve, 
and denote $S\lsta=S-\eta_0$ its complement. 

In using valuative criterion to prove properness, we need to take a finite base change $S'\to S$
ramified over $\eta_0$. By shrinking $ S$ if necessary, 
we assume there is an \'etale $S\to \Ao$ so that $\eta_0$ is the only point lying over $0\in\Ao$.
This way, we can take $S'$ to be of the form
$S'=S\times_{\Ao} \Ao$, where $\Ao\to \Ao$ is via $t\mapsto t^k$ for some integer $k\ge 2$. This way,
$\eta'_0\in S'$ lying over $\eta_0\in S$ is also the only point lying over $0\in \Ao$. One particular choice
of $S'$ is the degree five base change: $S_5=S\times_{\Ao}\Ao\to S$, where $\Ao\to\Ao$ is via $t\mapsto t^5$.

% $K$ its field of fraction and $\CC$ its residue field;
%denote by $\eta\in S$ its generic point, and denote by
%$\eta_0\in S$ its closed point. %and $K$ be the fractional field of $R$. We will denote
%We denote by $S'=\spec R'\to S$ a finite base change, and denote by $\eta'$ and $\eta'_0\in S'$ its generic and
%closed points. 
%For a family $\sC\to S$, we always denote
%by $\sC_0=\sC\times_S\eta_0$ the central fiber.
%To prove the properness, we need to show that certain family over $\eta$ can be extended
%to over $S'$ after a finite base change $S'\to S$.
To keep notations easy to follow, for a property P that holds after a finite base change $S'\to S$
of a family $\xi$ over $S$, we will say
``after a finite base change, the family $\xi$ has the property P'', meaning that we have already done
the finite base change $S'\to S$ and then replace $S'$ by $S$ for abbreviation of notations. 

In this and the next section, for $\xi\in \cW^{\text{pre}}\lggd(\CC)$ or $\cW^{\text{pre}}\lggd(S)$, we understand
\beq\label{xista}\xi=\big(\Si^\sC, \sC, \sL,\sN, \varphi, \rho, \nu\big)\in \cW^{\text{pre}}\lggd.
\eeq
Similarly, we will use subscript ``$\ast$" to denote families over $S\lsta$. Hence $\xi\lsta\in\cW^{\text{pre}}\lggd(S\lsta)$
will be of the form 
\beq\label{xi2}
\xi_*=\big(\Si^{\sC_\ast}, \sC_*, \sL_*,\sN\lsta, \varphi\lsta, \rho_*, \nu_*\big).
\eeq

\vsp

We first prove a simple version of the extension result we need.
%, the extension is easy since it is equivalent to a problem for stable maps. 
%Let's assume that $\rho\lsta=0$ on $\sC_*$ where $\sC_*$ is not necessarily geometrically irreducible nor smooth. 

\begin{prop}\label{prop-extension0}
Let $\xi_*\in \cW^-\lggd(\Slsta)$
be such that $\rho\lsta=0$. Then after a finite base change, $\xi\lsta$ extends to a $\xi\in \cW^-\lggd(S)$.
\end{prop}
\begin{proof}
Since $\rho\lsta=0$, $\nu_{2*}$ is nowhere vanishing and $\sN\lsta\cong\sO\lsta$. 
Thus $(\varphi\lsta, \nu_{1*})$ is a nowhere vanishing section of $H^0(\sL\lsta^{\oplus 5}\oplus \sL\lsta)$,  and  induces a morphism $f_*$ from 
$(\Si^{\sC\lsta}, \sC\lsta)$ to $\mathbb P^5$, such that $(f\lsta)\sta\sO_{\mathbb P^5}(1)\cong\sL\lsta$. 
By the stability assumption of $\xi\lsta$,
%\big(\Si^\sC_\ast, \sC_*, \sL_*, \varphi\lsta, \rho_*, \nu_*\big)$, 
this morphism is an $S\lsta$-family of stable maps. By the properness of moduli stack of stable maps, after a
finite base change we can extend $(\Si^{\sC_*}, \sC_*)$ to $(\Si^\sC, \sC)$ and extend $f_*$ to an $S$-family of stable maps $f$  
from $(\Si^\sC, \sC)$ to $\mathbb P^5$. Let $\sL=f\sta\sO_{\Pf}(1)$, which is an extension of $\sL\lsta$. 
Then $f$ is provided by a section $(\varphi, \nu_1)\in H^0(\sL^{\oplus 5}\oplus \sL)$, extending
$(\varphi\lsta, \nu_{1*})$. 
Since $[f,\Si^\sC,\sC]$ is stable, 
the central fiber $\sC_0$ is a connected curve with at worst nodal singularities. Define $\sN\cong \sO_{\sC}$ and $\nu_2$ to be the isomorphism $\sN\cong \sO_{\sC}$ extending $\nu_{2*}$. Define $\rho=0$.
Then $\xi=\big(\Si^\sC, \sC, \sL, \varphi, \rho, \nu\big)$ is a desired extension.
\end{proof}

%By valuative criterion, we need to show the existence of an extension 
%$\xi\in \cW^-\lggd(S')$ %=\big(\Si^\sC,  \sC, \sL, \varphi, \rho, \nu\big)$ in $ \cW^-\lggd(\Slsta)$ 
%after a finite base change $S'\to S$. 
The case involving $\varphi\lsta=0$ over some irreducible components is technically more involved.
We will treat this case by first studying the case $\sC\lsta$ is 
smooth. For this, 
we characterize stable objects in $\cW^{\text{pre}}\lggd(\CC)$. We say that an irreducible component $\sE\sub\sC$
is a rational curve if it is smooth and its coarse moduli is isomorphic to $\Po$.

\begin{lemm}\label{prepare1}
Let $p_1\ne p_2\in\Po$ be two distinct closed points,  $G\le \Aut(\Po)$ be the subgroup fixing $p_1$ and $p_2$,
and $L$ be a $G$-linearized line bundle on $\Po$ such that
$G$ acts trivially on $L|_{p_1}$. Then the following holds:
\begin{enumerate}
\item any invariant $s\in H^0(L)^G$ with $s(p_1)=0$ must be the zero section;
\item suppose $G$ acts trivially on $L|_{p_2}$, then $L\cong\sO_{\Po}$.
\end{enumerate}
% ane let $G\to \Aut_{\Po}(L)$ be the unique $G$-linearization of $L$ so that $G$ leaves $L|_p$ fixed.
\end{lemm}

\begin{proof}
Both are well-known. For convenience, we prove (1). By averaging over the maximal compact subgroup
of $G\cong\Gm$, we can find an $\bar s\ne 0\in H^0(L|_{\Po-p_2})^G$ such that $\bar s(p_1)\ne 0$. Since $\Po-\{p_1,p_2\}$
is a single $G$ orbit, $\bar s$ is nowhere vanishing. Now let $s\in H^0(L)^G$ such that $s(p_1)=0$. 
Then $s/\bar s$ is a regular function on $\Ao=\Po-p_2$, $G$-invariant, and vanishes at $p_1$. 
%But we know the only $G$-invariant functions on $\Ao$ are constant functions. 
Thus $s=0$.
\end{proof}

\begin{lemm}\label{stable-cri}
Let $\xi\in \cW^{\mathrm{pre}-}\lggd(\CC)$. It is unstable if and only if one of the following holds:
\begin{enumerate}
\item $\sC$ contains a rational curve $\sE$ such that $\sE\cap (\Si^\sC\cup \sC_{\mathrm{sing}})$ contains two points, 
and $\sL^{\otimes 5}|_\sE\cong \sO_\sE$;
\item $\sC$ contains a rational curve $\sE$ such that $\sE\cap (\Si^\sC\cup \sC_{\mathrm{sing}})$ contains one point, 
and either $\sL|_\sE\cong \sN|_\sE\cong \sO_\sE$ or $\rho|_\sE$ is nowhere vanishing;
\item $\sC$ is a smooth rational curve with
$\Si^\sC=\emptyset$, $d_0=d_\infty=0$.
%contains one nodal point of $\sC$; $\sE\cap \Si^\sC=\emptyset$, and $\rho|_\sE$ is nowhere vanishing;
\item $\sC$ is irreducible, $g=1$, $\Si^\sC=\emptyset$, and $\sL^{\otimes 5}\cong\sO_\sC$ and $\sL\dual\cong\sN$.
\end{enumerate}
\end{lemm}

\begin{proof} We first prove the necessary part. 
Let $\xi\in  \cW^{\mathrm{pre}-}\lggd(\CC)$ be unstable. For each irreducible $\sE\sub\sC$, let $\Aut_{\sE}(\xi)$ be the subgroup of  $\Aut(\xi)$ mapping $\sE$ to itself. There exists an $\sE$ such that $\Aut_{\sE}(\xi)$ is  of infinite order. %We first consider the case where $\Aut(\xi)$ is infinite.
If  the image  of $\Aut_{\sE}(\xi)\to\Aut(\sE)$ is finite, then
for a finite index subgroup $G'\le \Aut_{\sE}(\xi)$, $G'$ leaves $\sC$ fixed, thus $G'$ acts on $\xi$ by acting on the 
line bundles $\sL$ and $\sN$ via scaling. However, by  Definition \ref{def-curve}, that $G'$ leaves $(\varphi,\rho,\nu)$ invariant
implies that the image of $G'\to\Aut(\sL)\times\Aut(\sN)$ is finite. Since this arrow is injective,
it contradicts to the fact that $G'$ is infinite. Thus the group $G=\text{im}( \Aut_{\sE}(\xi)\to \Aut(\sE))$ is of infinite order.

We now consider the case where $\sE$ has arithmetic genus zero (thus smooth). We divide it into
several cases. The first is when $\sE\cap(\Si^\sC\cup \sC_{\text{sing}})$ contains two points, 
say $p_1$ and $p_2\in\sE$.
Then $G$ fixes both $p_1$ and $p_2$.
%Since $C$ is smooth and $C_0$ is reduced, $E\sub C$ must be a (-2)-curve. 
We claim that $G$ acts trivially on $\sL^{\otimes 5}|_{p_1}$. Indeed, 
when $\nu_2(p_1)=0$, then $\rho(p_1)\ne 0$, 
which implies that $G$ acts trivially on $\sL^{\vee\otimes 5}\otimes\omega_{\sC/S}^{\log}|_{p_1}$.
Since $G$ acts trivially on $\omega_{\sC/S}^{\log}|_{p_1}$, $G$ acts trivially on $\sL^{\otimes 5}|_{p_1}$.
When $\rho(p_1)=0$ and $\nu_1(p_1)\ne 0$, since $\nu_2(p_1)\ne 0$, then $G$ acts trivially on $\sN|_{p_1}\cong
\sL\dual|_{p_1}$.
When $\rho(p_1)=\nu_1(p_1)= 0$, then $\varphi(p_1)\ne 0$, which implies that $G$ acts trivially on $\sL|_{p_1}$.
Therefore, in any case $G$ acts trivially on $\sL^{\otimes 5}|_{p_1}$. By the same reason, it
acts trivially on $\sL^{\otimes 5}|_{p_2}$. Applying Lemma \ref{prepare1}, we conclude that $\sL^{\otimes 5}\cong \sO_\sE$. 
This is Case (1).

The second case (when $g_a(\sE)=0$) is when $\sE\cap(\Si^\sC\cup \sC_{\text{sing}})$ contains one point, 
say $p\in\sE$. Suppose $\rho|_\sE=0$, 
then $\nu_2|_\sE$ is nowhere vanishing, implying $\sN|_\sE\cong \sO_\sE$.
Thus $(\varphi_1,\cdots,\varphi_5, \nu_1)|_\sE$ is a nowhere vanishing section of $H^0(\sL^{\oplus 6}|_\sE)$.
Since $G$ is infinity and $(\varphi_1,\cdots,\varphi_5, \nu_1)|_\sE$ is $G$-equivariant, this is possible only if
$\sL|_\sE\cong\sO_\sE$ and $(\varphi_1,\cdots,\varphi_5, \nu_1)|_\sE$ is a constant section. This is Case (2).

The other case is when $\rho|_\sE\ne0$.
We argue that $\rho|_\sE$ is nowhere vanishing. Otherwise, $\nu_2|_\sE\ne 0$, and then $\deg\sN|_{\sE}\ge 0$.
Since $\xi\in \cW^{\text{pre}-}\lggd(\CC)$, we have $\varphi|_\sE=0$, thus $\nu_1|_{\sE}$ is nowhere vanishing
and $\sL\dual|_\sE\cong \sN|_\sE$. Because $\rho|_\sE\ne 0$ and $\deg\omega_\sC^{\log}|_\sE=-1$,
we must have $\deg\sL|_\sE<0$. Thus $\nu_2\in H^0(\sN|_\sE)=H^0(\sL\dual|_\sE)$ must vanish at some point.
Let $p_1$ and $p_2\in\sE$ be such that $\rho(p_1)=0=\nu_2(p_2)$. Since $(\rho,\nu_2)$ is nowhere vanishing,
we have $p_1\ne p_2$. Furthermore, since $G$ fixes $p$, $p_1$ and $p_2$, and is infinite, $p=p_1$ or $p_2$.

Suppose $p=p_1$, then $\nu_2(p_1)\ne 0$ implies that $G$ acts trivially on $\sL|_{p_1}$. Hence $G$ acts
trivially on $\sL\mof\otimes\omega_\sC^{\log}|_{p_1}$. For the same reason, $\rho(p_2)\ne 0$
implies that $G$ acts trivially on $\sL\mof\otimes\omega_\sC^{\log}|_{p_2}$. Applying Lemma
\ref{prepare1}, we conclude that $\sL\mof\otimes\omega_\sC^{\log}|_\sE\cong\sO_\sE$,
contradicting to $\rho\ne 0$ and vanishing somewhere.

Suppose $p=p_2$. By the same reasoning, we conclude that
$G$ acts trivially on $\sL\mof\otimes\omega_\sC^{\log}|_{p_2}$ and $\sL|_{p_1}$.
Since $G$ acts trivially on $\omega_\sC^{\log}|_{p_2}$, applying Lemma \ref{prepare1}, we conclude that
$\deg\sL|_\sE=0$, contradicting to $\deg\sL|_{\sE}<0$.
Combined, we proved that if $\rho|_\sE\ne 0$, then $\rho|_\sE$ is nowhere vanishing. This is Case (2).

The third case (when $g_a(\sE)=0$) is when $\sE\cap(\Si^\sC\cup \sC_{\text{sing}})=\emptyset$. 
A parallel argument shows that in this case we must have
$\sL\cong\sN\cong\sO_\sC$. This conclude the study of the case $g_a(\sE)=0$.
%Since $\sC$ is connected, thus $\sC=\sE\cong \Po$. We claim that then $\sL\cong \sN\cong \sO_\sC$.
%Since $\xi\in \cW^{\text{pre}-}\lggd(\CC)$, either $\varphi=0$ or $\rho=0$. Suppose $\rho=0$, then
%$\nu_2$ is nowhere vanishing, implying $\sN\cong \sO_\sE$, and $(\varphi,\nu_1)$ is nowhere vanishing
%and defines a morphism to $\PP^5$. 
%Since it is $G$-equivariant, it is a constant map. Thus $\sL\cong\sO_\sC$, which is Case (3). 
%
%Suppose $\varphi=0$. 
%Then $\nu_1$ is nowhere vanishing and $\sL\dual\cong\sN$. We claim that this is impossible. 
%Since $\sC\cong \Po$ is a scheme, $\deg\sL\in \ZZ$. Thus as $\rho\ne 0$, $\deg\sL\le -1$,
%which forces both $\rho$ and $\nu_2$ vanish somewhere. Since $(\rho,\nu_2)$ is nowhere vanishing,
%$\rho$ and $\nu_2$ vanish at distinct points, say $\rho(p_1)=0$ and $\nu_2(p_2)=0$.
%Note that $G$ fixes $p_1$ and $p_2$.
%Then $G$ acts trivially on $\sN|_{p_1}$, thus we can write $\sN=\sO_\sC(a p_2)$ and then
%$\sL\mof\otimes\omega_\sC^{\log}\cong \sO_\sC(-p_1+(5a-1)p_2)$, as a $G$-linearized line bundle.
%Since $\rho(p_2)\ne 0$, $G$ acts trivially on $\sL\mof\otimes\omega_\sC^{\log}|_{p_2}$.
%This forces $5a-1=0$, impossible since $a$ is an integer. 
%
The remaining case is when $g_a(\sE)=1$, then $\sE\cap \Si^\sC=\emptyset$, and a similar argument shows that it must
belong to Case (4). Combined, this proves that
if $\xi$ is unstable, then one of (1)-(4) holds.

% and 
%$\sE\cup \overline{(\sC-\sE)}=\emptyset$.
%Since $\sC$ is connected, $\sE=\sC$. In case $\sE$ is smooth, since $G$ acts homogeneously on $\sE$ and lifts to 
%$\sL$ and $\sN$, we must have $\sL=\sN=\sO_\sC$, which is Case (4). The other case is when
%$\sE$ is nodal, then since $G$ lifts to actions on $\sL$ and $\sN$, and applying Lemma \ref{prepare1} to the normalization of
%$\sE$ we conclude that $\deg\sL=\deg\sN=0$. If both $\nu_1$ and $\nu_2\ne 0$, then we must have
%$\sL\cong\sN\cong\sL_\sC$; if $\nu_1\ne 0$ but $\nu_2=0$, then $\rho\ne 0$ and then
%$\sL^{\otimes 5}\cong \sO_\sC$ and $\sL\dual\cong \sN$; if $\nu_1=0$ but $\nu_2\ne 0$, then 
%$\varphi\ne 0$ and we get $\sL\cong\sN\cong \sO_\sC$. All of them belong to Case (4).

We now prove the other direction that whenever there is an $\sE\sub\sC$ that satisfies one of (1)-(4), then $\xi$ is unstable. Most of the cases can be argued easily, except a sub-case of (2) when $\rho|_\sE$ is 
nowhere vanishing, which we now prove. 

Since $\sE\cap (\Si_{\sC}\cup \sC_{\text{sing}})$ contains one point, say $p\in \sC$, we have
$\deg\omega^{\log}_{\sC}|_{\sE}=-1$. Since $\rho|_{\sE}$ is nowhere vanishing, we have
$\deg\sL^{\otimes 5}|_{\sE}=-1$.  Thus $p$ must be a stacky point. Hence $\sE\cong\mathbb P_{1,5}$ as stacks. Let $\mathbb P_{1, 5}=\text{Proj} (k[x, y])$ where $\deg x=1$ and $\deg y=5$. Then $p$ corresponds to the point $[0, 1]$. Let $G=\mathbb G_a=\mathbb C$ act on $\mathbb P_{1, 5}$ as follows: $x\to x$, $y\to \lambda x^5+y$ for $\lambda \in G$. The $G$-action on $\mathbb P_{1, 5}$ lifts to an action of $\omega^{\log}_{\sC}|_{\sE}$ as well as $\sL^{\vee}|_{\sE}\cong \sO_{\mathbb P_{1,5}}(1)$. One can check via local calculations that $G$ acts trivially on $\big(\omega^{\log}_{\sC}|_{\sE}\big)|_p$ as well as $\big(\sL^{\vee}|_{\sE}\big)|_p$, thus trivially on 
$\big(\sL^{\vee\otimes 5}\otimes \omega^{\log}_{\sC}\big)|_p$. Since $\sL^{\vee\otimes 5}\otimes \omega^{\log}_{\sC}|_{\sE}\cong \sO_{\sE}$,  and since $\mathbb G_a=\mathbb C$ has no non-trivial characters, by Prop.1.4 in \cite{FMK}, $G$ acts on $\big(\sL^{\vee\otimes 5}\otimes \omega^{\log}_{\sC}\big)|_{\sE}$ trivially as well.  Hence $G$ acts trivially on $\rho|_{\sE}$. Therefore the group $G$ is a subgroup of the automorphism group of $\xi|_{\sE}$. 
\end{proof}

\begin{coro}\label{split}
Let $\xi\in \cW^{\mathrm{pre}-}\lggd(\CC)$. Let $\pi: \ti\sC\to\sC$ be the normalization of $\sC$,
%and let $\ti\sC_i$ be its connected components with $\pi_i:\ti\sC_i\to\sC$ the tautological morphism.
let $\Si^{\ti\sC}=\pi\upmo(\Si^\sC\cup \sC_{\mathrm{sing}})$, and let $(\ti\sL,\ti\sN,\ti\varphi,\ti\rho,\ti\nu)$ be the pullback of 
$(\sL,\sN,\varphi,\rho,\nu)$ via $\pi\sta$. Write $\ti\sC=\coprod_a \ti\sC_a$ the connected component decomposition,
and let $\ti\xi_a$ be $(\Si^{\ti\sC}\cap \ti\sC_a,\ti\sC_a)$ paired with 
$(\ti\sL,\ti\sN,\ti\varphi,\ti\rho,\ti\nu)|_{\ti\sC_a}$.
Then $\xi$ is stable if and only if all $\ti\xi_a$ are stable. 
\end{coro}

\begin{proof}
If $\xi$ is unstable, then it contains an irreducible $\sE$ satisfying one of (1)-(4) in Lemma \ref{stable-cri}.
This $\sE$ (or its normalization) will appear in one of $\ti\xi_a$, making it unstable. The other direction is
the same. This proves the Corollary.
\end{proof}

\subsection{The baskets}

 %$\xi\lsta=(\Si^{\sC\lsta},\sC\lsta,\sL\lsta,\cdots)$. 
We will first studying a special case.

\vsp
\noindent
{\bf Special type}:  {\sl Let $\xi\lsta\in \cW^-\lggd(S\lsta)$ be of the form \eqref{xi2} such that 
$\sC_\ast$ is smooth (and connected), $\varphi\lsta=0$,
$\nu_{2\ast}\ne 0$ 
and $\rho\lsta\ne0$.}

%The case $\nu_{2\ast}=0$ will be treated separately.

%For any stack $\sX$ of finite presentation, we denote by $[\sX]_r$ its pure $r$-dimensional part.
%Assume that both $(\rho_*=0)$ and $(\nu_{2*}=0)$ do not contain vertical divisors of $\sC_*\to S_*$.

\begin{prop}\label{prop-extension1}
Let $\xi\lsta$ over $S\lsta$ %\lsta\in \cW^-\lggd(S\lsta)$ 
be of special type. %, namely, $\sC\lsta$ is smooth and geometrically irreducible.
%Assume further that $\rho_*\ne 0$ and $\nu_{2*}\ne 0$ (i.e. non-trivial).
Then after a finite base change, 
\begin{enumerate}
\item[({\bf a1})] $(\Si^{\sC\lsta}, \sC_*)$ extends to a pointed twisted curve $(\Si^\sC, \sC)$ over $S$ such that $\sC-\Si^\sC$ is a scheme, $\sC$ is smooth, and
the central fiber $\sC_0$ %=\sC\times_S \eta_0$ 
is reduced with at worst nodal singularities and smooth irreducible components;
\item[({\bf a2})] $\sL_*$ and $\sN_*$ extend to invertible sheaves $\sL$ and $\sN$ respectively on $\sC$ so that $\nu\lsta$ extends
to a surjective $\nu=(\nu_{1},\nu_{2}): \sL\dual\oplus \sO_{\sC}\to\sN$;
\item[({\bf a3})] $\rho\lsta$ %either extends to the zero section of $\sL^{\vee\otimes 5}\otimes \omega^{\log}_{\sC/S}$ or 
extends to a $\rho\in \Gamma(\sL^{\vee \otimes 5}\otimes \omega^{\log}_{\sC/S}(\sD))$ for a divisor 
$\sD\subset\sC$ contained in the central fiber $\sC_0$ such that $\rho$
restricting to every irreducible component of $\sC_0$ is non-trivial; %$D_i\sub  \sC_0$.
\item[({\bf a4})] $\overline{(\rho_*=0)}\cap \overline{(\nu_{2*}=0)}=\emptyset$,  $\overline{(\rho_*=0)}$ and $\overline{(\nu_{2*}=0)}$ intersect $\sC_0$ transversally.
\end{enumerate}
\end{prop}

\begin{proof}
%Then $\nu_{1*}\colon \sL_*\dual\to \sN_*$ is an isomorphism,   and $\rho_*$ is a non-zero section of $\sL_*^{\dual 5}\otimes \otimes  \omega^{\log}_{\sC_*/S_*}$.
First, possibly after  a finite base change, we can assume that $\overline{(\rho_*=0)}\cup  
\overline{(\nu_{2*}=0)}$ is a union of disjoint sections of $\sC\lsta\to\Slsta$,
and that if we let $\Si\lsta^{\text{ex}}$ be the union of those sections of $\overline{(\rho_*=0)}\cup  
\overline{(\nu_{2*}=0)}$ that are not contained in $\Si^{\sC\lsta}$, then $\Si\lsta^{\text{ex}}$ is disjoint from
$\Si^{\sC\lsta}$. %$\Slsta$ is a generic point, this is possible.
If $(\Si^{\sC\lsta}\cup \Si\lsta^{\text{ex}},\sC\lsta)$ is a stable pointed curve,   let
$\Si\lsta^{\text{au}}=\emptyset$.
Otherwise,   let $\Si\lsta^{\text{au}}$ be some extra sections of $\sC\lsta\to \Slsta$, 
disjoint from $\Si^{\sC\lsta}\cup \Si\lsta^{\text{ex}}$,
so that after letting $\Si\lsta^{\text{comb}}=\Si^{\sC\lsta}\cup \Si\lsta^{\text{ex}}\cup \Si\lsta^{\text{au}}$, 
the pair $(\Si\lsta^{\text{comb}}, \sC\lsta)$ is stable.

Since $(\Si^{\text{comb}}_*, \sC\lsta)$ is stable,
possibly after a finite base change, it extends to an $S$-family of stable twisted curves $(\Si^{\text{comb}}, \sC')$
such that all singular points of its central fiber $\sC'_0$  are non-stacky. Thus
after blowing up $\sC'$ along the singular points of $\sC'_0$ if necessary, 
taking a finite base change, and followed by a minimal desingularization,
we can assume that the resulting family $(\Si^\sC,\sC)$ is a family of pointed twisted curves with smooth $\sC$ satisfying Condition (a1).  Condition (a4) is satisfied due to the construction.

Since  $\varphi_*$ is identically zero,  $\nu_{1*}$ is an isomorphism.  We can extend $\sN_*$ to an invertible sheaf $\sN$ on $\sC$ so that $\nu_{2*}$ extends to a section $\nu_2$ of $\sN$. Let $\sL\cong \sN^{\vee}$. We extend $\nu_{1*}$ to an isomorphism $\nu_1\colon \sL^{\vee}\to \sN$, and extend $\rho_*$ to a section $\rho$ 
satisfying (a3), for a choice of $\sD$. This proves the proposition.
\end{proof}

We will work with the coarse moduli $C$  of $\sC$. 
%To simplify the discussion, we first look at the case where $\sC_0$ has no singular irreducible components.

\begin{defi} 
An $S$-family of pre-stacky pointed nodal curves is a flat $S$-family $(\Si,C)$ of
pointed nodal curves (i.e. not twisted curves) so that each marked-section 
$\Si_i$ (of $\Si$) is either assigned pre-stacky or 
assigned regular. 
We call it a good family if in addition $C$ is smooth, and all irreducible components of the central finer
$C_0=C\times_S\eta_0$ are smooth.
 %A good $S$-family of $n$-pointed twisted curves consists of a two-dimensional smooth DM stack $\sC$, a
%proper flat morphism $\pi: \sC\to S$, and $n$ sections $\Si$ %=\Si_1\coprod\cdots\coprod \Si_k$ of $\pi$ 
%such that all irreducible components of $\sC_0$ are smooth  and $\sC-\Si$ is a scheme.
%%(1). the reduced part $(\sC_0)\lred$ has only normal crossing singularities;\\
%%(2). the complement $\sC-\Sigma$ is a scheme.
\end{defi}

Note that if we let $C$ be the coarse moduli of the $\sC$ in Proposition \ref{prop-extension1}, let
$\Si_i\sub C$ be the image of $\Si^\sC_i\sub \sC$ under $\sC\to C$, and call $\Si_i$ pre-stacky if
$\Si^\sC_i$ is stacky, and call it regular otherwise, then $(\Sigma, C)$ with this assignment is a good
$S$-family of pre-stacky pointed nodal curves. 

%For simplicity, we will reorder the $\Si_i$'s so that the first
%$m$ of them are pre-stacky and the remainder regular.

Given an $S$-family of pointed twisted curve $(\Si^\sC, \sC)$ so that the only non-scheme points of $\sC$ are
possibly along $\Si^\sC$, applying the procedure
described, we obtain a pre-stacky pointed nodal curve $(\Si, C)$. We call this procedure
un-stacking. Conversely, applying the root construction (cf. \cite{A-G-V, Cad}) to
the $S$-family of pre-stacky pointed nodal curves $(\Si,C)$ obtained,
we recover the original family $(\Si^\sC,\sC)$;
we call the later the stacking of $(\Si,C)$.
%To keep track of the extensions $\sC, \sL$ and $\sN$, etc., we introduce the notion of basket (of divisors).

\begin{defi}\label{bask}
Let $(\Si,C)$ be a good $S$-family of pre-stacky pointed nodal curves,
%of which the first $m$ sections $\Si_{i\le m}$ are pre-stacky, 
and let $D_i$, $i\in \Lam$, be irreducible components of $C_0$. A pre-basket of $(\Si,C)$ is a data
\beq\label{BAS}\cB=(B+{\sum}_{i\in\Lam} l_iD_i, A+{\sum}_{i\in \Lam} m_iD_i),
\eeq
where
\begin{enumerate}
\item  $A=\sum_{i=1}^{k_1}a_i A_i$, where $A_1,\cdots, A_{k_1}$ are disjoint sections of $C\to S$
such that for any pair $(i,j)$, either $A_i\cap \Sigma_j=\emptyset$ or $A_i=\Sigma_j$, 
$a_i\in \ofth\ZZ_{>0}$
when  $A_i=\Si_{j}$ for some pre-stacky $\Si_j$, otherwise $a_i\in \ZZ_{>0}$;
\item  $B=\sum_{i=1}^{k_2}b_i B_i$, where $b_i\in \ZZ_{>0}$, $B_1,\cdots, B_{k_2}$ are disjoint sections of $C\to S$
such that for any pair $(i,j)$, either $B_i\cap \Sigma_j=\emptyset$ or $B_i=\Sigma_j$, and  when $B_i=\Sigma_j$,
 $\Si_j$ must be assigned regular;
\item $A_1,\cdots, A_{k_1}, B_1,\cdots, B_{k_2}$ are mutually disjoint and intersect $C_0$ transversally;
\item $5m_i \in \ZZ$ and $l_i\in \ZZ$;
\end{enumerate}
such that %we have isomorphism
%and the and $A_0=\sum_{j=1}^n} A_\Si$, $A_0$ and $B$ are disjoint Cartier divisors of $C$ having no vertical components of $C\to S$,
%$A_0=$
\beq\label{basket*}
\sO_{C}(B+\sum l_iD_i)\cong \sO_{C}(5A+\sum 5m_iD_i)\otimes \omega^{\log}_{C/S}.
\eeq
We call $\cB$ a basket if  in addition it satisfies $l_i\ge 0$, $m_i\ge 0$ and $l_im_i=0$ for all $i$.
\end{defi}
%$l_i\in\ZZ$ and $m_i\in\ofth\ZZ$, such that
%$A$ and $B$ intersect $\sC_0$ transversally, and
%%$A$ and $B$ , and  $B\cap A=\emptyset$;\\
%\beq
%\sO_{\sC}(B+\sum l_iD_i)\cong \sO_{\sC}(5A+\sum 5m_iD_i)\otimes \omega^{\log}_{\sC/S}.
%\eeq
%%where $\omega^{\log}_{\sC/S}=\omega_{\sC}(\Si)\otimes\pi\sta\omega_S\dual$.

\begin{defi}\label{bask2}
We say a  basket $\cB$ final if it satisfies \\
%(i). $l_i\ge 0$; $m_i\ge 0$, and $l_im_i=0$ for all $i\in \Lam$;\\
%We say the basket $\cB$ is final if it is good and satisfies \\
(i). for every $i\in\Lam$, $B\cap D_i=\emptyset$ if $m_i\ne 0$, and $A\cap D_i=\emptyset$ if $l_i\ne 0$; \\
(ii). for distinct $i\ne j\in\Lam$ such that $l_im_j\ne 0$,  $D_i\cap D_j=\emptyset$.
\end{defi}

Let $(\sC, \sL, \sN,\rho,\nu)$ and $\sD$ be given by Proposition \ref{prop-extension1}.
Let $\{\sD_i\, |\, i\in\Lam\}$ be the set of irreducible components of $\sC_0$. We form
\beq\label{sample}
\sA=\overline{(\nu_{2*}=0)}, \ \sB=\overline{(\rho_*=0)},\
(\nu_2=0)=\sA+\sum m_i\sD_i,\ \sD=-\sum l_i\sD_i,
\eeq
 where the summations run over all $i\in\Lambda$.
By the construction, $\nu_2$ and $\rho$ induce isomorphisms
$\sN\cong \sO_\sC(\sA+\sum m_i \sD_i)$ and $\sO_\sC\cong \sL^{\vee 5}\otimes\omega_{\sC/S}^{\log}(\sD-\sB)$.
Using $\sL\dual\cong \sN$, we obtain an isomorphism 
\beq\label{basket2}
\sO_{\sC}(\sB+\sum l_i\sD_i)\cong \sO_{\sC}(5\sA+\sum 5m_i\sD_i)\otimes \omega^{\log}_{\sC/S}.
\eeq

Let $(\Si,C)$ be the good $S$-family of pre-stacky pointed nodal curves 
that is the un-stacking of $(\Si^\sC,\sC)$
as explained before Definition \ref{bask}. Let $D_i\sub C_0$ be the 
image of $\sD_i$. Since $\sC$ away from $(\nu_2=0)$ is a scheme,
and by the construction carried out in the proof of
Proposition \ref{prop-extension1}, we have
$\sB=\sum_{i=1}^{k_2} b_i \sB_i$, where $\sB_i$ are sections of $\sC\to S$ and $b_i\in \ZZ_{>0}$,
and for any $(i,j)$ either $\sB_i\cap \Si^\sC_j=\emptyset$ or $\sB_i=\Si^\sC_j$, and in the 
later case $\Si_j^\sC$ is 
a scheme.
We let $B_i\sub C$ be the image of $\sB_i$. %and let $B=\sum_{i=1}^{k_2} b_i B_i$.
For $\sA$, it can also be written as 
$\sA=\sum_{i=1}^{k_1} a_i \sA_i$, where $\sA_i$ are sections of $\sC\to S$.  
Let $A_i$ be the image of $\sA_i$. We form
\beq\label{AB}
A=\sum_{A_i \not\in \{\text{pre-stacky}\, \Si_j\}} a_i A_i+\sum_{A_i \in\{\text{pre-stacky}\, \Si_j\}} \frac{a_i}{5} A_i,
\and  B=\sum_{i=1}^{k_2} b_i B_i.
\eeq

\begin{lemm} \label{bask3}
Let $(\Si,C)$ be as before, let $\cB$ in \eqref{BAS}  be such that  the coefficients 
$l_i$ and $m_i$ are given
in \eqref{sample},  and let $A$ and $B$ be given in \eqref{AB}. Then $\cB$
 %=\bl B+\sum_{i\in\Lam} l_i D_i, A+\sum_{i\in \Lam} m_i D_i\br$
is a pre-basket.
\end{lemm}

\begin{proof}
That $\cB$ satisfies (1)-(4) in Definition \ref{bask} follows from the proof of Proposition \ref{prop-extension1}. 
For the isomorphism \eqref{basket*}, we notice that by our choice of $A$ and $B$, we have
$$\sO_{\sC}(\sB+\sum l_i\sD_i)\cong \sO_{C}(B+\sum l_i D_i)\otimes_{\sO_C}\sO_\sC
$$ and
$$\sO_{\sC}(5\sA+\sum 5m_i\sD_i)\otimes \omega^{\log}_{\sC/S}\cong
\bl\sO_{C}(5A+\sum 5m_iD_i)\otimes \omega^{\log}_{C/S}\br\otimes_{\sO_C}\sO_\sC.
$$
Therefore, \eqref{basket*} follows from \eqref{basket2}. This proves the Lemma.
\end{proof}

\subsection{Restacking}

In this subsection, we fix an $S$-family of  pre-stacky $n$-pointed nodal curves $(\Si,C)$
and a final basket $\cB$ on it in the notations of Definition \ref{bask} and \ref{bask2}.
Let $t$ be a uniformizing parameter of $R$, where $S=\spec R$, that is the pullback of 
the standard coordinate variable of $\Ao$ via the map $S\to\Ao$  specified at the beginning of 
\S \ref{Sub3.1}.  
Let $R_5=R[z]/(z^5-t)$, and let $S_5=\spec R_5$.

\begin{lemm}\label{5change}
Let $C$ be a flat $S$-family of nodal curves,
let $N$ be the singular points of the central fiber $C_0$, %)_{\text{sing}}$,
and let $M$ be an (integral) effective Cartier divisor on $C$ 
such that $M=5M_h+M_0$, where $M_0$ (resp. $M_h$) is an integral Weil divisor contained in $C_0$ 
(resp. none of its irreducible components lie in $C_0$).
%and $M_h$ is an integral (Weil) divisor without any components in $C_0$.\black  
Then there is a unique $S_5$-family of twisted curves $\ti\sC$ such that
\begin{enumerate}
\item let $\ti N\sub\ti\sC_0$ be the singular points of $\ti\sC_0$, then 
$\ti\sC-\ti N\cong (C- N)\times_SS_5$;
%and $\Si^{\ti\sC}=\Si\times_SS_5$ under this isomorphism, and $\Si^{\ti\sC}_i$ is
%$\mufive$-banded section if and only if $\Si_i$ is assigned stacky. \\
\item let $\ti\phi: \ti\sC-\ti N\to C$ be the morphism induced by (1), then 
$\ti\phi\sta(M)$ is divisible by $5$, and $\frac{1}{5}\ti\phi\sta(M)$ extends to a Cartier divisor on $\ti\sC$, denoted
by $\ti M_{\ofth}$;
\item %Among all pairs $(\ti\Si,\ti\sC)$ satisfying (1)-(3), there is a unique one such that
each $\zeta\in \ti N$ is either a scheme point or a $\mufive$-stacky point of $\ti\sC$, and the tautological map
$\Aut(\zeta)\to \Aut(\sO_{\ti\sC}(\ti M_{\ofth})|_{\zeta})$ is injective.
\end{enumerate}
\end{lemm}

\begin{proof}
%Let $\Sigma$ be the markings on $\sC$. Since $\sC\to S$ is a good family, $\sC-\Sigma$ is a scheme. Therefore $\sC$ is obtained from the coarse
%moduli $C=|\sC|$ by the standard root construction. So we can ignore the stacky structure on the markings in the following construction since we can recover the stacky structure on the markings in the end by the root construction. 
Let $p\in N$ be a singular point of $C_0$. 
%In case $p$ is a node of a single irreducible component of $C_0$,
%then the issue to be discussed does not occur. Now suppose $p$ is the 
%intersection of two irreducible components $D_1$ and $D_2$.
Pick an \'etale open
neighborhood $q\colon V\to C$ of $p\in C$ so that $V$ is an open subscheme
of $(xy=t^k)\sub \spec(R[x,y])$, as $S$-schemes. Let $D_1$ and $D_2$ in $V$  be
$(x= t=0)$ and $(y=t=0)$, respectively (Example 6.5.2 in \cite{Har}). When $k\neq 1$, $D_1$ and $D_2$ are 
Weil but not Cartier divisors. Write $q\upmo M=5A+n_1D_1+n_2D_2$, where $A=q\upmo M_h$ 
is an integral Weil divisor with 
no irreducible components contained in $D_1\cup D_2$. Let $\text{CL}(V)$ (resp. $\text{Car}(V)$) be the Weil 
(resp.  Cartier) divisor class groups of $V$ respectively.
It is known that $\text{CL}(V)/\text{Car}(V)=\mathbb Z/k\mathbb Z$, generated by $D_1$ (or $D_2$). 
Thus $A$ is linearly equivalent to $l D_1+B$ for an integer $l$ and a Cartier divisor $B$. 
Since $M$ is a Cartier divisor, we have $5l+n_1-n_2\equiv 0(k)$ (i.e. $\equiv 0\!\mod(k)$). 
Here we used the fact that $D_2= -D_1\in \text{CL}(V)/\text{Car}(V)$.

%Let $\text{g.c.d.}(5, k)$ be the greatest common divisor of $k$ and $5$,  and let $\ell=\text{g.c.d.}(5, k)/5$. If $5$ divides $k$, let $\ti C=C$, $\ti R =R$,  $z=t$, $\ti q\colon \ti C\to C$ be the identity map, and $\ti V=V$. In this case, $\ell =k/5$ and $\ti V=(xy=z^{5\ell})\sub \text{Spec}(\ti R[x, y])$. If $5 $ and $k$ are relatively prime, 
Consider the base change $\ti C\defeq C\times_SS_5\to C$.
%be a base change of $C$; $\ti C$ is normal but not smooth. 
Let $\ti p$ be the node in the central fiber of $\ti C$ corresponding to $p$.
Let $\ti V=V\times_S S_5$ and $\ti q: \ti V\to V$ be the projection. It is an \'etale neighborhood of $\ti p\in \ti C$, 
and is an open subsheme of $(xy=z^{5k})\sub \text{Spec}(R_5[x, y])$.  
%Let $\ti R=R_5$. Then we can write $\ti V=(xy=z^{5\ell})\sub \text{Spec}(\ti R[x, y])$.
 
Let $\ti A$ be $\ti q\upmo(A)$, $\ti D_1=(x= z=0)$, and $\ti D_2=(y=z=0)$. Since $\ti q\upmo(D_i)=5 \ti D_i$,
the pullback  $\ti q\upmo(M)=5\ti A+5n_1\ti D_1+5n_2\ti D_2$, and $\frac{1}{5} \ti q\upmo(M)$ away from $\ti p$ is Cartier.
Since $A=l D_1\in \text{CL}(V)/\text{Car}(V)$, $\ti A=l (5\ti D_1)\in \text{CL}(\ti V)/\text{Car}(\ti V)$.
Thus 
$$\ti q\upmo(M)=5\ti A+5n_1\ti D_1+5n_2\ti D_2\equiv (25l+5n_1-5n_2)\ti D_1\in \text{CL}(\ti V)/\text{Car}(\ti V).
$$
Thus $\frac{1}{5} \ti q\upmo(M)$ is Cartier only when $5l+n_1-n_2\equiv 0(5k)$. 
%, $q^*(M/5)$ cannot extend to a Cartier divisor near  $\ti p$. 

To make it Cartier when $5l+n_1-n_2\not\equiv 0(5k)$, we introduce $\mufive$-stacky structure at $\ti p$
as follows.
Consider $(uv=z^k)\sub {\rm Spec}( R_5[u, v])$ and the morphism $(uv=z^k)\to (xy=z^{5k})$
via $x\mapsto u^5$ and $y\mapsto v^5$. 
Let $U\sub (uv=z^k)$ be the open subscheme
mapped onto $\ti V$,
and let $\zeta\in \mufive$ act on $U$ via $\zeta\cdot (u, v)=(\zeta u, \zeta^{-1} v)$. Then
the quotient $U/\mufive\cong \ti V$.  Let $\phi: U\to V$ be the induced projection, let
$\bar D_1=(u=z=0)$ and $\bar D_2=(v=z=0)$. Then $\phi\upmo D_i=5\bar D_i$ and 
hence $\frac{1}{5} \phi\upmo(M)=\bar A+n_1\bar D_1+n_2\bar D_2$, where
$\bar A=\phi\upmo(M_h)$, and is $\mufive$-invariant. Since $A\equiv lD_1\in \text{CL}(V)/\text{Car}(V)$,
$\bar A=l\phi\upmo(D_1)=5l\bar D_1\in \text{CL}(U)/\text{Car}(U)$, and is $\mufive$-invariant.
Hence $\frac{1}{5} \phi\upmo(M)\equiv (5l+n_1-n_2)\bar D_1\equiv 0\in \text{CL}(U)/\text{Car}(U)$, 
thus is Cartier.

%extends to a $\mufive$-invariant integral Cartier divisor on $U$. To be precise, in $U$  the extension is  $\bar A+n_1\bar D_1+n_2\bar D_2$ where $\bar A$ is the pullback (Weil) divisor of $\ti A$ via the map $U\to \ti V$ which is $\mufive$-invariant. Note that $\bar D_i$ is a $5$-fold covering of $\ti D_i$. In $\text{CL}(U)$, the divisor  $\bar A+n_1\bar D_1-n_2\bar D_2$ is linearly equivalent to $l \bar D_1+n_1\bar D_1-n_2\bar D_2$. \black
Therefore, if $5l+n_1-n_2\equiv 0(5k)$, we do nothing at $\ti p$;
otherwise, we introduce a stacky structure at $\ti p$ by replacing a neighborhood of $\ti p\in \ti C$ by
the quotient stack $[U/\mufive]$ (cf. \cite{A-V}) and denote the resulting stack by $\ti\sC$. 
By repeating this over all $p\in N$, we obtain $\ti\phi:\ti\sC\to C$
such that $\ti M_{\ofth}=\frac{1}{5} \ti \phi\upmo(M)$ is an integral Cartier divisor
satisfying the requirements of the Lemma. 
%Finally, we let $\Si^{\ti\sC}_i\sub \ti\sC$
%be the section corresponding to $\Si_i\sub C$. We endow $\ti\sC$ a $\mufive$-stacky structure along
%$\Si^{\ti\sC}_i$ is $\Si_i$ is assigned stacky. Since $\ti\sC$ is smooth along $\Si^{\ti\sC}_i$, this is possible.
%The resulting $(\Si^{\ti\sC},\ti\sC)$ satisfies the requirement of the Lemma.
\end{proof}

\begin{coro}\label{5change1}
Let $C$ be a flat $S$-family of nodal curves, let $N$ be the singular points of $C_0$, 
and $v\in H^0(C-N, \cM)$ be a section of an invertible sheaf $\cM$ on $C-N$ so that 
$\cM^{\otimes 5}$ extends to an invertible sheaf
on $C$. 
Then there is a unique $S_5$-family of twisted nodal curves $\ti\sC$ such that
\begin{enumerate}
\item let $\ti N\sub\ti\sC_0$ be the singular points of $\ti\sC_0$, then 
$\ti\sC-\ti N\cong (C- N)\times_SS_5$; 
% and $\Si^{\ti\sC}=_{\mathrm{set}}\Si\times_SS_5$ under this isomorphism. 
\item 
there is an invertible sheaf $\ti\sM$ on $\ti\sC$ and a section $\ti v\in H^0(\ti\sM)$ so that,
letting $\ti\phi: \ti\sC-\ti N\to C$ be the morphism induced by (1), then 
$\ti\sM|_{\ti\sC-\ti N}\cong \ti\phi\sta\cM$, and $\ti v|_{\ti\sC-\ti N}=\ti\phi\sta v$;
%$q\sta(M)$ is divisible by $5$, and $\frac{1}{5}q\sta(M)$ extends to a Cartier divisor on $\ti\sC$, denoted
%by $\frac{1}{5} \ti M$.  
\item %Among all pairs $(\ti\Si,\ti\sC)$ satisfying (1)-(3), there is a unique one such that
each $p\in \ti N$ is either a scheme point or a $\mufive$-stacky point of $\ti\sC$, and
the tautological map $\Aut(\ti p)\to \Aut(\ti\sM|_{\ti p})$ is injective.
\end{enumerate}
\end{coro}

\begin{proof}
Since $C$ is normal, and $\cM^{\otimes 5}$ extends to an invertible sheaf on $C$, $v^5$ extends to a
regular section over $C$, thus $M=\overline{(v^5=0)}$ is a Cartier divisor on $C$. As $M|_{C-N}=5(v=0)$,
we can write $M=5M_h+M_0$, where $M_0$ is supported on $C_0$ and no irreducible components of
$M_h$ lie in $C_0$.

Let $\ti\phi: \ti\sC\to C$ be the $S_5$-family of twisted curves constructed in the previous Lemma for the
Cartier divisor $M=5M_h+M_0$,
and $\ti M_{\ofth}$ be the Cartier divisor so that $\ti\phi\upmo(M)=5\ti M_{\ofth}$.
Let $\ti\sM=\sO_{\ti\sC}( \ti M_{\ofth})$. Then $\ti\sM$ is invertible, with a tautological
section $\ti v\in H^0(\ti\sM)$ so that $(\ti v=0)=\ti M_{\ofth}$. Because 
$$(\ti v=0)\cap (\ti \sC-\ti N)=\bl(v=0)\times_S S_5\br\cap(\ti\sC-\ti N),
$$
we conclude that we have isomorphism $\ti\sM|_{\ti\sC-\ti N}\cong \ti\phi\sta \cM|_{\ti\sC-\ti N}$ 
so that $\ti v|_{\ti\sC-\ti N}=\ti\phi\sta v|_{\ti\sC-\ti N}$. This proves the corollary.
\end{proof}

Let $(\Si,C)$ be a good $S$-family of  pre-stacky pointed curves, and let $\cB$ be a final basket,
in the notation of Definition \ref{basket*} and \ref{basket2}. We shall provide a procedure to construct a
family in $\cW^{\text{pre}}\lggd(S)$.

We first restacking the pre-stacky
pointed curve $(\Si\times_S S_5,C\times_S S_5)$ to obtain an $S$-family of pointed twisted
curve $(\Si^{\sC}, \sC)$. Let $q: \sC-(\sC_0)_{\text{sing}}\to C$ be the projection.
Let $M=5A+\sum 5m_iD_i$, which is a Cartier divisor on $C$ so that $D_i$ are supported along $C_0$
and no irreducible component of $A$ lies in $C_0$.
%By our assumption on $A$, we see that 
Thus $\ofth q\sta M$ is a Cartier divisor. We then apply Lemma \ref{5change} and Corollary \ref{5change1}
to $(\Si^{ \sC},\sC)$ to obtain an $S$-family of pointed twisted curve $(\Si^{\ti \sC},\ti \sC)$ such that
it is isomorphic to $(\Si^\sC,\sC)$ away from the singular points of the central fiber,
and there is an invertible sheaf $\ti\sM$ on $\ti\sC$ with a section $\ti v$ so
that $\ti\sM$ is the extension of $\sO_{\ti\sC-(\ti\sC_0)_{\text{sing}}}( \ofth q^* M)$ 
and $\ti v$ is the extension of the tautological section of the latter.

Let $\ti\phi: \ti\sC\to C$ be the tautological morphism, 
let $\ti\sN=\ti\sM$, and let $\ti\nu_2= \ti v\in H^0(\ti\sN)=H^0(\ti\sM)$.
Let $\ti\nu_1: \ti\sL\cong \ti\sN\dual$. By \eqref{basket*}, we conclude that
$$\sO_{\ti\sC}(\ti\phi^*(B+\sum l_iD_i))\cong \ti\sL^{\vee 5}\otimes \omega^{\log}_{\ti\sC/S}.
$$
%(Here we view $q^*(B+\sum l_iD_i)$ as the minimal extension to a Cartier divisor of 
%$\sC$; since $\cB$ is final, it is effective.) We let
Let $\ti\rho\in\Gamma( \ti\sL^{\vee 5}\otimes \omega^{\log}_{\ti\sC/S})$ be induced by the above isomorphism
and the tautological inclusion $\sO_{\ti\sC}\sub \sO_{\ti\sC}(\ti\phi^*(B+\sum l_iD_i))$. 
Because $\rho\lsta$ vanishes along $\Si^{\sC\lsta}_{(1,\rho)}$, $\ti\rho$ lifts to a section in
$\Gamma( \ti\sL^{\vee 5}\otimes \omega^{\log}_{\ti\sC/S}(-\Si^{\ti\sC}_{(1,\rho)}))$. This proves

\begin{lemm}\label{bask-family}
Let  notations be as stated. Then $\ti\xi=(\Si^{\ti\sC},\ti\sC, \ti\sL,\ti\sN,\ti\varphi=0, \ti\rho,\ti\nu)$
constructed based on a final basket $\cB$
belongs to $\cW^{\text{pre}-}\lggd(S)$ for a choice of $(g,\gamma,\bd)$. 
\end{lemm}

%\begin{proof}
%\red
%The remainder case is when $\varphi\lsta=0$, $\rho\lsta\ne 0$ and $\nu_{2\ast}\ne 0$.
%In this case, we apply Proposition \ref{prop-extension1} to obtain an extension $\Si^\sC\sub \sC$, $\rho$, etc.,
%of $\Si^{\sC\lsta}\sub\sC\lsta$, $\rho\lsta$, etc. We then let $\Si\sub C$ be the associated good $S$-family
%of pre-stacky pointed nodal curves and let $\sB$ be the associated pre-basket as given before and by
%Lemma \ref{bask3} to obtain a
%basket $\sB$. We then apply the discussion in Subsection \ref{mods} to obtain a final basket
%$\sB=(B+\sum l_i D_i,A+\sum m_iD_i)$ on a good $S$-family of pre-stacky nodal curves, which we denote by
%$\Si\sub C$, over $S$.
%
%Note that $m_i\in \ofth\ZZ$. We let $M=\sO_C(\sum 5m_i D_i)$, which is an invertible sheaf. Applying
%Lemma \ref{5change}, after endowing necessarily stacky structure at $C\times_S S_5$ long the nodes of 
%its central fiber, say $\Si^{\ti\sC}\sub\ti \sC$ the resulting $S_5$-family with $q: \ti\sC\to C$ the projection,
%we obtain an invertible sheaf $\ti\sM$ on $\ti\sC$ such that
%$q\sta M|_{\ti\sC-\ti N}\cong \ti\sM^{\otimes 5}$. We let $\ti\sN=q\sta\sO_C(A)\otimes\ti\sM$,
%let $\ti\nu_{1}: \ti\sL\dual\cong \ti\sN$. Then the 
%\end{proof}

%As we will see, the fifth-power of $\sL$ in $\sL^{\vee 5}\otimes\omega_{\sC/S}^{\log}$
%forces us to work with $\ofth\ZZ$-divisors. This difficulty will be resolved by introducing $\mufive$-stacky points.

\subsection{Modifying brackets}\label{mods}

In the following subsections, 
we will perform a series of blowups, base changes and stablizations to 
obtain an extension in $\cW^-\lggd(S)$. 
Let $(\Si,C)$ be a good $S$-family of pre-stacky pointed nodal curves.

\begin{defi}\label{2.8}
Let $\cB$ be a pre-basket of $(\Si,C)$.
We say $\cB'$ is a modification of
$\cB$ if there is a finite base change $S'\to S$, a good $S'$-family $(\Si',C')$ 
of pre-stacky pointed curves so that $\cB'$ is a basket of $(\Si',C')$,
$(\Si',C')\times_{S'}S\lsta'\cong (\Si,C)\times_SS\lsta'$ as pre-stacky pointed nodal curves,
and under this isomorphism
$\cB'\times_{S'}S\lsta'=\cB\times_SS\lsta'$.
\end{defi}

We first show that we can find a basket $\cB'$ that is a modification of $\cB$ on $(\Si,C)$.
Indeed, let $r_i=\min(5m_i,l_i)$, and let
%Indeed, we let $l_i'=l_i+\min(5m_i,l_i)$ and $m_i'=\ofth\min(5m_i,l_i)$. Then
\beq\label{shift}
\cB'=(B+\sum (l_i-r_i)D_i, A+\sum (m_i- r_i/5)D_i).
\eeq
It is easy to see that $\cB'$ is a modification of $\cB$. % and a modification of $\cB$.
%Thus in the following, for any pre-basket we will replace it by its shift to obtain a basket.

In the following, we assume $\cB$ is a basket as in Definition \ref{2.8}.
We will construct modifications of the basket $\cB$
that will reduce the $\ZZ_{\ge 0}$-valued quantities
$$V_1(\cB)=\sum_{B\cap D_j\ne \emptyset} 5m_j ,\quad V_2(\cB)=\sum_{A\cap D_i\ne \emptyset} l_i,
\and V_3(\cB)=\sum_{D_i\cap D_j\ne \emptyset}5l_im_j.
$$

\begin{lemm}\label{construction1}
Let $(\Si,C)$ and $\cB$ be as stated.
Then there is a modification $\cB^\prime$ of $\cB$ such that
$V_1(\cB^\prime)=0$.
\end{lemm}

\begin{proof}
Let $\bar j$ be such that $m_{\bar j}>0$ and $p\in B\cap D_{\bar j}\neq \emptyset$. 
Since $\cB$ is a basket, $l_{\bar j}=0$.
By the definition of basket, $C_0$ is smooth at $p$ and $p\notin A$. 
%Note that $C$ is a scheme near $B$. This is because near $(\rho_*=0) $, $\nu_{2*}$ is nowhere vanishing. Since $\sN_*$ is representable, $C_*$ must have scheme structure near  $(\rho_*=0) $.
We let $\tau\colon \tilde C\to C$ be the blowup of $C$ at $p$,
let $E$ be the exceptional divisor, and let $\tilde\pi\colon \tilde C\to S$ be  the induced projection. 
In the following, for any Cartier divisor $G\sub C$,
we denote by $\tilde G$ its strict transform in $\ti C$.
Because $B$ is an integral divisor, $\tau^*B=\ti B+ l E$ with $1\le l \in \mathbb Z$. By the blowing up formula, we have
$\omega^{\log}_{\tilde C/S}=\tau^*\omega^{\log}_{C/S}(\epsilon E)$,
where $\epsilon=0$ (resp. $=1$) when $p\in\Si$ (resp. $p\not\in\Si$).
We give $\ti\Si$ the pre-stacky assignments according to that of $\Si$.
%Let $\ti\Si\sub \ti C$ be the proper-transform of $\Si\sub C$, with corresponding .
Form
$$\ti\cB=(\ti B+\sum l_i\ti D_i+(l+\epsilon)E,  \ti A+\sum m_j\ti D_j +m_{\bar j}E).
$$
We claim that it is a pre-basket of $(\ti\Si,\ti C)$. Indeed, the conditions (1)-(4) in the Definition \ref{bask}  
can be easily verified. It remains to verify the isomorphism (\ref{basket*}). Obviously we have
%\begin{eqnarray*}
$$\tau^*\sO_C(B+\sum l_iD_i)\cong\sO_{\ti C}(\ti B+\sum lE+\sum l_iD_i),
$$
$$
\tau^*\big((\sO_C(5A+\sum 5m_iD_i)\otimes \omega^{\log}_{C/S}\big)
\cong\sO_{\ti C}(5\ti A+\sum 5m_iD_i+5m_{\bar j}E)\otimes\tau^*\omega^{\log}_{C/S}, 
$$
and $\tau^*\omega^{\log}_{C/S}\cong \omega^{\log}_{\ti C/S}(-\epsilon E)$.
Combined, we get
\begin{eqnarray*}
\sO_{\ti C}(\ti B+\sum(l+\epsilon)E+\sum l_iD_i)\cong \sO_{\ti C}(5\ti A+\sum 5m_iD_i+5m_{\bar j}E)\otimes \omega^{\log}_{\ti C/S}.
\end{eqnarray*}
Thus $\ti\cB$ is a pre-basket. Then $(\ti \cB)'$ given in \eqref{shift} is a basket. 
By construction, it is a modification of $\cB$.

We check that $V_1((\ti \cB)')<V_1(\cB)$.
In fact, 
$$(\ti \cB)'=\bl \ti B+\sum l_i\ti D_i+(l+\epsilon-r)E,
\ti A+\sum m_j\ti D_j+(m_{\bar j}-\displaystyle\ofth r)E\br,
$$
where $1\le r=\min\{5m_{\bar j}, l+\epsilon\}\in \Bbb Z$ since $0\neq 5m_{\bar j}\in \mathbb Z$.  
Since $\ti B\cap \ti D_{\bar j}=\emptyset$ and $0\le m_{\bar j}-r/5<m_{\bar j}$, we have
$$
V_1((\ti \cB)')= \sum_{\ti B\cap \ti D_j\neq \emptyset, j\neq \bar j }5m_j+5(m_{\bar j}-r/5)
= \sum_{B\cap D_j\neq \emptyset} 5m_ j-r=V_1(\cB)-r< V_1(\cB).
$$
The lemma is proved by induction. %epeating this construction, we prove the Lemma. 
\end{proof}

\begin{lemm}\label{construction2}
Let $(\Si,C)$ and $\cB$ be as stated with  $V_1(\cB)=0$. 
Then there is a modification $\cB^\prime$ 
of $\cB$ such that
$V_1(\cB^\prime)=V_2(\cB^\prime)=0$.
\end{lemm}

\begin{proof}
Suppose there is an $l_{\bar i}>0$ such that $A\cap D_{\bar i}\neq \emptyset$. Since $\cB$ is a basket,
$m_{\bar i}=0$. 
Pick $p\in A\cap D_{\bar i}$.  Let $\tau\colon \ti C\to C$ be the blowup of $C$ at $p$. If $p$ lies on a  marking $\Si_i$, let  $\ti \Si_i\subset \ti C$ be the strict transform  of $\Si_i$.  By transversality, $\tau^*A=\ti A+ mE$ where $m\in \frac{1}{5}\mathbb Z$.
%$\rho^*D_i=\ti D_i$ for $i\neq j$, and $\rho^* D_{j}=\ti D_{j}+E$.
Consider 
$$\ti \cB=(\ti B+ \sum l_i\ti D_i+(l_{\bar i}+\epsilon)E, \ti A+\sum m_j\ti D_j+ mE),
$$
where $\epsilon =0$ or $1$ when $p\in \Si$ or $\not\in\Si$ respectively.
We have $\ti B\cap E=\emptyset$, $\ti A\cap \ti D_{\bar i}=\emptyset$, and $\ti A$ intersects $E$ transversally.
Like in the proof of the previous Lemma, it is direct to verify that $\ti \cB$ is a pre-basket. 
Like in \eqref{shift}, let %We replace $\ti\cB$ by its shift 
$$\cB'=\bl \ti B+\sum l_i\ti D_i+(l_{\bar i}+\epsilon-r)E, \ti A+\sum m_j\ti D_j+(m-\displaystyle\frac{r}{5})E\br,
$$
where $r=\min\{5m, l_{\bar i}+\epsilon\}\ge 1+\epsilon$.  Note that
when $p\not\in \Si$, $A$ is an integral Cartier divisor near $p$ and thus $m\ge 1$.  Thus
$0\le l_{\bar i}+\epsilon-r< l_{\bar  i}$. Thus $V_2(\cB')<V_2(\cB)$.
Repeating this construction, we prove the Lemma.
\end{proof}

\begin{lemm}\label{construction3}
Let $(\Si,C)$ and $\cB$ be as stated with  $V_1(\cB)=V_2(\cB)=0$. 
Then there is a final basket $\cB^\prime$ which is a
modification of $\cB$.
\end{lemm}

\begin{proof}
Suppose  there are pairs $D_{\bar i}\ne D_{\bar j}$ such that
$p\in D_{\bar i}\cap D_{\bar j}$, and $\ell_{\bar i}>0$ and $m_{\bar j}>0$. 
Take a base change $S_2\to S$, and let $C^\prime=S_2\times_S C$. 
Then near every node of the central fiber of $C^\prime$, % over a node of the central fiber of $C$, 
$C^\prime$ is locally of the form $xy=t^2$. Minimally resolve $C^\prime$ to get a smooth $\ti C$ with a 
$(-2)$-curve corresponding to each node.  Let $E$ be the $(-2)$-curve corresponding to the point $p$ mentioned earlier. 
Then $\omega_{\ti C/S_2}^{\log}=\ti\tau^*\omega_{C^\prime/S_2}^{\log}=\tau^*\omega_{C/S}^{\log}$, where 
$\ti\tau$ is the minimal resolution morphism $\ti C\to C^\prime$ and $\tau $ is the composition of $\ti \tau$ 
with the base change map $\tau^\prime\colon C^\prime\to C$. % from the base-change. 

Let $\cB=(B+\sum \ell_i D_i, A+\sum m_jD_j)$.  Since $A$ and $B$ do
not intersect $D_{\bar i}\cap D_{\bar j}$, we have
$\tau^* A=\ti A$ and $\tau^*B=\ti B$. From the base-change and the minimal resolution, we get a new pre-basket on $\ti C$: 
$(\ti B+\sum \ell_i\ti D_i+\ell_{\bar i}E,
\ti A+\sum m_j \ti D_j+m_{\bar j}E)$. Then let %we can construct a new good basket $\ti B$ on $\ti C$:
$$\ti \cB=(\ti B+\sum \ell_i\ti D_i +(\ell_{\bar i}-r)E, \ti A+\sum m_j \ti D_j+(m_{\bar j}-r/5)E),
$$ 
where $r=\min\{5m_{\bar j}, \ell_{\bar i}\}$ as in \eqref{shift}.  It is a basket.
Furthermore, we have $m_{\bar j}-r/5<m_{\bar j}$ and $\ell_{\bar i}-r< \ell_{\bar i}$, 
$E$ intersects $\ti D_{\bar i}$ and $\ti D_{\bar j}$ at the nodes, and $\ti D_{\bar i}\cap \ti D_{\bar j}=\emptyset$. It is clear that
$V_3(\ti\cB)< V_3(\cB)$. Also note that the central fiber of $\ti C$ is reduced. 
Repeating this procedure, we prove the Lemma. 
\end{proof}

\subsection{Existence of extensions} In this subsection, we prove

\begin{prop}\label{exist-0}
Let $\xi\lsta\in\cW^-\lggd(S\lsta)$ be such that $\sC\lsta$ is smooth. Then possibly after a finite base change
of $S$, $\xi\lsta$ extends to a $\xi\in \cW^{\mathrm{pre}-}\lggd(S)$.
\end{prop}

\begin{proof}
Let $\xi\lsta=(\Si^{\sC\lsta},\sC\lsta,\cdots)$. We distinguish several cases. The case $\rho\lsta=0$ 
is proved in Proposition \ref{prop-extension0}. The next case is when $\varphi\lsta=\nu_{2\ast}=0$. 
In this case $(\rho\lsta=0)=\emptyset$ and $\sL\lsta\dual\cong\sN\lsta$. Proceed as in the proof of Proposition
\ref{prop-extension1}, we extend $\Si^{\sC\lsta}\sub \sC\lsta$ to  $\Si^{\sC}\sub \sC$ such that $(\mathbf{a1})$ 
holds, extend $\sL\lsta\dual\cong\sN\lsta$ to invertible sheaves $\sL^{\vee}\cong\sN$,
and extend $\rho\lsta$ to $\rho\in H^0(\sL^{\vee\otimes 5}\otimes\omega_{\sC/C}^{\log}(\sD))$ 
for an effective divisor $\sD$ supported along the central fiber $\sC_0$.

We then apply Lemma \ref{5change} to endow $\sC\times_S S_5$ necessarily stacky structure along the
nodes of its central fiber so that the followings hold: denoting $\ti\sC$ the resulting $S_5$-family of twisted curves,
$q: \ti\sC\to \sC$ the projection, and $\ti N\sub \sC$ the set of nodal points of $\ti\sC_0$, then
$q\sta\sO_\sC(\sD)|_{\ti\sC-\ti N}\cong \ti\sM^{\otimes 5}$ for an invertible sheaf $\ti\sM$ on $\ti\sC$.
Let $\ti\sL=q\sta\sL\otimes\ti\sM\upmo$, $\ti\nu_1: \ti\sL\dual\cong \ti\sN$ and $\ti\nu_2=\ti\varphi=0$. Then $q\sta\rho$ 
provides a nowhere vanishing section $\ti\rho\in H^0(\ti\sL^{\vee\otimes5}\otimes \omega_{\ti\sC/C}^{\log})$.
Let $\Si^{\ti\sC}$ be the pullback of $\Si^\sC$. Then 
$\ti\xi\defeq (\Si^{\ti\sC},\ti\sC, \ti\sL, \ti\sN, \ti\rho, \ti\varphi, \ti \nu)\in\cW^{\mathrm{pre}-}\lggd(S_5)$
is a desired extension.

The last case is when $\varphi\lsta=0$, $\rho\lsta\ne 0$ and $\nu_{2\ast}\ne 0$. This case is proved by
the combination of Proposition \ref{prop-extension1}, Lemma \ref{bask3}, the finalization of baskets in \S
\ref{mods} and the restacking Lemma \ref{bask-family}.
\end{proof}

%In stabilizing the family $\xi$ constructed, we need the following contraction Lemma.
%
%\begin{lemm}\label{contract}\blue
%Let $\xi\in \cW^{\mathrm{pre}}\lggd(S)$ be so that $\sC$ is smooth, $\sC_0$ contains a connected chain of
%rational curves $\sD_1\cup\cdots\cup\sD_k$ so that $\Si^\sC\cap(\sD_1\cup\cdots\cup\sD_k)=\emptyset$,
%and $\nu_2|_{\sD_1\cup\cdots\cup\sD_k}=0$. %, and $\rho|_{\sD_1\cup\cdots\cup\sD_k}$ is nowhere vanishing.
%Then we can find %can contract $\xi$ along $\sC_1\cup\sCots\cup\sC_k$ to obtain 
%$\ti\xi=(\Si^{\ti\sC},\ti\sC, \ti\sL, \cdots)\in \cW^{\mathrm{pre}}\lggd(S)$ of which the following hold:
%there is a morphism $\phi: \sC\to\ti\sC$ so that $\phi(\sD_1\cup\cdots\cup\sD_k)=\ti p\in\ti\sC$ is a single
%point, $\phi: \sC-\sD_1\cup\cdots\cup\sD_k\to \ti\sC-\ti p$ is an isomorphism, 
%and $\phi\sta(\Si^\sC, \sL,\sN,\varphi, \rho,\nu)=(\Si^{\ti\sC}, \ti\sL,\ti\sN,\ti\varphi,\ti\rho, \ti\nu)$.
%\end{lemm}
%
%\begin{proof}
%{\red to be filled}.
%\end{proof}

\subsection{Stabilization}
\def\upre{^{\mathrm{pre}}}

Let 
%$(\Si,C)$ be a good $S$-family of $n$ pre-stacky pointed nodal curves and let $\cB$ be a final basket on $(\Si,C)$. We let 
%$\xi=(\Si^\sC, \sC, \sL, \sN,\nu, \rho)\in 
$\xi\in \cW\upre\lggd(S)$ be such that
%family given by Lemma \ref{bask-family} based on the basket $\cB$. 
$\xi\lsta=\xi\times_S S\lsta\in \cW^-\lggd(S\lsta)$. We will show 
how to modify $\xi$ along $\sC_0$ to obtain a new family $\xi'\in \cW^-\lggd(S)$, possibly after a finite base change,
such that $\xi\lsta\cong\xi'\lsta \in \cW^-\lggd(S\lsta)$. 
\vsp

%\begin{proof}
%If $s$ is nowhere vanishing, then $L\cong \sO_\Po$; we are done. Otherwise $p\not\in 
%s\upmo(0)\ne \emptyset$, and is invariant under $G$. As it is finite, after replacing $G$ by a finite
%index subgroup, we can assume that $G$ fixes $s\upmo(0)$. Since $G$ also fixes $p$, $G$ is 
%infinite implies that $G\cong \CC\sta$ and $s\upmo(0)=p'$ is a single point set. A direct check shows 
%that this is impossible. This shows that $s$ must be nowhere vanishing, and the Lemma is proved.
%\end{proof}
\begin{lemm}\label{case-irre}
Let $\xi\in \cW\upre\lggd(S)$ be such that $\xi\lsta\in\cW^-\lggd(S\lsta)$. Suppose the central fiber $\sC_0$ is irreducible, then
$\xi_0\in \cW^-\lggd(\eta_0)$.
\end{lemm}

\begin{proof} Suppose $\xi_0$ is unstable. Since $\sC_0$ is irreducible, by Lemma \ref{stable-cri}, either
$\sC_0$ is a smooth rational curve satisfying one of (1)-(3) in Lemma \ref{stable-cri}, or
$\sC_0$ satisfies (4) of the same Lemma.

In case $\sC_0$ is a smooth rational curve and satisfies one of (1)-(3) mentioned, since the properties
in (1)-(3) are deformation invariant, for general $s\in S\lsta$, $\sC_s$  satisfies the same property,
forcing $\xi_s$ unstable, contradicting to that $\xi\lsta$ is a family of stable objects. 

Therefore, $\sC_0$ must be the Case (4) in Lemma \ref{stable-cri}.
Therefore,  for a closed point $s\in S\lsta$, 
$\Si^{\sC_s}=\emptyset$, and $\deg\sL_s=\deg \sN_s=0$. Here $\Si^{\sC_s}=\Si^{\sC}\cap \sC_s$, ect.
By the non-vanishing assumption on 
$(\rho, \nu_2)$, $(\varphi,\nu_1)$ and $(\nu_1,\nu_2)$, as in the proof of
Lemma \ref{stable-cri}, when $\rho|_{\sC_s}\ne 0$, we conclude that $\sL^{\otimes 5}|_{\sC_s}\cong\sO_{\sC_s}$ and 
$\sL|_{\sC_s}\dual\cong\sN|_{\sC_s}$;
when $\varphi|_{\sC_s}\ne0$, we conclude that $\sL|_{\sC_s}\cong\sN|_{\sC_s}\cong\sO_{\sC_s}$. 
Therefore, $\xi_s$  for general $s\in S\lsta$ belongs to Case (4) of Lemma \ref{stable-cri}, thus must be unstable.
This proves the Lemma.
\end{proof}

We prove the desired  existence   when $\sC_0$ is reducible.

\begin{prop}\label{proper1}
Let $\xi\in\cW^{\mathrm{pre}-}\lggd(S)$ be such that $\xi\lsta\in \cW\lggd(S\lsta)$ and
%$\xi_0\in \cW^-\lggd(\eta_0)$ and
$\sC\lsta$ is smooth. Then possibly after a finite base change,
we can find a $\xi'\in\cW^-\lggd(S)$ such that $\xi\lsta\cong\xi\lsta'$. % and $\xi'_0\in \cW^-\lggd(\eta_0)$.
\end{prop}

\begin{proof}
Suppose $\xi_0$ is not stable, and $\sC_0$ is irreducible, by Lemma \ref{case-irre}, we are done.
In case $\sC_0$ is reducible, then by Lemma \ref{stable-cri}, we can find a rational curve  $\sE\sub \sC$ so that
either (1) or (2) of Lemma \ref{stable-cri} holds.

Let $\sE\sub\sC$ be of Case (1). We divide it into two subcases: Case-(1a) when both $\sE\cap \Si^\sC$ and $\sE\cap \sC_{0,\text{sing}}$ consist
of one point, and Case-(1b) when $\sE\cap \sC_{0,\text{sing}}$ consists of two points.
We look at Case-(1a). Let $p\in\sE$ be the node of
$\sC_0$ that lies in $\sE$. 
%We distinguish two cases. The first is when $\sC$ is a scheme near $p$. Then 
%$\deg\sL|_\sE=0$ implies that $\sL|_\sE\cong \sO_\sE$. 
Let $(E\sub C, \Si)$ be the coarse moduli of $(\sE\sub \sC,\Si^\sC)$. Then $E$ is a (-1)-curve of $C$. Let
$q:C\to C'$ be the contraction of $E$,  let $p'=q(E)\in C'$ and $\Si'=q(\Si)$.
We then introduce the stacky structure along $\Si'$ and $(C_0')_{\text{sing}}$ to obtain a family of 
twisted curves $\sC'\to S$ so that $q$ introduces an
isomorphism $\phi: \sC'-p'\cong \sC-\sE$. 

We claim that $\phi\sta(\sL,\sN,\rho,\varphi,\nu)$ extends
to   $(\sL',\sN',\rho',\varphi',\nu')$ on $\sC'$. Indeed, since $\sC'$ is smooth at $p'$, $\phi\sta\sL$
and $\phi\sta\sN$ extends to invertible sheaves $\sL'$ and $\sN'$ on $\sC'$, respectively, and the sections
$\phi\sta\rho$, $\phi\sta\varphi$ and $\phi\sta\nu$ extends to regular sections $\rho'$, $\varphi'$ and
$\nu'$. 

We now check that $(\rho', \nu'_2)$ is nonzero at $p'$. Indeed, 
since $\deg \sL|_\sE=0$ and $\omega_{\sC/S}^{\log}|_\sE\cong \sO_\sE$, if $\rho|_\sE\ne 0$
then it is nowhere vanishing. By (4) of Definition \ref{def-curve}, at least one of
$\rho|_\sE$ and $\nu_2|_\sE$ is nontrivial, thus at least 
one of $\rho|_\sE$ and $\nu_2|_\sE$ is nowhere vanishing. Consequently, at least one of
$\rho'(p')$ or $\nu'_2(p')$ is nonzero.
By the same reason, we conclude that both $(\varphi',\nu'_1)|_{p'}$ and $(\nu_1',\nu_2')|_{p'}$ are nonzero.
This concludes that $\xi'=(\Si^{\sC'},\sC', \sL',\cdots)\in \cW^{\text{pre}}\lggd(S)$. Finally, since ${\xi}\in
\cW^{\text{pre}-}\lggd(S)$, we have ${\xi'}\in \cW^{\text{pre}-}\lggd(S)$. 

We next consider then $\sE\sub\sC$ is of Case (2). As in Case-(1a), we form the 
coarse moduli $(E\sub C,\Si)$ of $(\sE\sub \sC, \Si^\sC)$, contract $E$ to obtain $C\to C'$, and
then reintroduce the stacky structure on $C'$ to get $\sC'$ so that $\Si^{\sC'}\sub \sC'-p'$ is isomorphic to
$\Si^\sC\sub \sC-\sE$. Next, we extend the pullback of $\sL$, etc. from $\sC'-p'$ to $\sC'$, and show that
the extensions give a new family $\xi'\in \cW^{\text{pre}-}\lggd(S)$. 

The new family $\xi'$ has the property that $\xi\lsta\cong\xi'\lsta$, $\sC'$ is smooth, and the number of rational curves
$\sE'$ in $\sC'$ listed as Case-(1a) or (2) is one less than that in $\sC$.
Therefore, after iteration, we can find a $\xi'\in \cW^{\text{pre}-}\lggd(S)$
so that $\xi\lsta=\xi'\lsta$, $\sC'$ is smooth, and no rational curve
$\sE'$ in $\sC'$ belongs to Case-(1a) or (2) in Lemma \ref{stable-cri}.

Therefore, to prove the Lemma, we only need to consider the case where the rational $\sE\sub\sC$ 
that makes $\xi_0$ unstable  characterized by Lemma \ref{stable-cri} are all in Case-(1b). 
Let $\sD=\sE_1\cup\cdots\cup\sE_k\sub\sC$ be
a maximal connected chain of Case-(1b) rational curves for the family $\xi$. As before,
  let $(D\sub C,\Si)$ be the coarse moduli of $(\sD\sub \sC, \Si^\sC)$. Then $D\sub C$ is a connected
chain of (-2)-curves. Let $q: C\to C'$ be the contraction of $D$,  $p'=q(D)\in C'$ and $\Si'=q(\Si)$. Note that
$p'\cap \Si'=\emptyset$. 

We distinguish two cases. The first is when $\sL|_\sD\cong \sO_\sD$.
We claim that then $\sN|_\sD\cong \sO_\sD$ as well.
Indeed, since $\xi_0\in \cW^{\text{pre}-}\lggd(\eta_0)$,
when $\rho|_\sD\ne 0$, then $\varphi|_\sD=0$, which forces $\nu_1|_\sD$ nowhere vanishing.
Thus $\sN|_\sD\cong \sL\dual|_\sD\cong \sO_\sD$. When $\rho|_\sD=0$, then $\nu_2|_\sD$
is nowhere vanishing, implying that $\sN|_\sD\cong \sO_\sD$. Note that by (2) of Definition \ref{def-curve},
$\sC$ is a scheme along $\sD$.

We then introduce stacky structure on $C'$ along $\Si'$ and $(C_0')_{\text{sing}}$
to obtain a family of twisted curves $\sC'$
so that the contraction morphism $q: C\to C'$ induces a contraction morphism
$\psi: \sC\to \sC'$ so that $\psi|_{\sC-\sD}: \sC-\sD\to\sC'-p'$ is an isomorphism of pointed 
twisted curves,
and that $p'$ is a scheme point of $\sC'$.
Let $(\sL',\sN', \varphi',\rho',\nu')=\psi\lsta(\sL,\sN,\varphi,\rho,\nu)$. It is direct to check that
$\xi'=(\Si^{\sC'},\sC', \sL',\cdots)\in \cW^{\text{pre}-}\lggd(S)$ and satisfies $\xi\lsta\cong\xi'\lsta$.

The other case is when $\sL|_\sC\not\cong \sO_\sD$. In this case, since $\deg\sL|_{\sE_i}=0$ for all
$\sE_i\sub \sD$, we   have $\varphi|_\sD=0$. Thus $\rho|_\sD$ and $\nu_1|_\sD$ are nowhere vanishing,
implying that $\sN|_\sD\cong \sL\dual|_\sD$.

%We continue to denote by $q: C\to C'$ the contraction morphism. 
To proceed, we introduce stacky structures on $C'$ along $\Si'$ and $(C_0')_{\text{sing}}-p'$
to obtain a family of twisted curves $\sC'$
so that the contraction morphism $q: C\to C'$ induces an isomorphism
$\phi: \sC'-p'\cong \sC-\sD$. (For the moment we keep $p'$ a scheme point of $\sC'$.)

Let $(\bar\sL,\bar\sN,\bar\rho,\bar\varphi,\bar\nu)$ be the pullback of $(\sL,\sN,\rho,\varphi,\nu)$ via $\phi$.
%Because $\xi_0\in \cW^{\text{pre}-}\lggd(\eta_0)$, we have $\nu_1|_\sD$ is nowhere vanishing.
%Thus near $p'$, $\sL'$ is isomorphic to $\sN^{\prime\vee}$, and $\nu_1'$ is non-vanishing near $\sD$.
%Since $\rho|_\sD$ is nowhere vanishing, and since $\omega_{\sC/S}^{\log}|_{\sD}\cong \sO_\sD$,
Since $\sL^{\otimes 5}|_\sD\cong \sO_\sD$, $\bar\sL^{\otimes 5}$ extends to
an invertible sheaf on $\sC'$. As $\bar \sL$ is isomorphic to $\bar \sN^{\vee}$ near $p'$, $\bar\sN^{\otimes 5}$ extends
to an invertible sheaf on $\sC'$.

We now consider $\bar\nu_2\in H^2(\sC'-p',\bar\sN)$. Since $\bar\sN^{\otimes 5}$ extends
to an invertible sheaf on $\sC'$, we can apply Corollary \ref{5change1} to $(\bar\nu_2,\bar\sN)$ to  
introduce a $\mufive$ stacky structure at $p'\in\sC'$ if necessary. After a finite base change, we continue to denote by $\sC'$ the resulting family of twisted curves. Then $\bar\sN$ extends to an invertible
sheaf $\sN'$ on $\sC'$ so that $\bar\nu_2$ extends to a regular section $\nu_2'$ of
$\sN'$. Since $\bar\sL^{\vee}$ is isomorphic to $\bar\sN$ near $p'$, we extend $\bar\sL$ to an
invertible sheaf $\sL'$ on $\sC'$ so that $\sL^{\prime\vee}$ is isomorphic to $\sN'$ near $p'$,
and extend the known isomorphism between $\bar\sL\dual$ and $\bar\sN$.
Because $\sC'$ is normal near $p'$, we can extend $\bar\varphi$, $\bar\rho$ and $\bar\nu_1$
to regular sections $\varphi'$, $\rho'$ and $\nu_1'$ over $\sC'$. 

We claim that $\xi'=(\Si^{\sC'},\sC', \sL',\cdots)\in \cW^{\text{pre}-}\lggd(S)$. For this, we only need to check
that $(\varphi',\nu_1')|_{p'}$, $(\rho',\nu_2')|_{p'}$ and $(\nu_1',\nu_2')|_{p'}$ are nonzero.
Indeed, since $(\rho=0)$ and $(\nu_1=0)$ are disjoint from $\sD$, the closures of $(\rho'=0)-p'$ and $(\nu_1'=0)-p'$ do 
not contains $p'$. Since $(\rho'=0)$ and $(\nu_1'=0)$ are pure codimension one closed subsets of $\sC'$,
we conclude that $\rho'$ and $\nu_1'$ are nonzero at $p'$. This proves that 
$(\varphi',\nu_1')|_{p'}$, $(\rho',\nu_2')|_{p'}$ and $(\nu_1',\nu_2')|_{p'}$ are nonzero.
Therefore, $\xi'\in \cW^{\text{pre}-}\lggd(S)$ and satisfies $\xi\lsta=\xi'\lsta$. 

We repeat this contraction and introducing stacky structures at the contracted point if necessary
to all connected chains of (-2) curves $\sE\sub \sC_0$ as argued, possibly after 
a finite base change, we obtain a family $\xi'\in \cW^{\text{pre}-}\lggd(S)$ so that
$\xi\lsta\cong \xi'\lsta$, and that there are no rational curves $\sE\sub \sC'_0$ belongs to the 
list in Lemma \ref{stable-cri}. This shows that $\xi'\in \cW^{-}\lggd(S)$ is a desired extension. 
\end{proof}

\subsection{Proof of properness}
\def\lalpsta{_{\alpha\ast}}

We prove the properness by gluing the extensions constructed in the previous subsections,
using the construction in \cite[Appendix]{A-G-V} (cf. \cite[Def. 1.4.1]{AF}).

Let $ \sX$ be an $S$-family of not necessary connected twisted nodal curves with two markings $\Gamma_1$ and $\Gamma_2$ which are $\bmu_5$-gerbes over $S$, $X$ be the moduli of $\sX$  with the natural projection $\pi\colon \sX\to X$,  $s_1, s_2\colon S\to X$ be two sections such that $s_i(S)=\pi(\Gamma_i)$. The line bundle $N_{\Gamma_i/\sX}^{\otimes 5}$ descends to the normal bundle of   $\pi(\Gamma_i)$ in $X$.

\begin{lemm}[{\cite[Def. 1.4.1]{AF}}]\label{glue2} With notations and assumptions as above. Assume 
 $s_1\sta N_{\Gamma_1/\sX}^{\otimes 5}\cong
s_2\sta N_{\Gamma_2/\sX}^{\vee\otimes 5}$.
Then possibly after a finite base change, we can find
an $S$-family of not necessary connected twisted nodal curves $\sX'$ together with an 
$S$-morphism $\alpha: \sX\to \sX'$ that is the gluing of $ \sX$ via (an appropriate $S$-isomorphism)
$\Gamma_1\cong\Gamma_2$. 
\end{lemm}

\begin{proof} %Without loss of generality, we can assume that $S$ is connected.
Possibly after a finite base change, we can find an $S$-isomorphism
$\gamma: \Gamma_1\to\Gamma_2$ and 
$N_{\Gamma_1/\sX}\otimes \gamma\sta N_{\Gamma_2/\sX}\cong\sO_{\Gamma_1}$
that induces the given isomorphism $s_1\sta N_{\Gamma_1/\sX}^{\otimes 5}\cong
s_2\sta N_{\Gamma_2/\sX}^{\vee\otimes 5}$. 
In case $\Gamma_1$ and $\Gamma_2$ lie in different connected components of $\sX$, the
gluing is given in
\cite[Definition 1.4.1]{AF}. The case $\Gamma_1$ and $\Gamma_2$ lie in the same connected
component of $\sX$ can be deduced by adopting the construction in the Appendix of \cite{A-G-V}.
%from the previous case as follows ({cf. \cite{}}).
%
%We first glue $\sX-\Gamma_2$ with $\sX-\Gamma_1$ along $\Gamma_1\sub \sX-\Gamma_2$
%and $\Gamma_2\sub \sX-\Gamma_1$. We denote the resulting stack by $\sY$. Let
%$\Gamma\sub \sY$ be the image of $\Gamma_1$ (and $\Gamma_2$).
%We then patch $\sY$ with $\sY$ (two copies) via the map $\sY-\Gamma\to\sY-\Gamma$ that send the 
%$\sX-\Gamma_1\cup\Gamma_2$ in the first copy $\sX-\Gamma_2\sub \sY$ to 
%the $\sX-\Gamma_1\cup\Gamma_2$
%in the second copy $\sX-\Gamma_2\sub\sY$ via the identity map $\id_{\sX-\Gamma_1\cup\Gamma_2}$.
%We denote the resulting stack by $\sZ$. Then the $\id_{\sX-\Gamma_1\cup\Gamma_2}$
%sending $\sX-\Gamma_1\sub\sY$ to $\sX-\Gamma_2\sub\sY$ extends to a $\ZZ_2$ action on
%$\sZ$. The desired gluing is $\sZ/\ZZ_2$. {\red I think this is done in the Appendix of [A-G-V].}
\end{proof}

We can also glue the sheaves and sections. 
Let the situation be as in Lemma \ref{glue2}, and let $\gamma:\Gamma_1\to\Gamma_2$ be the
isomorphism given in its proof. 

\begin{coro}\label{glue3}
Suppose we have an invertible
sheaf $\sL$ on $\sX$ and an isomorphism $\gamma\sta(\sL|_{\Gamma_2})\cong \sL|_{\Gamma_1}$.
Then the sheaf $\sL$ glues to get an invertible sheaf $\sL'$ on $\sX'$ via the exact sequence
$$0\lra \sL'\lra\alpha\lsta\sL\mapright{\epsilon} \alpha\lsta(\sL|_{\Gamma_1})\lra 0,
$$
where the arrow $\epsilon$ takes the form
$$\big(\alpha\lsta\sL\big)|_{\Gamma}\cong \alpha\lsta(\sL|_{\Gamma_1})\oplus\alpha\lsta(\sL|_{\Gamma_2})
\mapright{(\epsilon_1,-\epsilon_2)} \alpha\lsta(\sL|_{\Gamma_1}),
$$
where $\epsilon_1$ is the identity $\alpha\lsta(\sL|_{\Gamma_1})=\alpha\lsta(\sL|_{\Gamma_1})$,
and $\epsilon_2$ is the isomorphism
$\alpha\lsta(\sL|_{\Gamma_2})\cong \alpha\lsta(\sL|_{\Gamma_1})$ induced by the isomorphism
$\gamma\sta(\sL|_{\Gamma_2})\cong\sL|_{\Gamma_1}$ given.
Furthermore, suppose $s\in H^0(\alpha\lsta\sL)$ is a section so that $\epsilon(s)=0$, then $s$ lifts to a section
$s'\in H^0(\sL')$.
\end{coro}

\begin{prop}\label{proper-proof}
The degeneracy locus $\cW\lggd^-$ is a proper, closed substack of $\cW\lggd$.
\end{prop}

\begin{proof}
The closedness follows from the criterion \eqref{deg-loci}. We now prove the properness.

%The properness follows from repeated applying
%Proposition \ref{glue} to the family $(\Si^{\ti\sC},\ti\sC, \ti\sL,\ti\sN,\ti\varphi,\ti\rho,\ti\nu)$ constructed
%before Proposition \ref{glue}.
%
%We now apply this lemma and corollary.
Let $\xi_*=(\Si^{\sC\lsta}, \sC_*, \cdots)\in\cW^-\lggd(S\lsta)$.
Possibly after a finite base change, we can assume that every connected component of the singular locus 
$(\sC\lsta)_{\text{sing}}$ is the image of a section of $\sC\lsta\to S\lsta$. 
Let 
$$\pi: \sC\lsta^{\text{nor}}=\coprod \sC_{\alpha\ast}\to\sC\lsta
$$
be the normalization of $\sC\lsta$, with every
$\sC_{\alpha\ast}$ connected. After a finite base change, we can assume that $\sC_{\alpha\ast}\to S\lsta$
have connected fibers.
Let $\tau\lalpsta: \sC\lalpsta\to\sC\lsta$ be the tautological morphism. For each $\sC\lalpsta$,
we endow it with the markings the (disjoint) union of $\tau\lalpsta\upmo(\Si^{\sC\lsta})$
and $\tau\lalpsta\upmo((\sC\lsta)_{\text{sing}})$. Let $\sL\lalpsta$, etc., be the
pullbacks of $\sL\lsta$, etc., via $\tau\lalpsta$. By Corollary \ref{split},
$$\xi\lalpsta=(\Si^{\sC\lalpsta}, \sC\lalpsta,\sL\lalpsta,\sN\lalpsta,\varphi\lalpsta,\rho\lalpsta, \nu\lalpsta)\in
\cW^-_{g\lalp, n\lalp, \bd\lalp}(S\lsta),
$$
for a choice of $(g\lalp, n\lalp, \bd\lalp)$.

Applying Proposition \ref{proper1}, after a finite base change $S\lalp\to S$, we can extend $\xi\lalpsta$ to a
$\xi\lalp'\in \cW^-_{g\lalp, n\lalp, \bd\lalp}(S\lalp)$. We then pick a finite base change $\ti S\to S$, factoring
through all $S\lalp\to S$, and form $\xi\lalp=\xi\lalp'\times_{S\lalp} \ti S$. 
Therefore, after denoting $\ti S$ by $S$, we conclude that 
possibly after a finite base change, every $\xi\lalpsta$ extends to a 
$\xi\lalp\in \cW^-_{g\lalp, n\lalp, \bd\lalp}(S)$.

We now glue $\xi\lalp$'s to a $\xi$ that is a stable extension of $\xi\lsta$. 
Let $\ti\sC=\coprod \sC_{\alpha}$,
$\Upsilon\lsta\sub \sC\lsta$ be
a section of $(\sC\lsta)_{\text{sing}}$, and $\Upsilon_{1\lsta}\coprod \Upsilon_{2\lsta}\sub \sC\lsta^{\text{nor}}$
be the preimage of $\Upsilon\lsta$. Using $(\sC\lsta)^{\text{nor}}\to\sC\lsta$, they are markings in $\ti\sC\lsta$.
Since markings in $\sC_{\alpha\ast}$ extend to markings in $\sC\lalp$, after a finite base change, we can
assume that all $\Upsilon_{i\ast}$ extend to sections $\Upsilon_i$ in $\ti\sC$
such that the $S\lsta$-isomorphisms
$\Upsilon_{1\lsta}\cong \Upsilon\lsta\cong \Upsilon_{2\lsta}$ extend to an $S$-isomorphisms
$\sigma: \Upsilon_1\cong \Upsilon_2$. 

Then possibly after a finite base change, we can find an isomorphism 
$\sigma\sta N_{\Upsilon_2/\ti\sC}\otimes N_{\Upsilon_1/\ti\sC}\cong\sO_{\Upsilon_1}$ whose
restriction to $\Upsilon_{1\ast}$ is consistent with 
$\cal E xt^1(\Omega_{\sC\lsta},\sO_{\sC\lsta})\cong\sO_{\sC\lsta}$. Applying Lemma \ref{glue2},
we obtain a gluing $\sC$ of $\ti\sC$ along $\Upsilon_1\cong \Upsilon_2$, resulting a family of
twisted pointed nodal curves. After performing such gluing to all sections of $(\sC\lsta)_{\text{sing}}$,
we obtain an $S$-family of twisted nodal curves $\sC\to S$ that is an extension of $\sC\lsta\to S\lsta$.
We denote by
\beq\label{glue4}
\beta: \ti\sC\lra \sC
\eeq
the gluing morphism.

We next glue the sheaves and fields. Let $(\ti\sL,\ti\sN,\ti\varphi,\ti\rho,\ti\nu)$ be the sheaves and sections on 
$\ti\sC$ so that its restriction to $\sC\lalp$ is part of the extension $\xi\lalp$.
We will show that possibly after a finite base change, we can find $(\sL,\sN,\varphi,\rho,\nu)$ over 
$\sC$ together with isomorphisms
$(\ti\sL,\ti\sN)\cong \beta\sta(\sL,\sN)$ and $(\ti\varphi,\ti\rho,\ti\nu)=
\beta\sta(\varphi,\rho,\nu)$

Without loss of generality, we can assume that $(\sC\lsta)_\sing$ consists of a single $S\lsta$-section.
Thus $\sC$ is the gluing of $\ti\sC$ along $\Upsilon_1\cong\Upsilon_2$.
Let 
$$   
\iota_i:\Upsilon\lra \ti\sC
$$
be the composite $\Upsilon\cong \Upsilon_i\to \ti\sC$ of the tautological maps.
We first consider the case where $\iota_1\sta\ti\varphi\ne 0$. Then necessarily $\iota_2\sta\ti\varphi\ne 0$.
Since $\xi\lalp\in \cW^-_{g\lalp,\gamma\lalp,\bd\lalp}(S)$,
$\ti\rho=0$ in a neighborhood $\ti\sU$ of $\Upsilon_1\cup\Upsilon_2$ in $\ti\sC$, thus $\ti\nu_2$, which is nowhere vanishing
in $\ti\sU$, induces an isomorphism $\ti\sN|_{\ti\sU}\cong \sO_{\ti\sU}$, hence inducing
$\iota_i\sta\ti\sN\cong \sO_\Upsilon$ so that $\iota_i\sta\ti\nu_2=1$.
Note that in this case, $\ti\sC$ is a scheme along $\Upsilon_i$. 

For $i=1$ or $2$, we consider $(\iota_i\sta\ti\varphi_1,\cdots,\iota_i\sta\ti\varphi_5,\iota_i\sta\ti\nu_1)$, 
which is a nowhere vanishing section of $H^0(\iota_i\sta\ti\sL^{\oplus 6})$.
It induces a morphism $\beta_i: \Upsilon\to \PP^5$. Because $\beta_1|_{\Upsilon\lsta}=
\beta_2|_{\Upsilon\lsta}$, we have $\beta_1=\beta_2$. Consequently, we have a unique isomorphism $\phi: \iota_1\sta\ti\sL\cong 
\iota_2\sta\ti\sL$ so that 
\beq\label{glue5}
\phi\sta(\iota_2\sta\ti\varphi_1,\cdots,\iota_2\sta\ti\varphi_5, \iota_2\sta\ti\nu_1)=
(\iota_1\sta\ti\varphi_1,\cdots,\iota_1\sta\ti\varphi_5,\iota_1\sta\ti\nu_1).
\eeq
Using $\iota_i\sta\ti\sN\cong \sO_\Upsilon$ and $\iota_i\sta\ti\nu_2=1$, we have isomorphism 
$\phi':\iota_1\sta\ti\sN\cong 
\iota_2\sta\ti\sN$ so that 
$\phi^{\prime\ast}\iota_2\sta\ti\nu_2=\iota_1\sta\ti\nu_2$.

Applying the scheme version of Corollary \ref{glue3}, we obtain invertible sheaves $\sL$ and $\sN$ on $\sC$ with isomorphisms
$\beta\sta\sL\cong \ti\sL$ and $\beta\sta\sN\cong \ti\sN$ whose restrictions to $\Upsilon$ are
$\phi$ and $\phi'$ respectively. By   \eqref{glue5} and 
Corollary \ref{glue3}, we also obtain sections 
$\varphi\in H^0(\sL)^{\oplus 5}$
and $\nu_2\in H^0(\sN)$ that are liftings of 
$\beta\lsta \ti\varphi$ and $\beta\lsta \ti\nu_2$, respectively, which thus satisfy
$\beta\sta \varphi=\ti\varphi$ and $\beta\sta\nu_2=\ti\nu_2$,
under the given isomorphisms. 

It remains to check that $\ti\nu_1$ and $\ti\rho$ can be glued to sections over $\sC$. 
In this case, since $\sC$ is a scheme along $\beta(\Upsilon_1)=\beta(\Upsilon_2)$, using
that $\ti\sL$ glues to $\sL$ we conclude that $\ti\sL\mof\otimes\omega^{\log}_{\ti\sC/S}$
glues to $\sL^{\vee\otimes 5}\otimes\omega^{\log}_{\sC/S}$. Since $\ti\rho$ vanishes along
$\Upsilon_1\cup\Upsilon_2$, $\beta\lsta\ti\rho$ lifts to $\rho$ so that $\rho|_{\beta(\Upsilon_1)}=0$.
The gluing of $\ti\nu_1$ is similar.  This proves the existence of gluing in this case.

The other case is when $\iota_1\sta\rho\ne 0$, which implies that $\iota_2\sta\ti\rho\ne 0$. 
In this case, we must have $\ti\varphi|_{\ti\sU}=0$ over a neighborhood ${\ti\sU}$ of $\Upsilon_1\cup\Upsilon_2$ in
$\ti\sC$. Consequently, $\ti\nu_1|_{\ti\sU}$ is nowhere %$\iota_i\sta\ti\nu_1$ nowhere
vanishing, forcing $\ti\sL\dual|_{\ti\sU}\cong \ti\sN|_{\ti\sU}$. In particular, we have
the induced isomorphism $\iota_i\sta(\ti\sL\otimes\ti\sN)\cong \sO_{\Upsilon}$ and $\iota_i\sta\ti\nu_1=1$ 
under this isomorphism.

To proceed, we use the canonical isomorphisms $\iota_i\sta\omega_{\ti\sC/S}^{\log}\cong\sO_{\Upsilon}$
due to that $\Upsilon_i\sub\ti\sC$ is a section along smooth locus of fibers of $\ti\sC\to S$.
%Let $\Gamma_i\sub \ti C\to C$ be coarse moduli of $\Upsilon_i\sub \ti\sC\to S$, and let
%$\ti\pi: \ti\sC\to\ti C$ be the projection. Then since $\Gamma_i\sub \ti C$ is a smooth divisor, we have
%canonical isomorphism
%$\omega_{\ti C/S}(\Gamma_i)|_{\Gamma_i}\cong \sO_{\Gamma_i}$. 
%Because canonically $\iota_i\sta\omega_{\ti\sC/S}^{\log}=\iota_i\sta\bl\ti\pi\sta \omega_{\ti C/S}(\Gamma_i)|_{\Gamma_i}\br$,
%we obtain the desired isomorphism $\iota_i\sta\omega_{\ti\sC/S}^{\log}\cong\sO_{\Upsilon}$.
Using this isomorphism, we can view $\iota_i\sta\ti\rho$ as a section in $H^0(\iota_i\sta\ti\sL^{\vee\otimes 5})$. 
Using $\iota_i\sta\ti\sL\dual\cong \iota_i\sta\ti\sN$, $\iota_i\sta\ti\nu_2$ is a section in $H^0(\iota_i\sta\ti\sL\dual)$.
Because $(\iota_i\sta\ti\rho,\iota_i\sta\ti\nu_2)$ is nowhere vanishing, it defines a morphism $\beta_i: \Upsilon_i\to \PP_{(5,1)}$.
Because $\beta_1\times_S S\lsta=\beta_2\times_S S\lsta$, we have $\beta_1=\beta_2$. 
Thus there are isomorphisms 
$$\phi': \iota_1\sta\ti\sL\cong\beta_1\sta\sO_{\PP_{(5,1)}}(1)= \beta_2\sta\sO_{\PP_{(5,1)}}(1)\cong 
\iota_2\sta\ti\sL$$
%and sections $s_1\in H^0(\sO_{\PP_{(5,1)}}(5))$ and $s_2\in H^0(\sO_{\PP_{(5,1)}}(1))$ so that
so that $\iota_1\sta\ti\rho=\iota_2\sta\ti\rho$ and $\iota_1\sta\ti\nu_2=\iota_2\sta\ti\nu_2$ 
(with the known $\iota_i\sta\omega_{\ti\sC/S}^{\log}\cong \sO_\Upsilon$). 

Like before, using $\phi'$, and applying Corollary \ref{glue3}, we can glue $\ti\sL$ to get $\sL$ on $\sC$ so
that, letting $\iota: \Upsilon\cong \alpha(\Upsilon_1)\sub\sC$ be the tautological inclusion, the isomorphisms
$\iota_1\sta\ti\sL\cong \iota\sta\sL\cong \iota_2\sta\ti\sL$ induce the isomorphism $\phi'$.
We next glue $\ti\sN$. Let $\sU$ be the image of
$\ti\sU$ under $\ti\sC\to\sC$, which is a neighborhood of $\Upsilon\sub\sC$.
Then using $\ti\sL\dual|_{\ti\sU}\cong \ti\sN|_{\ti\sU}$, we see that
$\ti\sN$ glues to get $\sN$ on $\sC$ so that $\sN|_{\sU}\cong \sL^{\vee}|_{\sU}$, consistent with
the isomorphism $\ti\sL\dual|_{\ti\sU}\cong \ti\sN|_{\ti\sU}$.

As before, applying Corollary \ref{glue3}
and using that $\iota_1\sta\ti\nu_2=\iota_2\sta\ti\nu_2$ and
$\iota_1\sta\ti\rho_1=\iota_2\sta\ti\rho_2$ under $\phi':\iota_1\sta\ti\sL\cong \iota_2\sta\ti\sL$,
we conclude that $\ti\nu_2$ and $\ti\rho$ glue to 
$\nu_2$  and  $\rho$ of $\sN$ and $\sL^{\vee\otimes 5}\otimes\omega^{\log}_{\sC/S}$,
respectively. Since $\ti\varphi|_\sU=0$, it glues to $\varphi$ such that $\varphi|_{\sU}=0$.
For $\ti\nu_1$, since it induces isomorphism $\ti\sN|_{\ti\sU}\cong \ti\sL^{\vee}|_{\ti\sU}$,
and this isomorphism descends to $\sN|_{\sU}\cong \sL^{\prime\vee}|_{\sU}$, we see that
$\ti\nu_1$ glues to get $\nu_1$. Finally, we let $\Si^{\sC}$ be the image
of $\Si^{\ti\sC}-\bl\Upsilon_1\cup\Upsilon_2\br$, Then
$$\xi=(\Si^{\sC},\sC,\sL,\sN,\varphi,\rho,\nu)\in \cW^-_{g,\gamma,\bd}(S).
$$
This proves the Proposition.
\end{proof}

As the proof doesn't use the condition $\varphi_1^5+\ldots+\varphi_5^5=0$, we have
\begin{coro}\label{propercor}
Let $\cW^{\wti{}}\lggd\subset \cW\lggd$ be the reduced closed substack where close points are $\xi\in\cW\lggd(\CC)$ such that $(\varphi=0)\cup (\rho=0)=\sC$ (c.f. Lemma \ref{degenerate-locus}). Then $\cW^{\wti{}}\lggd$ is proper. 
\end{coro}
\section{Finite presentation of the degeneracy locus}

In this  section, we prove that $\cW\lggd$ is separated and $\cW^-\lggd$ is
of finite type.

\subsection{Separatedness}\label{separatedness}

In this subsection, we prove that $\cW\lggd$ is separated. % near $\cW^-_{g,n,\bd}$. 
As before,  %$R$ be a DVR over $\CC$, 
 $\eta_0\in S$ is a closed point in a smooth curve over $\CC$,  and $S\lsta=S-\eta_0$.
%We will show that if $\xi$, $\xi'\in \cWg(S)$ such that $\xi\lsta= \xi'\lsta$, then $\xi=\xi'$.

%We begin with the following special case.

\begin{lemm}\label{valuative2}
Let $\xi$, $\xi'\in \cW\lggd(S)$ be such that $\xi\lsta\cong \xi'\lsta\in \cW\lggd(S\lsta)$.
% and $\xi_0$, $\xi'_0\in \cW\lggd^-(\eta_0)$.
Suppose $\sC\lsta$ is smooth. Then $\xi\cong\xi'$.
\end{lemm}

\begin{proof}
Let $\xi=(\Si^\sC,\sC, \sL,\cdots)$ and $\xi'=(\Si^{\sC'},\sC', \sL',\cdots)$, let
$C$ (resp. $C'$)  be the coarse moduli of $\sC$ (resp. $\sC'$), and let $\pi: X\to C$ (resp. $\pi':X'\to C'$)
 be the minimal desingularization. Thus $\pi$ and $\pi'$ are contractions of chains of $(-2)$-curves.

Let $f: X\dashrightarrow X'$ be the birational map induced by $\xi\lsta\cong\xi'\lsta$,  let
$U_0\sub X$ be the largest open subset over which $f$ is well-defined. Suppose $U_0\subsetneq X$,
then $X-U_0$ is discrete. Let $X_1$ be the blowing up of $X$ at $X-U_0$. Inductively, suppose
$X_k\to X$ is a successive blowing up along points, let $U_k\sub X_k$ be the largest open subset over
which the birational $f: X_k\dashrightarrow X'$ is well-defined,  then let $X_{k+1}$ be the blowing up of 
$X_k$ along $X_k-U_k$. After finite steps, we have $X_{\bar k}=U_{\bar k}$, thus $X_{\bar k}=X_{\bar k+1}$
for large $\bar k$. Denote ${Y}=X_{\bar k}$ for large $\bar k$ with
$$\bar \pi: {Y}\lra X\and \bar f: {Y}\lra X'
$$
the induced projection and birational morphism. 

Let $E\sub {Y}$ be the exceptional divisor of $\bar\pi$. Write $E=\sum_{k\ge 1} E_k$,
where $E_k\sub E$ is the proper transform of the exceptional divisor of $X_k\to X_{k-1}$      
when $k\ge 2$ and the exceptional divisor of the map $X_1\to X$ when $k=1$. 
Let $E'\sub Y$ be the exceptional divisor of $\bar f$.
By our construction, $E$ and $E'$ share no common irreducible curves.
Let $Y_0=\cup_j^n D_j$ be the irreducible component decomposition of the central fiber $Y_0$.
%We let $D_i$ be the irreducible components of $Y_0$. 
%For each $D_i$, we define
%$m(i)=k$ if $D_i\sub E_k$, and $m(i)=0$ if $D_i\subsetneq E$.

Furthermore, by our construction for $V=Y-E'$,
$\bar f|_V:V\to \bar f(V)$ is
an isomorphism, and by the blowing up formula, we have
\beq\label{dualizing} \omega_{{Y}/S}^{\log}\cong \bar\pi\sta\omega_{X/S}^{\log}(\sum_i i E_i) \and
\omega_{{Y}/S}^{\log}|_V
\cong \bar f\sta\omega^{\log}_{X'/S}|_V.
\eeq
Let $L$ and $N$ (resp. $L'$ and $N'$) be the line bundles on $X$ (resp. $X'$)
that are the pullbacks of the descents of $\sL^{\otimes 5}$ and $\sN^{\otimes 5}$ (resp. $\sL^{\prime\otimes 5}$ and $\sN^{\prime\otimes 5}$) to $C$ (resp. $C'$).
Using $\xi\lsta\cong \xi'\lsta$, we can find integers $a_i$ and $b_i$ so that
\beq\label{LN}
\bar f\sta L'\cong \bar\pi\sta L(\sum a_i D_i)\and \bar f\sta N'\cong \bar\pi\sta N(\sum b_i D_i).
\eeq
Let $u_1$, $u_2$ and $h_j$ be the pullbacks of the descents of $\nu_1^5$, $\nu_2^5$
and $\varphi_j^5$, which are sections
of $L\otimes N$, $N$ and $L$, respectively. Denote by $u_1'$, $u_2'$ and $h_j'$ be the pullbacks of the descents of 
$\nu_1^{\prime 5}$, $\nu_2^{\prime 5}$
and $\varphi_j^{\prime 5}$ similarly. By the same reason, we will view $\rho$ and $\rho'$
as sections in $H^0(X, L\otimes\omega_{X/S}^{\log})=H^0(\sC, \sL^{\vee\otimes 5}\otimes\omega_{\sC/S}^{\log})$,
and in $H^0(X', L^{\prime -1}\otimes\omega_{X'/S}^{\log})$ respectively.

We now show that $E=\emptyset$, namely, $f$ is a morphism. Suppose not,  let
$D_j\sub E$ be an irreducible component with $x=\bar\pi(D_j)\in X$. 
We first remark that $x$ is a smooth point of $X_0$. Indeed, that $x$ is a singular point of $X_0$ implies
that $V_0=V\times_S \eta_0$ is not reduced. On the other hand, since $\bar f|_V: V\to X'$ is an $S$-isomorphism
onto its image and since  $X_0'$ is reduced, we conclude that $V_0$ is reduced too. This proves that $x$ is a smooth
point of $X_0$. 

Let $x=\bar\pi(D_j)$ be as before. Then by the construction of
$\bar\pi$, $\bar\pi\upmo(x)$ is a tree of rational curves. By reindexing $Y_0=\cup_{j=1}^n D_j$, 
we can assume that $D_1+\cdots+ D_k$ form a maximal chain of rational curves in $\bar\pi\upmo(x)$,
namely, $D_i\sub E_i$ for $i\le k$, and $D_i\cap D_{i+1}\ne \emptyset$ for $i<k$, thus $D_k\sub Y$ is a 
$(-1)$-curve. Therefore by
our construction of $(\bar \pi,\bar f)$, $\bar f(D_1),\cdots, \bar f(D_k)$ is a chain of rational 
curves in $X'$, and $\bar f(D_k)$ is a $(-1)$-curve in $X'$  since $x$ is  a smooth point of $X_0$. In particular, the image of 
$\bar f(D_k)$ in $\sC_0'$, denoted by $\sD'$, is a rational curve. Let $z\in D_k$ be a general point
and let $y=\bar f(z)$. Note $x=\bar \pi(z)$.

\begin{subl}\label{case2}
Let the situation be as stated. Then we   have $u_2(x)=0$.
\end{subl}

\begin{proof}
We prove by contradiction. Suppose $u_2(x)\ne 0$.  We claim that $a_k=b_k=0$. 

We divide it into two cases. 
The first is when $u_1(x)\ne 0$.
Since $\bar\pi\sta u_2\in H^0(V, \bar\pi\sta N)$ and non-trivial along ${V\cap D_{k}}$,
and since $\bar f\sta u_2'\in H^0(V, \bar f\sta N')$,
using \eqref{LN} and that $\bar\pi\sta u_2|_{V\times_S S\lsta}=\bar f\sta u_2'|_{V\times_S S\lsta}$, 
we conclude that
$b_{k}\ge 0$. Similarly, using that $u_1(x)\ne 0$, we conclude that $a_{k}+b_{k}\ge 0$.

We claim that $a_{k}+b_{k}=0$. Suppose not, then by \eqref{LN}, we have $u_1'(y)=0$,
thus $u_2'(y)\ne 0$ and $(h_k'(y))\ne 0$, which forces $b_{k}\le 0$ and $a_{k}\le 0$, contradicting to $a_{k}+b_{k}>0$. This proves $a_{k}+b_{k}=0$.

We now prove $a_k=b_k=0$. Suppose not. Since  $a_{k}+b_{k}=0$ and 
$b_{k}\ge 0$, we must have $b_{k}>0$. Then $a_{k}<0$ and $u_2'(y)=0$. 
Note that because of \eqref{dualizing} and \eqref{LN}, for any $i\le k$ and a dense open $U\sub D_i$ 
that is disjoint from the nodes of $Y_0$,
\beq\label{Dd1}
\bar f\sta(L^{\prime\vee}\otimes\omega^{\log}_{X'/S})|_U
\cong \bar \pi\sta(L^{\vee}\otimes\omega^{\log}_{X/S})((i-a_{i})D_{i})|_U.
\eeq
Applying to $i=k$, we conclude that $\rho'(y)=0$, contradicting to $u_2'(y)=0$. 
Thus $a_{k}=b_{k}=0$.

The other case is when $u_1(x)=0$.
Since $u_1(x)=0$, we have $(h_i(x))\ne 0$. 
Similar to the argument above, using $h_i(x))\ne 0$ (resp. $u_2(x)\neq 0 $), we conclude 
that $a_{k}\ge 0$ (resp. $b_{k}\ge 0$).

We now show that $b_{k}=0$. Suppose not, that is $b_k>0$, 
then we must have $\bar f\sta u_2'|_{D_{k}}=0$,
which forces $\bar f\sta \rho'|_{D_{k}}\ne 0$. Applying \eqref{dualizing} and \eqref{LN}, 
we must have $k-a_{k}\le 0$, thus $a_{k}\ge k\ge 1$,
which forces $\bar f\sta u_1'|_{D_{k}}=0$,
violating that $(u_1',u_2')$ is nowhere vanishing. This proves $b_{k}=0$.

A similar argument shows that $a_{k}>0$ would lead to $\bar f\sta h_k'|_{D_{k}}=\bar f\sta u_1'|_{D_{k}}=0$,
a contradiction. Therefore, $a_{k}=b_{k}=0$ in this case, too.

%We now show that $D_k$ is a (-1)-curve of $Y$ contradicts with $a_k=b_k=0$.
Let $\sD'\sub\sC'$ be the irreducible component whose image in 
$C'$ is the same as the image of $D_k$ under
$Y\to X'\to C'$. Since $\bar f(D_k)$ is a (-1)-curve, $\sD'$ is a smooth rational curve in $\sC'$ and contains 
exactly one node of $\sC'_0$ and at most one marking of $\Si^{\sC'}$. %; we denote this node by $q'$.

Next, we use \eqref{dualizing} and $a_k=b_k=0$ to conclude that $\rho'|_{\sD'}=0$.
Therefore, $\sN'|_{\sD'}\cong \sO_{\sD'}$, and $(\varphi_1',\cdots, \varphi_5',
\nu_1')|_{\sD'}$ defines a morphism $\beta':\sD'\to \mathbb P^5$. Let
$q'\in \sD'$ be the node of $\sC_0'$. By \eqref{LN} and that
$a_k=b_k=0$, we conclude that the pullback of $(\varphi_1',\cdots, \varphi_5',\nu_1')|_{\sD'-q'}$
 to $Y$ equals to the pullback of $(\varphi_1,\cdots, \varphi_5,\nu_2)|_{x}$, thus
$\beta':\sD'\to\mathbb P^5$ is a constant map. 
Since $\sD'$ contains one node of $\sC'_0$ and at most one marking in $\Si^{\sC'}$, adding that
$\rho'|_{\sD'}=0$,
by Lemma \ref{stable-cri} we conclude that $\xi_0'$ is unstable, a contradiction. This proves the Sublemma.
\end{proof}

We continue to denote by $D_1+\cdots+D_k$ a maximal chain of rational curves in $\bar\pi\upmo(x)$
with $D_i\in E_i$.

%To proceed, we introduce the notion of the depth of irreducible $D_j\sub E$. For irreducible
%$D_i$, $D_j\sub E$, we say
%$D_i\prec D_j$ if $D_i\cap D_j\ne \emptyset$ and $m(i)<m(j)$ (thus necessarily $m(i)+1=m(j)$.)
%For irreducible $D_i\sub E$, we define its depth to be
%$\delta(i)=\max\{k\mid D_i=D_{i_0}\prec\cdots\prec D_{i_k}\sub E\}$. Note that an irreducible
%$D_i\sub E$ is a (-1)-curve if and only if $\delta(i)=0$.

\begin{subl}\label{case3}
We have $a_{k}=k$,  $a_i\le a_{i+1}-2$ for $i<k$, and $a_{i}+b_{i}=0$ for all $i\le k$.
\end{subl}

\begin{proof}
First, we have \eqref{Dd1}. We claim  that for $i\le k$,
%(Recall $m(i_0)=j$ if $D_{i_0}\in E_j$.) We claim that
\beq\label{Dd2}
i-a_{i}\ge 0 \and a_{i}+b_{i}= 0.
\eeq
In fact, by the previous Sublemma, we know $u_2(x)=0$, thus $\rho(x)\ne0$ and 
$\bar\pi\sta \rho|_{D_i}\ne 0$. Adding that $\bar f\sta \rho'$ is regular along $V$ and coincides with 
$\bar\pi\sta \rho$ over $V\times_S S\lsta$, we obtain the first inequality in \eqref{Dd2}.
Similarly, using $u_1(x)\ne 0$, we obtain $a_{i}+b_{i}\ge 0$.

We now show $a_{i}+b_{i}=0$. Suppose not, say $a_i+b_i>0$, 
then $\bar f\sta u_1'|_{D_{i}}=0$,
which forces $\bar f\sta h_j'|_{D_{i}}\ne 0$ for some $j$, and by \eqref{LN}, we obtain $a_i\le 0$. As
$a_i+b_i>0$, we obtain $b_i>0$, and hence $\bar f\sta u_2'|_{D_{i}}=0$,
contradicting to $(u_1'(y), u_2'(y))\ne 0$. This proves \eqref{Dd2}.

We next prove $a_k=k$. Suppose not,
by \eqref{Dd2} we have $k-a_k\ge 1$. Then 
$\bar f\sta \rho'|_{D_k}=0$, which forces $\bar f\sta u_2'|_{D_k}$ nowhere vanishing. By 
\eqref{LN}, we have $b_{k}\le 0$; adding $a_{k}+b_{k}=0$, we have $a_{k}\ge 0$.

We claim $a_k>0$. Suppose not, i.e.,  $a_k=0$. Let $\sD'\sub\sC'$, as before, be the irreducible component whose image in 
$C'$ is the same as the image of $D_k$ under
$Y\to X'\to C'$. As argued in the proof of the previous Sublemma, $\sD'$ is a smooth rational 
curve in $\sC'$ and contains one node $q'$ of $\sC'_0$.
As $a_k=0$, $\rho'|_{\sD'}=0$,  $\sN'|_{\sD'}\cong \sO_{\sD'}$,
and $(\varphi_1',\cdots, \varphi_5',\nu_1')|_{\sD'}$ defines a morphism $\beta':\sD'\to\mathbb P^5$,
which turns out to be constant, as argued before. Therefore, as $\sD'$ contains at most one
marking of $\Si^{\sC'}$, $\xi_0'$ becomes unstable, a contradiction. This proves that $a_k$ is positive.

Therefore, by the property $\bar f\sta \rho'|_{D_k}=0$, etc., 
we conclude that $\rho'|_{\cD'}=\varphi'|_{\sD'}=0$, and both $\nu_1'|_{\sD'}$ and $\nu_2'|_{\sD'}$ are
nowhere vanishing.
Because $\sD'$ contains one node and at most one marking of $\sC'_0$,
$\xi_0'$ becomes unstable, a contradiction. 
This proves that $a_k=k$.

Finally, we prove $a_i\le a_{i+1}-2$.
%next consider $D_i\in E$ of $\delta(i)\ge 1$.
%We claim that for any $D_i\prec D_j\sub E$, we have $a_i\le a_j-2$. Suppose not, then
Let 
$\Lam=\{1\le i< k\mid a_i\not\le a_{i+1}-2\}$. 
The intended inequality is equivalent to $\Lam=\emptyset$.
Suppose $\Lam\ne\emptyset$, and let $i$ be the largest element in $\Lam$.
Suppose $i=k-1$. Since $a_k=k$, we have $a_{k-1}\ge a_k-1=k-1$.
By \eqref{Dd2}, we conclude that $a_{k-1}=k-1$, which implies $\deg\bar f\sta L'|_{D_k}=-1$.

Consequently, using \eqref{Dd1}, we conclude that $\bar f\sta\rho'|_{D_k}$ is nowhere vanishing.
Like before, let $\sD'\sub\sC'_0$ be the irreducible component associated to $D_k\sub E$, via $Y\to C'$ and
$\sC'\to C'$. Then $\sD'$ is a rational curve, contains one node of $\sC'_0$, and with 
$\varphi'|_{\sD'}=0$, $\rho'|_{\sD'}$ nowhere vanishing and $\deg\sL'|_{\sD'}=-\ofth$.
By Lemma \ref{stable-cri}, this makes $\xi'_0$ unstable, a contradiction. Therefore, $i<k-1$.

Since $i+1\not\in\Lam$,
$(i+1)-a_{i+1}\ge 1$, which forces $\bar f\sta\rho'|_{D_{i+1}}=0$, and then $\bar f\sta u_2'|_{D_{i+1}}$
is nowhere vanishing and $\bar f\sta N'|_{D_{i+1}}\cong\sO_{D_{i+1}}$. We claim that
$\bar f\sta L'|_{D_{i+1}}\cong\sO_{D_{i+1}}$. Indeed, as $\bar f\sta N'|_{D_{i+1}}\cong\sO_{D_{i+1}}$,
$(\bar f\sta h'_1,\cdots, \bar f\sta h'_5, \bar f\sta u_1')|_{D_{i+1}}$ defines a morphism
$\beta:D_{i+1}\to\mathbb P^5$. By the second inequality in (\ref{Dd2}) and  \eqref{LN},  it is a constant map. Thus
$\bar f\sta L'|_{D_{i+1}}\cong\sO_{D_{i+1}}$.

Now let $D_i, D_{i+2}, D_{k_2},\cdots, D_{k_l}$ 
be the irreducible components
in $E$ that interest with $D_{i+1}$. Since ${i+1}\not\in\Lam$, $a_{i+1}\le a_{i+2}-2$.
Possibly by changing to a different maximal chain of rational curves in $\bar\pi\upmo(x)$,
we can assume without loss of generality that $a_{i+1}\le a_{k_s}-2$ for all $2\le s\le l$. 
Because $\bar f\sta L'|_{D_{i+1}}\cong \sO_{D_{i+1}}$, we must have
$(a_i-a_{i+1})+\sum_s(a_{k_s}-a_{i+1})=0$. Therefore, $a_i-a_{i+1}\le -2$, a contraction. 
This proves $\Lam=\emptyset$,
and the Sublemma follows.
\end{proof}

We continue our proof of Lemma \ref{valuative2}. We keep the maximal chain of rational curves
$D_1,\cdots, D_k\sub E$. We claim $k=1$. Otherwise, $k\ge 2$ and $a_1\le a_2-2\le 0$. 
By \eqref{Dd1} we 
obtain $\bar f\sta \rho'|_{D_1}=0$ and thus $\bar f\sta u_2'|_{D_1}$ is nowhere vanishing. Since $u_2(x)=0$ by Sublemma \ref{case2}, we get $\bar\pi\sta u_2|_{D_1}=0$. Thus from (\ref{LN}), we get $b_1<0$ and hence $a_1=-b_1>0$, a contradiction to $a_1\le 0$. Thus $k=1$. 

Next we prove that every $D_\ell\subset Y_0$ not in $E$ that intersects one of $D_i\subset E$ must 
be contracted by $\bar f$. Indeed, 
if $\bar f(D_\ell)$ is not a point, since $\bar \pi(D_\ell)$ is not a point as well, we must have 
$a_\ell=b_\ell=0$ and $D_\ell\cap V$ is dense in $D_\ell$. Since $D_i$ is a $(-1)$-curve, the combination 
of (\ref{dualizing}), (\ref{LN}) and $\rho(x)\neq 0$ (since $u_2(x)=0$) implies $\bar f\sta\rho'|_{D_i}$ is nowhere vanishing. From (\ref{LN}), we also have $\deg(\bar f\sta L'|_{D_i})=-1$ since $a_i=1$ and $a_\ell=0$. Thus $\xi_0'$ satisfies the condition (2) in the Lemma \ref{stable-cri} and hence is not stable, a contradiction. 
\black

In conclusion, we have proved that $E$ is a disjoint union of (-1)-curves, likewise for $E'$, and
that every irreducible component in $Y_0$ not in $E$ must be in $E'$.
%Reversing the role of $X$ and $X'$, we conclude that the exceptional divisor $E'$ of $Y'\to X'$ is also a
%disjoint union of (-1)-curves, and further every irreducible $D_i\sub E$ intersects with one $D_l\sub E'$.
Since $Y\to X$ is by first blowing up smooth points of $X_0$, and since $Y_0$ is connected, this is possible 
only if both $E\cong\Po$, and
then $Y$ is the blowing up of $X=S\times\Po$ at a single point in $X_0$.

Then a direct analysis shows that this is impossible, assuming both $\xi_0$ and $\xi_0'$ are stable.
(As this analysis is straightforward, we omit the details here.)
This proves that $E=\emptyset$
and $f: X\to X'$ is a birational morphism. By symmetry, $f\upmo: X'\to X$ is also a birational morphism.
Therefore, $f: X\cong X'$ is an isomorphism.

%A straightforward case by case analysis 
%In this case, $E\cong \Po\sub Y_0$ be contracted by $\bar \pi$, and let $E'\cong\Po$ be the other
%irreducible component of $Y_0$ contracted by $\bar f$. Then by the Sublemmas proved, we have
%\beq\label{iii}
%\bar f\sta(L^{\prime\vee}\otimes\omega^{\log}_{X'/S})
%\cong \bar \pi\sta(L^{\vee}\otimes\omega^{\log}_{X/S}),
%\eeq
%and $\bar f\sta L'\cong \bar \pi\sta L(E-E')$. Consequently, we have 
%$\bar f\sta L'|_E\cong \bar \pi\sta L(E-E')|_E\cong \sO_E(-2)$ and $
%\bar f\sta L'|_{E'}\cong\sO_{E'}$.
%Therefore, $\deg \sL^{\otimes 5}|_{\sC_\eta}=-2$, and thus $\varphi_\eta=0$ and $\sL_\eta\cong \sN\dual_\eta$,
%which implies $\rho_\eta\ne 0$, and hence $\rho_{\eta_0}\ne 0$. Applying \eqref{iii}, we
%conclude that both $\rho_{\eta_0}$ and $\rho'_{\eta_0}$ are nowhere vanishing.
%Combined with $\deg \sL^{\otimes 5}|_{\sC_\eta}=-2$, we conclude that 
%$\omega|_{\sC/S}^{\log}|_{\sC_\eta}\cong\sO_{\sC_\eta}(-2)$ and
%consequently $\Si^\sC=\emptyset$. But then $\sC_\eta$ is a scheme and $\deg\sL|_{\sC_\eta}\in\ZZ$, violating that
%$\deg \sL^{\otimes 5}|_{\sC_\eta}=-2$. 

Knowing that $f$ is an isomorphism, a parallel argument shows that
\beq\label{isocc}
f\sta L'\cong L,\quad f\sta N'\cong N, \and   f\sta(\rho',h_k', u_i')=(\rho,h_k,u_i).
\eeq
We prove that this implies $C\cong C'$. Indeed,
it is easy to show that a $D_i\sub X$ is contracted by $\pr: X\to C$ if and only if $D_i\sub X$ is a (-2)-curve and
$L|_{D_i}\cong N|_{D_i}\cong \sO_{D_i}$. Therefore, $D_i$ is contracted by the map $X\to C$ if and only if
$f(D_i)$ is contracted by the map $X'\to C'$. This proves that $f: X\cong X'$ induces an isomorphism $\bar\phi: C\cong C'$.

Let $\Delta$ (resp. $\Delta'$) be the set of singular points of $\sC_0$ (resp. $\sC'_0$). 
Let $p: \sC\to C$ and $p':\sC'\to C'$ be the coarse moduli morphisms.
Then the isomorphisms $\bar\phi$ and \eqref{isocc} (with $a_i=b_i=0$) induce isomorphisms
\beq\label{iso01}
\bar\phi\sta p'\lsta(\Si^{\sC'},\sL^{\prime\otimes 5}, \sN^{\prime\otimes 5}, \varphi_k^{\prime 5}, \rho', \nu_i^{\prime 5})
\cong p\lsta(\Si^{\sC},\sL^{\otimes 5}, \sN^{\otimes 5}, \varphi_k^{ 5}, \rho, \nu_i^{5});
\eeq
and isomorphisms
\beq\label{iso02}
\phi: \sC-\Delta\mapright{\cong} \sC'-\Delta', %\quad \phi\sta\omega_{\sC'/S}^{\log}\cong \omega_{\sC/S}^{\log},
\quad \phi\sta \sL'\cong \sL, \quad \phi\sta \sN'\cong \sN, \quad \phi\sta(\varphi',\rho',\nu_i')=(\varphi,\rho,\nu_i),
\eeq
%and identities $\phi\sta(\varphi',\rho',\nu_1',\nu_2')=(\varphi,\rho,\nu_1,\nu_2)$,
extending $\xi\lsta\cong \xi'\lsta$. 

Now let $p\in \Delta$ be a point and let $p'\in\Delta'$ be the corresponding point. Pick
an  open subset $\sU\sub\sC$ of $p\in\sC$ so that $\sU\cap \Delta=p$. Let
$\sU'\sub\sC'$ be the   open subset of $p'\in\sC'$ so that $\sU'\cap \Delta'=p'$
and $\phi(\sU-p)=\sU'-p'$. If $p$ is a scheme point, then $\phi$
extends to a morphism $\ti\phi: (\sC-\Delta)\cup\sU\to (\sC'-\Delta')\cup\sU'$
so that \eqref{iso02} extends to
\beq\label{iso03}
%\ti\phi: (\sC-\Delta)\cup\sU\mapright{\cong} (\sC'-\Delta')\cup\sU'
\ti\phi\sta \sL'\cong \sL, \quad \ti\phi\sta \sN'\cong \sN, \quad \ti\phi\sta(\varphi',\rho',\nu_i')=(\varphi,\rho,\nu_i).
\eeq
However, by (2) in Definition \ref{def-curve}, this implies $p'$ is also a scheme point,
and $\ti\phi$ is an isomorphism. If $p'$ is a scheme point, the same conclusion holds by switching
the role of $\sC$ and $\sC'$. Finally, when both $p$ and $p'$ are stacky points, 
then both are $\mufive$-srtacky
points. Thus a local argument shows that $\phi$ extends to an isomorphism 
$\ti\phi: (\sC-\Delta)\cup\sU\to (\sC'-\Delta')\cup\sU'$
so that \eqref{iso02} extends to \eqref{iso03}. By going through this local extension throughout  all points in $
\Delta$, we conclude that $\phi$ extends to an isomorphism $\ti\phi:\sC\to\sC'$ so that \eqref{iso02} extends to \eqref{iso03}.
This proves that $\xi\cong \xi'$.
\end{proof}

\begin{prop}\label{valuative3}
Lemma \ref{valuative2} holds without assuming that $\sC\lsta$ is smooth. 
\end{prop}

\begin{proof}
Let $\xi=(\Si^\sC, \sC,\cdots)$ be as before, and let
$\xi\lalp$, $\alpha\in \Xi$, be families constructed as in Corollary \ref{split}. 
Let $\xi'=(\Si^{\sC'},\sC',\cdots)$ and likewise $\xi'\lalp$, $\alpha\in \Xi$, be the similar
decomposition. Here both $\xi\lalp$ and $\xi'\lalp$ are indexed by the same set $\Xi$ because
$\xi\lsta\cong \xi'\lsta$. By Corollary \ref{split}, all $\xi\lalp$ and $\xi'\lalp$ are stable families of
MSP-fields. 

Because $\xi\lsta\cong \xi'\lsta$, we have $\xi_{\alpha\ast}\cong \xi'_{\alpha\ast}$.
By Lemma \ref{valuative2}, $\xi_{\alpha\ast}\cong \xi'_{\alpha\ast}$ extends to $\xi\lalp\cong \xi'\lalp$.
Then a direct argument shows that as the isomorphisms $\xi\lalp\cong \xi'\lalp$ are consistent with
the isomorphism $\xi\lsta\cong \xi'\lsta$, they induce an isomorphism $\xi\cong \xi'$, extending $\xi\lsta\cong \xi'\lsta$.
As the argument is straightforward, we omit the details here. 
\end{proof}

Applying valuative criterion of separateness of DM stacks, Proposition \ref{valuative3} proves

\begin{prop} The stack $\cWg$ is separated.
\end{prop}

\subsection{$\cWg^-$ is of finite type}\label{sepa}
\def\Up{\Upsilon}

Let $\xi\in\cWg^-(\CC)$ be of the presentation $(\Si^\sC,\sC,\cdots)$,
let  $\sC_0=\sC\cap (\nu_1=0)_\redd$,  $\sC_\infty=\sC\cap (\nu_2=0)_\redd$,  and
$\sC_1=\sC\cap (\rho=\varphi=0)_\redd$ be  as before .
Let $\sC_{01}$ (resp. $\sC_{1\infty}$) be the union of irreducible components of $\sC$ in
$(\rho=0)$ (resp. $(\varphi=0)$) but not in 
$\sC_0\cup\sC_1\cup\sC_\infty$. % but intersect with $\sC_0$ (resp. $\sC_\infty$).
%We continue to denote by $\sC\lalp^{\dim 1}$ the pure 1-dimensional part of $\sC\lalp$.

\begin{lemm}\label{0-dec}
We have 
\begin{enumerate}
\item The curves $\sC_0$, $\sC_1$ and $\sC_\infty$ are mutually disjoint; 
\item no two of
$\sC_0,\sC_{01},\sC_1,\sC_{1\infty},\sC_\infty$ share common irreducible components;
\item $\sC=
\sC_0\cup\sC_{01}\cup\sC_{1}\cup\sC_{1\infty}\cup\sC_\infty$.
%\item[2.] $\nu_1|_{\sC_\infty\cup \sC_{1\infty}\cup\sC_1}$, $\nu_2|_{\sC_0\cup \sC_{10}\cup \sC_1}$
%and $\rho|_{\sC_\infty}$
%are nonwhere vanishing
%\item $(\varphi_1,\ldots,\varphi_5)|_{\sC_0}$
%defines a morphism $f_0: \sC_0\to \PP^4$.
%\item[4.] We have
%$\nu_1|_{\sC_0}=0$;
%$\varphi|_{\sC_1}=\rho|_{\sC_1}=0$, and
%$\nu_2|_{\sC_\infty}=0$.
\end{enumerate} 
\end{lemm}

\begin{proof}
For $\xi\in \cWg^-(\CC)$, (1) follows from their definitions. Furthermore, any irreducible component $\sA\sub\sC_{01}$
has $\varphi|_{\sA}\ne 0$. Likewise, any irreducible $\sA\sub \sC_{1\infty}$ has
$\rho|_\sA\ne 0$.
Since $(\varphi=0)\cup(\rho=0)=\sC$, we conclude that $\sC_{01}$ and $\sC_{1\infty}$ share no
common irreducible components. The other parts of (2) are similar. 
Finally, by the same reasoning $(\varphi=0)\cup(\rho=0)=\sC$, we have (3).
\end{proof}

We derive some information of the degrees of $\sL$ and $\sN$ along $\sC_a$. 
First, since $\nu_2|_{\sC_0\cup\sC_{01}\cup\sC_1}$ is nowhere vanishing, 
$\sN|_{\sC_0\cup\sC_{01}\cup\sC_1}
\cong \sO_{\sC_0\cup\sC_{01}\cup\sC_1}$. Thus 
\beq\label{dinf}
d_\infty=\deg\sN=\deg\sN|_{\sC_{1\infty}\cup\sC_{\infty}}.
\eeq
Similarly, since $\nu_1$ restricted to 
$\sC_1\cup\sC_{1\infty}\cup\sC_\infty$ is nowhere vanishing, 
$$\sL\otimes\sN|_{\sC_1\cup\sC_{1\infty}\cup\sC_{\infty}}\cong
\sO|_{\sC_1\cup\sC_{1\infty}\cup\sC_{\infty}}.
$$
Combined, we get 
\beq\label{dz}
\deg\sL|_{\sC_{0}\cup\sC_{01}}=d_0, \quad \deg\sL|_{\sC_{1}}=0,\quad \deg\sL|_{\sC_{1\infty}\cup\sC_{\infty}}=
%-\deg\sN|_{\sC_{1\infty}\cup\sC_{\infty}} =
-d_\infty.
\eeq

For $\xi\in\cW\lggd^-(\CC)$,  let $\Up_\xi$ be the dual graph of $\Si^\sC\sub \sC$. Namely,
vertices of $\Up_\xi$ are associated to irreducible components of $\sC$, edges of $\Up_\xi$ are associated
to nodes connecting two different irreducible components of $\sC$, and legs of $\Up_\xi$ 
are associated to markings $\Si^\sC$. Let $V(\Up_\xi)$ be the set of vertices, and  $E(\Up_\xi)$ be the set of edges. Furthermore, each vertex $v\in V(\Up_\xi)$ is decorated by 
$g_v\defeq g(\sC_v)$, the arithmetic genus of the irreducible component $\sC_v$ associated to $v$.

Given a decorated graph $\Up$ as above (i.e. a connected graph with legs, vertices $v$ decorated by 
$g_v\in \ZZ_{\ge 0}$, and without circular-edges), we say $v\in V(\Up)$ is stable (resp. semistable) if
$2g_v-2+|E_v|\ge 1$ (resp. $\ge 0$), where $E_v$ is the set of legs and edges in $\Up$ attached to $v$.
%(Note that the graph $\Up$ we are considering does not contain any edge attached to a single vertex.)

\begin{prop}
The set $\Theta\defeq \{\Up_\xi\mid \xi\in \cW\lggd^-(\CC)\}$ is a finite set.
\end{prop}

\begin{proof}
Let $\xi\in\cW^-\lggd(\CC)$. As the curve $(\Si^\sC\sub \sC)$ may not   stable, knowing the total genus
and the number of legs of $\Up_\xi$ is not sufficient to bound the geometry of $\Up_\xi$. Our approach is
to use the information of line bundles $\sL$ and $\sN$ on $\sC$ given by $\xi$ to add legs to $\Up_\xi$ to 
form a semistable $\ti\Up_\xi$

First,   let $\index=\{0,01,1,1\infty,\infty\}$. Using (3) of Lemma \ref{0-dec}, for $a\in \index$ we define
$V(\Up_\xi)_a=\{v\in V(\Up_\xi)\mid \sC_v\sub \sC_a\}$. By the the stability criterion Lemma \ref{stable-cri},
we know that all $v\in V(\Up_\xi)_1\cup V(\Up_\xi)_\infty$ are stable. 

We now add new legs to vertices in $V(\Up_\xi)$, called auxiliary legs. 
For every $v\in V(\Up_\xi)_0\cup V(\Up_\xi)_{01}$, we 
add $3\deg\sL|_{\sC_v}$ auxiliary legs to $v$. By \eqref{dz}, the total new legs added to all $v\in
V(\Up_\xi)_0\cup V(\Up_\xi)_{01}$ is $3d_0$.

We next treat the vertices in $V(\Up_\xi)_{1\infty}$. We first introduce 
$$E_{V(\Up_\xi)_\infty}
=\{e\in E(\Up_\xi)\mid e\in E_v\ \text{for some}\ v\in V(\Up_\xi)_\infty\}.
$$
For $v\in V(\Up_\xi)_{1\infty}$, we define
$$\delta(v)=5\deg\sN|_{\sC_v}-|E_v\cap E_{V(\Up_\xi)_{\infty}}|\in \ZZ_{\ge 0}.
$$
Here it takes value in $\ZZ_{\ge 0}$ because $\nu_2$ is non-vanishing along $\sC_v$ and vanishes at 
$\sC_v\cap \sC_\infty$. To every $v\in V(\Up_\xi)_{1\infty}$, we add
$2\delta(v)$ number of auxiliary  legs to $v$. 

We now show that %there is a constant $N$ depending only on $(g,\gamma,\bd)$ so that
the number of new legs added to $V(\Up_\xi)_{1\infty}$ is bounded by $10d_\infty+4g+2\ell$. 
Let $\sC_\infty^1, \ldots, \sC_\infty^r$ be the connected components of $\sC_\infty$, $m_i$ be the number 
of markings on $\sC_\infty^i$.
 Because $\rho|_{\sC_\infty^i}$ is nowhere vanishing, using the discussion leading to
\eqref{dz}, we have
$$0=-5\deg\sL|_{\sC_\infty^i}+\deg\omega_{\sC}^{\log}|_{\sC_\infty^i}=5\deg\sN|_{\sC_\infty^i}+\bl2g(\sC_\infty^i)-2+|\sC_\infty^i\cap\sC_{1\infty}|+m_i\br.
$$
%On the other, for $v\in A_1$, using that $\nu_2|_{\sC_v}\ne 0$ and vanishes along $\sC_v\cap\sC_\infty$,
%we conclude that 
%$$\deg\sN|_{\sC_v}\ge \ofth |E_v\cap E_{V(\Up_\xi)_\infty}|.
%$$
Therefore, using that $\deg\sN|_{\sC_v}=0$ unless $v\in V(\Up_\xi)_{1\infty}\cup V(\Up_\xi)_\infty$, 
and using that $\sum_{v\in V(\Up_\xi)_{1\infty}} |E_v\cap E_{V(\Up_\xi)_{\infty}}|=\sum_{i=1}^r|\sC_\infty^i\cap\sC_{1\infty}|$,
we obtain
$$d_\infty=\deg\sN=\sum_{i=1}^r  \ofth\bl2-2g(\sC_\infty^i) -|\sC_\infty^i\cap\sC_{1\infty}|-m_i \br+
\sum_{v\in V(\Up_\xi)_{1\infty}}\deg \sN|_{\sC_v}
$$
$$=\frac{2r}{5}- \sum_{i=1}^r\ofth\bl m_i+2g(\sC_\infty^i)\br+
\ofth \sum_{v\in V(\Up_\xi)_{1\infty}} \delta(v).\quad \ \,
$$
Thus the total number of auxiliary legs added to $v\in V(\Up_\xi)_{1\infty}$ is %$\sum_{v\in A_2}10\deg \sN|_{\sC_v}$,
bounded by $10d_\infty+4g+2\ell$, and so is the number $r$ of connected components of $\sC_\infty$.

Let $\ti\Up_\xi$ be the resulting graph after adding auxiliary legs to $v\in V(\Up_\xi)$ according
to the rules specified above. We now study the stability of vertices of $\ti\Up_\xi$. Let $v\in V(\Up_\xi)$ be
a not-stable vertex. Then $v\in V(\Up_\xi)_{1\infty}$ and $\deg\sN|_{\sC_v}=\ofth|E_v\cap E_{V(\Up_\xi)_\infty}|$.
Because $v$ is not stable, $|E_v|=1$ or $2$. If $|E_v|=1$, we have $\deg\sN|_{\sC_v}=\ofth$. By
Lemma \ref{stable-cri}, $\xi$ is not stable, impossible. Thus $|E_v|=2$, and $v$ is a strictly semistable vertex of
$\ti\Up_\xi$.

We now show that $\ti\Up_\xi$ contains no chain of strictly semistable vertices of length more than two.
%(A chain of vertices consists of distinct vertices $v_1,
%\cdots,v_{k\ge 3}$ and distinct edges $e_1,\cdots,e_{k-1}$ so that $v_i$ and $v_{i+1}$ are vertices of $e_i$ for all $i$.) 
Indeed, suppose $v_1,v_2,v_3\in V(\ti\Up_\xi)$ with edges $e_1$ and $e_2$ forming a chain of
unstable vertices in $\ti\Up_\xi$, where $e_1$ connects $v_1$ and $v_2$, and $e_2$ connects $v_2$ and $v_3$. 
By our construction, all $v_i\in V(\Up_\xi)_{1\infty}$. Since $\xi$ is stable, 
by Lemma \ref{stable-cri}, $\deg\sN|_{\sC_{v_2}}\geq \ofth$. Since both $v_1$ and $v_3$ are not in $V(\Up_\xi)_\infty$,
$E_{v_2}\cap E_{V(\Up_\xi)_\infty}=\emptyset$, a contradiction.
%thus $v_2$ has at least two legs  and at least two edges
%in $\ti\Up_\xi$ attached to it, violating that $v_2$ is not stable in $\ti\Up_\xi$.

Let $\Xi$ be the set of connected graphs with legs and whose vertices $v$ are decorated by non-negative integers
$g_v$. For a $\Up\in \Xi$, we define its genus $g(\Up)$ to be $\frac{1}{2} \dim H^1(\Up,\QQ)+\sum_{v\in V(\Up)}g_v$.
For a pair of integers $(g,l)$,  let $\Xi_{g,l}$ be the set of genus $g$ graphs in $\Xi$ having exactly $l$ edges.
We say $\Up\in \Xi$ quasi-stable if all its vertices are semistable and that there is no chain of 
strictly semistable vertices in $\Up$ of length more than two. Let $\Xi_{g,l}^{q.s.}$
be the set of quasi-stable graphs in $\Xi_{g,l}$. Since the subset of stable $\Up$ in $\Xi_{g,l}$ is finite,
we see easily that $\Xi_{g,l}^{q.s.}$ is also finite. 

By our bound on the legs added to $\Up_\xi$ to derive $\ti\Up_\xi$, we conclude that
$$\{\ti\Up_\xi\mid \xi\in \cW\lggd^-(\CC)\}\sub \coprod_{l\le 3\ell+10d_\infty+4g} \Xi_{g,l}^{q.s.},
$$
where the later is finite. Consequently, $\Theta=\{\Up_\xi\mid \xi\in \cW\lggd^-(\CC)\}$ is finite.
\end{proof}

\begin{prop}
The collection $\cW\lggd^-(\CC)$ is bounded.
\end{prop}

\begin{proof} Since $\Theta$ is finite, we only need to show that to any $\Up\in
\Theta$, the set
$\cW_\Up=\{\xi\in \cW^-\lggd(\CC)\mid \Up_\xi\cong \Up\}$ is bounded. 

Let $\xi\in\cW\lggd^-(\CC)$ and let $\Up_\xi$ be its decorated dual graph. The information of $\xi$ indeed
provides us three more assignments 
$$\iota_\xi: V(\Up_\xi)\to \index \and d_{\sL,\xi},\ d_{\sN,\xi}: V(\Up_\xi)\lra \ZZ,
$$
where $\iota_\xi(v)=a$ if $\sC_v\sub\sC_a$, $d_{\sL,\xi}(v)=\deg\sL|_{\sC_v}$, and $d_{\sN,\xi}(v)=\deg\sN|_{\sC_v}$.

We show that given $\Up\in\Theta$, the set
\beq\label{d-d}
\{(\Up_\xi,\iota_\xi,d_{\sL,\xi},d_{\sM,\xi})\mid \xi\in \cW\lggd^-(\CC) \ \text{and}\ \Up_\xi\cong \Up\}
\eeq
is finite. First, since $\index$ is finite, the possible choices of $\iota_\xi:V(\Up)\to\index$ are finite.
Once $\iota_\xi$ is chosen, the map $d_{\sN,\xi}$ has the following properties: 
$d_{\sN,\xi}$ takes value $0$ on $\iota_\xi\upmo\bl\{0,01,1\}\br$,  for 
$v\in \iota_\xi\upmo(\infty)$, $d_{\sN,\xi}(v)=\ofth(2-2g_v-|E_v|)$, and $d_{\sN,\xi}$ restricted to 
$\iota_\xi\upmo\bl{1\infty}\br$ takes values in
$\ofth\ZZ_{>0}$. Since $\sum_v d_{\sN,\xi}(v)=d_\infty$, the choices of $d_{\sN,\xi}$ are finite.

%Given $\Up\in \Theta$, since  there is only finitely many choice of maps $\iota:V(\Up)\to\index$.
%We fix one such. We next look at the possible assignment $d_N:V(\Up)\to\ofth\ZZ$, after fixing a choice of
%$\iota$. We first set $d_N$ to take value $0$ on $\iota\upmo(0)\cup\iota\upmo(01)\cup\iota\upmo(1)$, and for 
%$v\in \iota\upmo(\infty)$, we set $d_N(v)=\ofth(2-2g_v-|E_v|)$. Finally, the values of $d_N$ on $V(\Up)_{1\infty}$ 
%lie in $\ofth\ZZ_{>0}$. Since $\sum_v d_N(v)=d_\infty$, the choices of $d_N$ is finite.

For the assignment $d_{\sL,\xi}: V(\Up)\to\ofth\ZZ$, as $d_{\sL,\xi}$ takes positive values on
$\iota_\xi\upmo(\{0,01\})$, takes $0$ on $\iota_\xi\upmo(1)$, and that $d_{\sL,\xi}+d_{\sN,\xi}$ takes value $0$ on
$\iota_\xi\upmo(\{1\infty,\infty\})$, the possible choices of $d_{\sL,\xi}$ are also finite.
Combined, we prove that the set \eqref{d-d}
is finite.

Finally, fixing a $\xi'\in\cW\lggd^-(\CC)$, %choice $(\Up_\xi,\iota_\xi,d_{\sL,\xi},d_{\sN,\xi})$, 
it is direct to check that the collection
$$\{\xi\in \cW\lggd^-(\CC)\mid (\Up_\xi,\iota_\xi,d_{\sL,\xi},d_{\sM,\xi})\cong (\Up_{\xi'},\iota_{\xi'},d_{\sL,\xi'},d_{\sN,\xi'})\}
$$ 
is bounded. This proves the proposition.
\end{proof}

\end{document}